\documentclass[11pt]{article}
\usepackage{amsthm,amsmath,natbib}
\usepackage{amsfonts}
\usepackage{mathtools}
\usepackage{stfloats}
\usepackage{float}
\usepackage{xcolor}
\RequirePackage[colorlinks,citecolor=blue,urlcolor=blue]{hyperref}
\usepackage{todonotes}
\usepackage{amssymb}
\usepackage{cancel}
\usepackage{comment}
\usepackage[margin=1.75in]{geometry}

\usepackage[shortlabels]{enumitem}
\usepackage{thmtools}

\newcommand{\pref}[1]{Proposition \ref{prop:#1}}

\newcommand{\cref}[1]{Corollary \ref{lm:#1}}

\theoremstyle{definition}
\newtheorem{theorem}{Theorem}
\newtheorem{lemma}{Lemma}
\newtheorem{definition}{Definition}
\newtheorem{proposition}{Proposition}
\newtheorem{corollary}{Corollary}
\newtheorem{claim}{Claim}
\newtheorem{conjecture}{Conjecture}

\newtheorem{assumption}{Assumption}
\newtheorem{fact}{Fact}

\newcommand{\f}[1]{\textsc{#1}}

\newcommand{\ind}{\perp\!\!\!\!\perp} 

\newcommand{\baf}[4]{\prescript{#1}{#2}{\unrhd}^{#3}_{#4}}
\newcommand{\bafd}{\baf{{\cal B}}{H_1}{}{H_2}} 

\newcommand{\sbaf}[4]{\prescript{#1}{#2}{\rhd}^{#3}_{#4}}
\newcommand{\sbafd}{\prescript{{\cal B}}{H_1}{\rhd}^{}_{H_2}}

\newcommand{\nul}{\mathcal{N}}
\newcommand{\btl}{\blacktriangleleft}
\newcommand{\btleq}{~ \underline{\blacktriangleleft} ~}
\newcommand{\nbtleq}{~ \cancel{\underline{\blacktriangleleft}} ~}
\newcommand{\bteq}{~ \stackrel{\blacktriangle}{=} ~}

\DeclareMathOperator{\ran}{ran}

\newcommand{\Ax}[1]{Axiom #1}

\newcommand{\todoinline}[1]{\textcolor{red}{\textbf{TO-DO:\{}} #1\textcolor{red}{\textbf{\}}}}

\title{Robust Bayesianism and Likelihoodism}
\author{Conor Mayo-Wilson and Aditya Saraf}
\begin{document}

\maketitle

\begin{abstract}
    We defend a new theory of statistical evidence, which we call \textit{Robust Bayesianism} (\f{rb}).  We prove that, under widely accepted assumptions,  \f{rb} entails the law of likelihood \citep{royall_statistical_1997}, the likelihood principle \citep{berger_likelihood_1988}, and a variety of other widely-accepted ``statistical principles'', e.g., the sufficiency principle \citep{birnbaum_foundations_1962, birnbaum_more_1972} and stopping-rule principle \citep{berger_likelihood_1988}.  The main technical contribution of this paper is to extend \emph{some} of those results to a qualitative framework in which experimenters are  justified only in making comparative, non-numerical judgments of the form ``$A$ given $B$ is more likely than $C$ given $D$.''
\end{abstract}

We defend a new theory of statistical evidence, which we call \textit{Robust Bayesianism} (\f{rb}). Roughly, \f{rb} is the thesis that (1) evidence is what makes all decision-makers with \emph{permissible beliefs}
more confident in some hypothesis, and (2)  one piece of evidence is stronger than another if the former effects a larger change in confidence than the latter.  We prove that if  the set of permissible prior beliefs equals the set of probability functions, then \f{rb} entails the law of likelihood \citep{royall_statistical_1997}, the likelihood principle \citep{berger_likelihood_1988}, and a variety of other widely-accepted ``statistical principles'', e.g., the sufficiency principle \citep{birnbaum_foundations_1962, birnbaum_more_1972}.  The main technical contribution of this paper is to extend \emph{some} of those results to a qualitative framework in which experimenters are  justified only in making comparative, non-numerical judgments of the form ``$A$ given $B$ is more likely than $C$ given $D$.'' 

Our framework and results are important for three reasons. First,  many instances of good scientific reasoning do not employ probability or statistics.  For instance, when Lavoisier judged that his data provided good evidence against the phlogiston theory of combustion and for the existence of what we now call ``oxygen'', no statistics was invoked. Thus, one might ask which statistical methods -- if any -- are special cases of more general non-probabilistic, methods.   Although there many attempts to generalize Bayesian methods by using non-probabilistic models of belief (e.g., by using imprecise probabilities \citep{walley_statistical_1991} or  Dempster-Shafer belief functions \citep{dempster_upper_1968}), there are comparatively fewer attempts to generalize classical/frequentist or likelihoodist statistical methods and principles to qualitative/comparative frameworks.  In this paper, we \emph{attempt} to extend likelihood-based principles to a purely qualitative setting.

Second, our results highlight limitations of \emph{likelihoodism}, which is, very roughly, the thesis that all evidential meaning of some observation is captured by the likelihood function.  One (slightly) more precise likelihoodist thesis -- the so-called \emph{law of likelihood} -- asserts that the likelihood ratio is a measure of evidential strength. 
We prove that, unfortunately, there is no qualitative analog of the law.  To us, that suggests the plausibility of the likelihoodism may depend on a modeling artifact, namely, the assumption that hypotheses assign \emph{precise} numerical probabilities to experimental outcomes. 

To see the implications for methodology, notice that the law of likelihood is sometimes used to motivate maximum-likelihood estimation and related techniques that make exclusive use of the likelihood function \citep[Chapter 5]{edwards_likelihood_1984}.  Although there may be good reasons for making use of such methods, we think our results show the law of likelihood is not one of them. In the end, we believe our results provide reason to endorse \citet[p. 141]{berger_likelihood_1988}'s conclusion that ``the only satisfactory method of [statistical] analysis based on [the likelihood principle] seems to be robust Bayesian analysis.'' 

Finally, we show that, unlike the law of likelihood, the famous \emph{likelihood principle} (\f{lp}) has a qualitative analog, and similarly, several of \f{lp}'s controversial consequences (including the stopping rule principle) do generalize naturally to the qualitative setting.  Further, those principles are also consequences of our theory of evidence, \f{rb}.  Thus, proponents of classical statistics cannot dismiss those controversial principles by criticizing \f{lp}  or orthodox Bayesianism, as those principles follow from a substantially weaker theses.  

As the name suggests, our theory of evidence is inspired by ``robust Bayesian analysis'' \citep{berger_robust_1984, berger_robust_1990},  which is the practice of ensuring that statistical recommendation are robust across a wide range of prior distributions and utility/loss functions.  The standard justification for robust Bayesian analysis is that, practically speaking, a statistician can rarely precisely/completely elicit a decision-maker's prior probability and loss function.  Thus, to ensure that the statistician's recommendations match what the decision-maker would choose were she maximally informed, the statistician must recommend decisions that are optimal (or nearly so) for all of the priors and loss functions compatible with the results of elicitation.

The motivation for our theory of evidence is different from that of robust Bayesian analysis.  Instead of trying to aid a \emph{single} decision-maker make a \emph{specific} decision, our theory of evidence aims to provide a  qualitative summary of ``what the evidence says'', which we interpret as specifying how \emph{all} decision-makers' degrees of confidence should change in light of various experimental outcomes.  Such a summary might, we hope, aid decision-makers in a \emph{variety} of decision problems, regardless of their prior beliefs or preferences.

In \autoref{sec:TheoriesOfEvidence}, we first describe the purposes of a theory of statistical evidence and then summarize the two theories that we will study in this paper:  \f{rb} (\autoref{subsec:RobustBayesianism}) and likelihoodism (\autoref{subsec:likelihoodism}). We prove that the two theories are equivalent when the set of permissible prior beliefs is the set of all probability distributions over the parameter space. Further, we provide an equivalent formulation of \f{lp}, that we call the \textit{Mixture-Sufficiency Principle} (\f{msp}), which generalizes to qualitative contexts.

In \autoref{sec:qual}, we show that \f{msp} and the \textit{weak} law of likelihood can be formalized in a system of qualitative probability. In \autoref{sec:qualdiff}, we explore some of the major differences between the qualitative and quantitative framework.
Finally, in \autoref{sec:sps}, we consider the controversial \textit{stopping rule principle}, which states that an experiment's stopping rule is irrelevant. We prove that a qualitative analog of the stopping rule principle is entailed by \f{rb}, and thus one cannot object to the stopping rule principle by simply rejecting \f{lp} or quantitative Bayesianism.







\section{Theories of Statistical Evidence:  Robust Bayesianism and Likelihoodism}
\label{sec:TheoriesOfEvidence}

Theories of statistical evidence abound.  For Fisherians, a $p$-value is a measure of evidence against the null hypothesis.    For likelihoodists (e.g., \citep{hacking_logic_1965}, \citep{edwards_likelihood_1984} \citep{royall_statistical_1997}, \citep{sober_evidence_2008}), the likelihood ratio $P_{\theta_1}(E)/P_{\theta_2}(E)$ specifies how much the evidence ``favors'' $\theta_1$ over $\theta_2$.  For Bayesian confirmation theorists, there are many possible measures of evidence (e.g., see \citep{fitelson_plurality_1999} and \citep[p. 107]{evans_measuring_2015}), but the Bayes Factor is popular \citep{goodman_toward_1999-1}.


As we have said, our goal is to investigate which existing theories of statistical evidence are special cases of more general, qualitative/comparative theories of evidence and how those theories interact with similar qualitative theories of rational belief/confidence.   
When we began this project, we believed that likelihoodism (see \autoref{subsec:likelihoodism}) was the most plausible theory of statistical evidence.  But we encountered a series of obstacles formulating qualitative analogs of fundamental likelihoodist theses, and we began to see those theses as obscuring, rather than clarifying, when all decision-makers' beliefs should (or should not) change.  Our struggles led to an alternative theory: robust Bayesianism (\f{rb}).

According to \f{rb}, evidence is, by definition, that which makes \emph{all} decision-makers with \emph{permissible prior beliefs} more confident in some hypothesis.   What counts as ``permissible'' will depend on context.  Background scientific knowledge might make certain prior beliefs impermissible.  Perhaps most commonly,   the set of permissible priors is constrained by knowledge that features of a randomizing device (e.g., the chance that a random number generator outputs 7) are causally independent of the parameter of interest (e.g., the survival rate of a drug).  We return to this type of constraint in later sections.

Institutional rules might also determine what is permissible.  For instance, in a courtroom, one might be obliged to assign high prior probability to a defendant's innocence.     Thus, all of the definitions below contain a parameter $\mathcal{B}$ that represents the set of permissible beliefs in a given context.


\subsection{Notation}
Before formalizing any theories of evidence, we introduce some notation.  For now, we will think of an experiment $\mathbb{E}$ as a pair $\langle \Omega^{\mathbb{E}}, \{P^{\mathbb{E}}_{\theta}\}_{\theta \in \Theta} \rangle$, where $\Omega^{\mathbb{E}}$ represents the possible outcomes of the experiment and $P^{\mathbb{E}}_{\theta}(\cdot)$ is a probability distribution over $\Omega^{\mathbb{E}}$  that specifies how likely each outcome is if $\theta$ is the true value of the parameter.  Elements of $\Theta$ represent \emph{simple/point} hypotheses.

Given an experiment $\mathbb{E}$, we let $\Delta^{\mathbb{E}}=\Theta \times \Omega^{\mathbb{E}}$.  We use $H_1, H_2$ etc. to denote non-empty subsets of $\Theta$, and $E, F$, etc. to denote non-empty subsets of $\Omega$. We drop the superscript $\mathbb{E}$ when it is clear from context.

We will typically assume that both $\Theta$ and $\Omega$ are finite so that it is possible for $P_{\theta}$ to be defined on the power set algebra of $\Omega$ (and similarly, for joint distributions $Q$ on $\Delta$ to be defined on its power set).  However, the results in this section  extend to continuous spaces in straightforward ways.  Thus, some of our examples involve continuous distributions, and in such cases, the reader is expected to fill in the relevant measure-theoretic details.

\subsection{Robust Bayesianism}
\label{subsec:RobustBayesianism}
Recall, \f{rb} is roughly the doctrine that evidence is that which makes every decision-maker with \emph{permissible} prior beliefs comparatively more confident in some hypothesis.     In this section, we assume that permissible beliefs are always representable by a prior probability distribution $Q$ over $\Theta$.  We relax this assumption in upcoming sections. 

Given a prior probability distribution $Q$ over $\Theta$, we define the posterior in the standard way: 
$$Q^{\mathbb{E}}(H|E) = \frac{Q^{\mathbb{E}}(E \cap H)}{Q^{\mathbb{E}}(E)}  := \frac{\sum_{\theta \in H}P^{\mathbb{E}}_{\theta}(E) \cdot Q(\theta)}{\sum_{\theta \in \Theta }P^{\mathbb{E}}_{\theta}(E) \cdot Q(\theta)}$$
for any $H \subseteq \Theta$.  Here, the sums should be replaced by appropriate integrals in the continuous case.

We will write $Q(\cdot|H)$ and $Q(\cdot|E)$ instead of $Q(\cdot|H \times \Omega)$ and $Q(\cdot|\Theta \times E)$, and similarly for events to the left of the conditioning bar.
 
We can now introduce the central definition/concept of robust Bayesianism.

\begin{definition}\label{defn:bSupport}
 Suppose $H_1, H_2 \subseteq \Theta$ are disjoint.   Let $E$ and $F$ be (sets of) outcomes of experiments $\mathbb{E}$ and $\mathbb{F}$ respectively. Say $E$ ${\cal B}$-\emph{Bayesian supports} $H_1$ over $H_2$ \emph{at least as much as} $F$ if $Q^{\mathbb{E}}(H_1|E \cap (H_1 \cup H_2)) \geq Q^{\mathbb{F}}(H_1|F \cap (H_1 \cup H_2))$ for all priors $Q \in \mathcal{B}$ for which $Q(F \cap (H_1 \cup H_2)) > 0$.  In this case, we write $E \bafd F$.  If the inequality is strict for at least one such $Q$, then we say the support is \emph{strict}, and we write  $E \sbafd F$.
\end{definition}

In words, $E \bafd F$ if, given that hypothesis $H_1$ or $H_2$ must be true, observing $E$ raises your posterior in $H_1$ at least as much as observing $F$.
Strictly speaking, one should write $\langle \mathbb{E}, E \rangle \bafd \langle \mathbb{F}, F \rangle$ (instead of $E \bafd F$) to keep track of the experiments from various outcomes are obtained.  Why?  A set of outcomes $E$ might be obtained in several different experiments, as shown in the next example.  However, we typically write $E \bafd F$ when the experiments are clear from context.\\

\noindent \textbf{Example 1:} Suppose $\mathbb{E}$ is the experiment in which a coin of unknown bias $\theta \in [0,1]$ is flipped $10$ times, and let  $\mathbb{F}$ be the experiment in which the coin is flipped until eight successes. Let $E$ be an outcome of $\mathbb{E}$ in which eight heads are observed, and let $F$ denote an outcome of $\mathbb{F}$ in which 12 trials are required.

Suppose $\mathcal{B}$ consists of all Beta$(\alpha, \beta)$ priors such that  $1 \leq \alpha, \beta < (\alpha + \beta) = 4$ and $\alpha \leq 2$.  Let $H_1 = (1/2,1]$ be the hypothesis the coin is biased towards heads, and let $H_2 = [0,1/2]$ be the complement.

After learning $E$, every permissible permissible posterior is of the form Beta($\alpha + 8, \beta + 2)$ where $\alpha$ and $\beta$ are subject to constraints (a) and (b) above.  Similarly, upon learning $F$, a permissible prior is of the form  Beta($\alpha + 8, \beta + 4)$.  Some quick computations show that $E$ ${\cal B}$-Bayesian supports $H_1$ over $H_2$ strictly more than $F$.

Notice that, if $E$ is a sequence in which the last heads occurs on the tenth toss, then $E$ is an outcome of both $\mathbb{E}$ and of $\mathbb{F}$.  It is a familiar fact that the beta prior is conjugate to both the binomial and negative binomial distributions.   Thus, regardless of which experiment is conducted, after learning $E$, every permissible permissible posterior is of the form Beta($\alpha + 8, \beta + 2)$.  It follows that both $\langle \mathbb{E}, E \rangle \bafd \langle \mathbb{F}, E \rangle$ and vice versa.
\begin{flushright}
    $\Box$
\end{flushright}

Importantly, Bayesian support is, in general, not a total relation. In other words, there often exist $E$ and $F$ such that neither $E \bafd F$ nor $F \bafd E$.   Such cases arise when decision-makers may permissibly disagree about whether $E$ should increase one's comparative confidence in $H_1$ over $H_2$ more than $F$. Thus, Bayesian support is not \emph{quantifiable}:  in general, there is no real-valued function $f_{H_1,H_2}$ such that $f_{H_1,H_2}(E) \geq f_{H_1,H_2}(F)$ precisely when $E \bafd F$.    This means that, unlike other theories of evidence, \f{rb} will in general not permit one to quantify strength of statistical evidence.

However, it turns out that the Bayesian support relation is total for \emph{simple/point} hypotheses.  Why?  \autoref{clm:BSupportCharacterized} below entails that $E$ Bayesian supports $\theta_1$ over $\theta_2$ at least as much as $F$ if and only if  $P_{\theta_1}(E)/P_{\theta_2}(E) \geq P_{\theta_1}(F)/P_{\theta_2}(F)$.  Because the standard ordering $\leq$ of the real numbers is total -- and hence, two likelihood ratios can always be compared -- it follows that the Bayesian support relation is likewise total for simple hypotheses. In Section 3.3, we show that the qualitative analog of the support relation is not total, even for simple/point hypotheses.  


The definition of ``support'' allows us to compare the strength of two pieces of evidence.  But we might also want to consider a fixed piece of evidence $E$ and ask which hypotheses $E$ supports.  The following definition of ``favoring'' does that.

\begin{definition}\label{defn:bFavors}
  $E$ $\mathcal{B}$-\emph{Bayesian favors} $H_1$ to $H_2$ if $E \bafd \Omega$.  Say it strictly does if $E \sbafd \Omega$. 
\end{definition}

Bayesian favoring is a special case of Bayesian support.  In essence, the support relation \emph{compares} (a) how much $E$ Bayesian favors $H_1$ to $H_2$ to (b) how much $F$ Bayesian favors $H_1$ to $H_2$.   Like support, favoring is also not total, in general.  In other words, sometimes $E$ neither favors $H_1$ over $H_2$ nor vice versa.\\

\noindent \textbf{Example 1 (continued):} Consider again the experiment $\mathbb{E}$ in Example 1.  Some straightforward calculations show that both $E$ and $F$ ${\cal B}$-favor $H_1$ over $H_2$ if the set of permissible priors $\mathcal{B}$ equals the specific subset of the Beta priors described in Example 1. If $\mathcal{B}$ contains \emph{all} Beta priors, however, neither $E$ nor $F$ favors $H_1$ over $H_2$.  For example, consider the Beta(54,1) prior $Q$, which is heavily peaked near $1$.  Direct computation shows that, although $Q(H_2|E) \approx 1.12 \cdot 10^{-16}$ is very small, it is approximately twice the prior probability $Q(H_2)$.  Thus, learning $E$ raises the probability of $H_2$, not $H_1$, for an experimenter with prior $Q$.
\begin{flushright}
    $\Box$
\end{flushright}

The notions of support and favoring both involve two fixed hypotheses, $H_1$ and $H_2$.  The last definitions we consider allow us to investigate when two observations provide equivalent evidence in favor of or against \emph{all} hypotheses of interest.  In other words, the last two definitions that we introduce will allows us to investigate when all decision-makers regard $E$ and $F$ as being equivalently informative, in the sense that any decision-maker will judge $E$ and $F$ to have the same effect on her beliefs.

\begin{definition}\label{defn:bPosteriorEquivalence}    Let $E$ and $F$ denote outcomes of experiments $\mathbb{E}$ and $\mathbb{F}$ respectively.  Say $E$ and $F$ are \emph{Bayesian posterior equivalent} if for all priors $Q$  (1)   $Q^{\mathbb{E}}(\cdot|E)$ is well-defined if and only if $Q^{\mathbb{F}}(\cdot|F)$ is, and
(2) $Q^{\mathbb{E}}(H|E) = Q^{\mathbb{F}}(H|F)$ for all hypotheses $H$ and for which those conditional probabilities are well-defined.
\end{definition}

We are not the first researchers to isolate posterior equivalence as an important notion (e.g., see Example 3 of \cite{wechsler_birnbaums_2008}), and as we show in future sections section,  posterior equivalence is implicitly discussed in many presentations of the likelihood principle.  It is natural to ask how posterior equivalence is related to the following equivalence relation, which is defined in terms of support.

\begin{definition}\label{defn:bFavoringEquivalence}  Let $E$ and $F$ be outcomes of experiments $\mathbb{E}$ and $\mathbb{F}$ respectively.  $E$ and $F$ are $\mathcal{B}$-\emph{Bayesian support equivalent} if  $F \bafd E$ and  $E \bafd F$ for all disjoint hypotheses $H_1$ and $H_2$.
\end{definition}

It is easy to show that these two definitions are equivalent.

\begin{claim}
For all $\mathcal{B}$, two experimental outcomes are ${\cal B}$-posterior equivalent if and only if they are ${\cal B}$-support equivalent.
\end{claim}

We omit the proof, as it uses only basic probability theory. However, details of all omitted proofs (including the elementary ones) in this paper can be found in the supplementary materials.

Like support, the domains of the binary relations called ``favoring'', ``support equivalence'', and ``posterior equivalence'' all contain \emph{pairs} of the form $\langle \mathbb{E}, E \rangle$, where the first coordinate specifies a experiment and the second specifies an outcome.  When the relevant experiments are clear from context, we omit mentioning them.\\

\noindent \textbf{Example 1 (continued):} The events $E$ and $F$ from Example 1 are \emph{not} ${\cal B}$-posterior equivalent (and hence, not support equivalent) as learning $E$ produces a Beta$(\alpha+8, \beta+2)$ posterior whereas learning $F$ produces a Beta$(\alpha+8, \beta+4)$ posterior.  In contrast, $\langle \mathbb{E}, E \rangle $ and $\langle \mathbb{F}, E \rangle$ are posterior equivalent (and hence, support equivalent).
\begin{flushright}
    $\Box$
\end{flushright}

\subsection{Likelihoodism}
\label{subsec:likelihoodism}
A common criticism of Bayesian measures of statistical evidence is that they are ``subjective'', i.e., they often depend on the specific prior used. \f{rb} partially avoids this criticism because support and favoring are assessed relative to a \emph{set} of permissible priors. But one might still worry that there is no ``objective'' way to pick a set of permissible priors in a given context.

Worries like these often motivate \emph{likelihoodism}, which is roughly the thesis that all evidential meaning of an observation is captured by the likelihood function.   Because experimenters often agree on likelihood functions (but not priors), likelihoodism provides a potentially more objective way of measuring evidence.  At least three  distinct theses are called ``the likelihood principle'' and are used to motivate likelihood-based methods (e.g., maximum likelihood estimation).  For clarity, we distinguish the three:
    \begin{itemize}
       
        \item Likelihood principle (\textsc{lp}):  Let $E$ and $F$ be outcomes of two experiments $\mathbb{E}$ and $\mathbb{F}$ respectively.  If there is some $c > 0$ such that $P^{\mathbb{E}}_{\theta}(E) =  c \cdot P^{\mathbb{F}}_{\theta}(F)$ for all $\theta \in \Theta$, then $E$ and $F$ are evidentially equivalent.\footnote{Our statement of \f{lp} is identical to \citet{birnbaum_more_1972}'s.  It is perhaps slightly weaker than \citet[p. 271]{birnbaum_foundations_1962}'s statement, which contains the additional (but imprecise) clause that, ``the evidential meaning of any outcome $x$ of any experiment E is characterized fully by giving the likelihood function''.  \citet[p. 19]{berger_likelihood_1988}'s statement of \f{lp} contains a similar second clause.
        However, note that \citet{berger_likelihood_1988} then defend stronger versions of \f{lp} that account for nuisance parameters (Section 3.5) and explain how to generalize the principle beyond the discrete case (Section 3.4).}
        
        \item Weak Law of likelihood (\textsc{ll}): $P^{\mathbb{E}}_{\theta_1}(E)>P^{\mathbb{E}}_{\theta_2}(E)$ if and only if the data $E$ favors $\theta_1$ over $\theta_2$.\footnote{
         \citet{sober_evidence_2008} and \citet{forster_why_2010} carefully distinguish between \f{ll} and $\f{ll}^+$ but endorse both theses.  Although both \cite[p. 19]{royall_statistical_1997} and \cite[p. 32]{sober_evidence_2008} claim Hacking endorses $\f{ll}^+$, what \citep[Chapter 5]{hacking_logic_1965} calls the ``law of likelihood'' is closest to \f{ll}, as he never discusses the \emph{degree} to which one hypothesis is favored over another.}
        \item Strong Law of likelihood ($\textsc{ll}^+$): $P^{\mathbb{E}}_{\theta_1}(E)/P^{\mathbb{E}}_{\theta_2}(E) > P^{\mathbb{E}}_{\theta_1}(F)/P^{\mathbb{E}}_{\theta_2}(F)$ if and only if $E$ supports $\theta_1$ over $\theta_2$ more than $F$ does. Throughout this paper, we will use the term \textit{supports} to indicate this quaternary relation between $E$, $F$, $\theta_1$, and $\theta_2$.
    \end{itemize}
For likelihoodists (who endorse the above principles but do not embrace Bayesianism), the notions of ``evidential equivalence'' and ``favoring'' are undefined primitives that are axiomatized by principles like \f{lp} and \f{ll}. Likelihoodists, therefore, defend \f{lp}, \f{ll}, and $\f{ll}^+$ by arguing the three principles accord with (i)  our informed reflections about evidential strength, (ii) accepted statistical methods that have been successfully applied in science, and (iii) other plausible and less-controversial ``axioms'' for statistical inference (e.g., the sufficiency principle).
    
Notice that \f{ll} and $\f{ll}^+$ concern only  \emph{simple/point} hypotheses (i.e., elements of $\Theta$); the extension of those theses to  \emph{composite} hypotheses (i.e., subsets of $\Theta$) is controversial, especially when  nuisance parameters are present (see e.g., \cite{royall_statistical_1997} [\S 1.7 and Chapter 7]
and \cite{bickel_strength_2012} for different strategies).

Some statisticians argue for an even stronger law of likelihood.
\begin{itemize}
    \item Numerical law of likelihood ($\f{nll}$): The likelihood ratio $P^{\mathbb{E}}_{\theta_1}(E)/P^{\mathbb{E}}_{\theta_2}(E)$ quantifies the degree to which the data $E$ favors $\theta_1$ over $\theta_2$.\footnote{Proponents of $\f{nll}$ typically endorse it in conjunction with $\f{ll}$. For instance, \citet[p. 3]{royall_statistical_1997} defines  the ``law of likelihood'' to be the thesis that, ``If hypothesis $A$ implies that the probability that a random variable $X$ takes the value $x$ is $p_A(x)$, while hypothesis $B$ implies that the probability is $p_B(x)$, then the observation $X=x$ is evidence supporting $A$ over $B$ if and only if $p_A(x)>p_B(x)$, and the likelihood ratio, $p_A(x)/p_B(x)$ measures the strength of that evidence.''  \citet[p. 31]{edwards_likelihood_1984}'s ``likelihood axiom'' is the conjunction of what we call $\f{ll}^+$ and the stronger claim that ``Within the framework of a statistical model, all the information which the data provide concerning the relative merits of two
        hypotheses is contained in the likelihood ratio of those hypotheses.'' The second claim is similar to the second clauses of   \citet[p. 271]{birnbaum_foundations_1962} and \citep[p. 9]{berger_likelihood_1988}'s statements of \f{lp}.}  
\end{itemize}
Advocates of \f{nll} may even propose specific numerical cutoffs that indicate strength of evidence.  For example,  \citet[p. 12]{royall_statistical_1997} claims ``a likelihood ratio of 8 is pretty strong evidence.'' 

Our main goal is to investigate which principles extend to a purely comparative statistical framework. While we can formulate analogs of $\f{ll}$ and $\f{lp}$ in a non-numerical framework (despite their use of multiplication and division, respectively), $\f{nll}$ clearly lacks any purely qualitative analog. Thus, we will not consider $\f{nll}$ further.    An example can clarify and distinguish the three remaining theses.  \\

\noindent \textbf{Example 2:} Suppose two experiments are designed to distinguish $\Theta = \{\theta_1, \theta_2\}$.  The possible outcomes of the experiments ($\omega_1, \omega_2, \omega_3, \omega_1'$, etc.) are the column headers in the tables below, and the likelihood functions  are the column vectors.   We intentionally list only two possible outcomes of Experiment $\mathbb{F}$.

\begin{center}
    \begin{tabular}{l r}
         
	\begin{tabular}{l | l | l | l}
		\multicolumn{4}{c}{Experiment $\mathbb{E}$} \\
		\hline
		$\ $ & $\omega_1$ & $\omega_2$ & $\omega_3$ \\
		\hline
		$\theta_1$ & .423 & .564 & .011\\
		$\theta_2$ & .039 & .052 & .909 \\
		\end{tabular}
&
	\begin{tabular}{l | l | l | l}
		\multicolumn{4}{c}{Experiment $\mathbb{F}$} \\
		\hline
		$\ $ & $\omega_1'$ & $\omega_2'$ & $\ldots$ \\
		\hline
		$\theta_1$ & .846 & .12  & $\ldots$ \\
		$\theta_2$ & .078 &  .01 & $\ldots$ \\
		\end{tabular}\\ 
    \end{tabular}
\end{center}

According to \f{lp}, the outcomes $\omega_1$ and $\omega_2$ from $\mathbb{E}$ are evidentially equivalent to one another because $P^{\mathbb{E}}_{\theta_1}(\omega_1) = 3/4 \cdot P^{\mathbb{E}}_{\theta_1}(\omega_2)$ for all $\theta \in \{\theta_1, \theta_2\}$.     Similarly, both $\omega_1$ and $\omega_2$ are evidentially equivalent to $\omega_1'$ in $\mathbb{F}$ as $P^{\mathbb{E}}_{\theta_1}(\omega_1) = 1/2 \cdot P^{\mathbb{F}}_{\theta_1}(\omega_1')$ for all $\theta$. However, \f{lp} says nothing about which of the hypotheses in $\Theta$ are favored by $\omega_1,\omega_2$, and $\omega_1'$.  In general, \f{lp} tells one only when two samples should yield identical estimates, but it says nothing about what those estimates should be.

According to $\f{ll}$, the outcome $\omega_1$ ``favors'' $\theta_1$ over $\theta_2$ because $P^{\mathbb{E}}_{\theta_1}(\omega_1) > P^{\mathbb{E}}_{\theta_2}(\omega_1)$.  Similarly, \f{ll} entails  that $\omega_2$ favors $\theta_1$ over $\theta_2$.  However, $\f{ll}$ tells one nothing about how strong those pieces of evidence are, and in particular, whether $A$ and $B$ provide evidence of equal or different strength.

In contrast, $\f{ll}^+$ entails that both $\omega_1$ and $\omega_2$ are weaker pieces of evidence than than $\omega_2'$ because $P^{\mathbb{E}}_{\theta_1}(A)/ P^{\mathbb{E}}_{\theta_2}(\omega_1) < 12 = P^{\mathbb{F}}_{\theta_1}(\omega_2')/ P^{\mathbb{F}}_{\theta_2}(\omega_2')$.  Notice that both $\f{ll}^+$ and $\f{lp}$ allow one to draw conclusions about evidential value \emph{without knowing the sample spaces of the two experiments}, as comparing the likelihood functions of $E$ and $F$ does not require one to know what outcomes other than $E$ and $F$ might have been observed.  It is for this reason that authors often stress that likelihoodist principles entail the ``irrelevancy of the sample space'' \citep[\S 1.11]{royall_statistical_1997}.
\begin{flushright}
    (End Example)
\end{flushright}


Thus, distinguishing the three theses is important for understanding which statistical tests and techniques are justified by each.  For instance, although \f{ll} entails that the \f{mle} is always favored over rivals,  it does not say by how much. Thus, \f{ll} is of little use when the fit of the  \f{mle} needs to be weighed against considerations of simplicity and prior plausibility.  Only \f{nll} and $\f{ll}^+$, we think, can be used to justify likelihoodist estimation procedures, but sadly, neither principle has a qualitative analog.


Although the three theses should be distinguished carefully, they are, unsurprisingly, related.  Notice that $\f{ll}^+$ \emph{unifies} \f{ll} and \f{lp}, in the sense that it entails both theses (given plausible additional assumptions).    Why? Let $E$ be a piece of evidence and let $\theta_1$ and $\theta_2$ be two different hypotheses. Say that $E$ \textit{favors} $\theta_1$ over $\theta_2$ if $E$ supports $\theta_1$ over $\theta_2$ more than the sure event $\Omega$. To show that $\f{ll}^+$ entails \f{ll}, notice that $E$ supports $\theta_1$ over $\theta_2$ more than $\Omega$ if $P^{\mathbb{E}}_{\theta_1}(E)/P^{\mathbb{E}}_{\theta_2}(E) > P^{\mathbb{E}}_{\theta_1}(\Omega)/P^{\mathbb{E}}_{\theta_2}(\Omega)$. Since the likelihood of $\Omega$ is always $1$, this is equivalent to $P^{\mathbb{E}}_{\theta_1}(E) > P^{\mathbb{E}}_{\theta_2}(E)$. So if we define favoring in terms of support, $\f{ll}^+$ entails that $E$ favors $\theta_1$ over $\theta_2$ precisely if $P_{\theta_1}(E) > P_{\theta_2}(E)$, exactly as $\f{ll}$ asserts.

Similarly, to show that $\f{ll}^+$ entails $\f{lp}$ under plausible assumptions, say that $E$ and $F$ are evidentially equivalent if $E$ supports $\theta_1$ over $\theta_2$ exactly as much as $F$ does, for all $\theta_1$ and $\theta_2$. In other words, $E$ and $F$ favor all hypotheses by equal amounts. If $\f{ll}^+$ holds, then $E$ and $F$ are evidentially equivalent precisely if 
\begin{equation}\label{eqn:lrentailslp}
\frac{P_{\theta_1}(E)}{P_{\theta_2}(E)} = \frac{P_{\theta_1}(F)}{P_{\theta_2}(F)}  
\end{equation}
for all $\theta_1$ and $\theta_2$.  If there is $c>0$ such that $P_{\theta}(E) = c \cdot P_{\theta}(F)$ for all $\theta$, then \autoref{eqn:lrentailslp} holds.  So $\f{ll}^+$ entails \f{lp}.

The careful reader will have noticed that, in deriving \f{ll} from $\f{ll}^+$, we defined ``favors'' in a way analogous to the way that ``Bayesian favoring'' was defined in terms of ``Bayesian support.''  This is not a coincidence.  As we show in the next section, the relations picked out by our definitions of Bayesian support, Bayesian favoring, and posterior/support equivalence are, under one assumption, mathematically equivalent to the likelihoodist relations of support, favoring, and evidential equivalence as axiomatized by $\f{ll}^+$, \f{ll}, and \f{lp} respectively.

\subsection{The Link:  Likelihoodism and Robust Bayesianism}
\label{subsec:LikeRB}
In this section, we discuss three elementary propositions that show that \f{rb} and likelihoodism are equivalent when the set of permissible beliefs $\mathcal{B}$ is the \textit{universal} set of all possible priors, $\mathcal{U}$.    The main results of our paper -- in the next section -- are the qualitative analogs of the second and third claim below.

To motivate our results, notice that \f{ll}, \f{lp}, and $\f{ll}^+$ are often said to be ``compatible'' with Bayes rule \cite[p. 28]{edwards_likelihood_1984}.  Here is how that ``compatibility'' is often explained for $\f{ll}^+$.   Notice Bayes' Rule entails that
$$ \frac{Q^{\mathbb{E}}(\theta_1|E)}{Q^{\mathbb{E}}(\theta_2|E)} = \frac{P^{\mathbb{E}}_{\theta_1}(E)}{P^{\mathbb{E}}_{\theta_2}(E)} \cdot \frac{Q(\theta_1)}{Q(\theta_2)}$$
Although this calculation is standard in introductory remarks about Bayes rules (e.g., see \cite[p. 8]{gelman_bayesian_2013}), we note that it holds  \emph{for any prior} $Q$.  So the likelihood ratio is a  measure of  the degree to which \emph{all} Bayesians' posterior degrees of belief in $\theta_1$ increase (or decrease) upon learning $E$. That simple calculation can be used to prove: 

\begin{claim}\label{clm:BSupportCharacterized}
Suppose $H_1$ and $H_2$ are finite and disjoint and that $\mathcal{U}$ is the set of all probability distributions over $\Theta$.  Let $E$ and $F$ be outcomes of experiments $\mathbb{E}$ and $\mathbb{F}$ respectively. Then  $E \baf{{\cal U}}{H_1}{}{H_2}{} F$ if and only if (1) for all $\theta \in H_1 \cup H_2$, if $P^{\mathbb{F}}_{\theta}(F) > 0$, then $P^{\mathbb{E}}_{\theta}(E) > 0$, and (2) for all $\theta_1 \in H_1$ and $\theta_2 \in H_2$:
\begin{equation*}
    P^{\mathbb{E}}_{\theta_1}(E) \cdot  P^{\mathbb{F}}_{\theta_2}(F) \geq P^{\mathbb{E}}_{\theta_2}(E) \cdot P^{\mathbb{F}}_{\theta_1}(F)
\end{equation*}
Similarly, $E \sbaf{{\cal U}}{H_1}{}{H_2}{} F$ if and only if condition 1 holds and the inequality in condition 2 is strict for at least choice of $\theta_1 \in H_1$ and $\theta_2 \in H_2$.
\end{claim}

Though mathematically simple, the claim is conceptually important.  The right-hand side of the biconditional in \autoref{clm:BSupportCharacterized} contains only likelihoods.  So the claim says that, when the set of permissible priors is as large as possible, claims about Bayesian support are equivalent to claims about likelihoods.  Specifically, the claim asserts that $E \baf{{\cal U}}{H_1}{}{H_2}{} F$ exactly when $\f{ll}^+$ says that $E$ provides at least as strong evidence as $F$ does for every $\theta \in H_1$ over every $\theta_2 \in H_2$.

Because we define ``favoring'' in terms of ``support'', it is perhaps unsurprising that likelihood functions completely characterize the favoring relation.

\begin{claim}\label{clm:bLLFavoring} Suppose $\mathcal{U}$ is the set of all probability distributions over $\Theta$.  Then $E$ $\mathcal{U}$-Bayesian favors $H_1$ to $H_2$ if and only if (1) $P_{\theta}(E) > 0$ for all $\theta \in H_1 \cup H_2$ and (2) $P_{\theta_1}(E) \geq P_{\theta_2}(E)$ for all $\theta_1 \in H_1$ and $\theta_2 \in H_2$.  The favoring is strict if and only if condition 1 holds and the inequality in condition 2 is strict for at least choice of $\theta_1 \in H_1$ and $\theta_2 \in H_2$.
\end{claim}

The claim entails that, if $H_1=\{\theta_1\}$ and $H_2=\{\theta_2\}$ are simple hypotheses, then $E$  strictly Bayesian favors $H_1$ to $H_2$ if and only if \f{ll} entails that $E$ favors $H_1$ to $H_2$.

Finally, it is well-known that, if \f{lp} entails that $E$ and $F$ are evidentially equivalent, then they are posterior equivalent \citep[p. 56]{edwards_bayesian_1984}.  The converse is also true.

\begin{claim}\label{clm:bLPFavoringPosteriorEquivalence}
	\f{lp} entails that $E$ and $F$ are evidentially equivalent if and only if $E$ and $F$ are $\mathcal{U}$-posterior equivalent and $\mathcal{U}$-support equivalent.  
\end{claim}

The three claims above show that there is an alternative way of explaining why (1) $\f{ll}^+$ and $\f{lp}$ are intuitively plausible in many cases and (2) $\f{ll}^+$ ``unifies'' \f{ll} and \f{lp} as shown in the previous section.  To see why, consider one last philosophical thesis about statistical evidence.\\

\noindent Robust Bayesian Support Principle:  In settings in which $\mathcal{B}$ is the set of permissible prior beliefs,  $E$ provides at least as good statistical evidence for $H_1$ over $H_2$ as $F$ if $E \bafd F$.\\

By \autoref{clm:BSupportCharacterized}, the robust Bayesian support principle is equivalent to $\f{ll}^+$ when there are no restrictions on permissible priors.  And by \autoref{clm:bLLFavoring} and \autoref{clm:bLPFavoringPosteriorEquivalence}, it entails \f{ll} and \f{lp} with some plausible, additional assumptions, thereby ``unifying'' \f{ll} and \f{lp} in roughly the same way $\f{ll}^+$ does.  Thus, $\f{ll}^+$ may be plausible only because it is equivalent to the robust Bayesian support principle in settings in which there are no obvious constraints on permissible priors.  And the robust Bayesian support principle is plausible, we conjecture, precisely because it captures the idea that evidence persuades \emph{all} agents with permissible beliefs to change their beliefs in particular ways. This is important because, as we show later, the robust Bayesian principle generalizes to qualitative settings in which $\f{ll}^+$ cannot be formulated.

\subsection{LP:  Equivalent Statements}
\label{sec:lpequiv}

We will show that \f{lp} has a natural generalization in our qualitative framework, but to motivate our generalization, we first explore some statistical principles that are known to be equivalent to \f{lp} in the quantitative setting.

Statisticians with divergent philosophies often nonetheless agree that the evidence in a sample is entirely summarized by a sufficient statistic.  Some call this thesis the \emph{sufficiency} principle.\footnote{What we call ``sufficiency'' is often called ``weak sufficiency'', and ``sufficiency'' is often used to describe a strictly stronger principle \citep{birnbaum_more_1972}.  We ignore those distinctions here.}
Recall, a statistic $T:\Omega^{\mathbb{E}} \rightarrow \mathbb{R}$ is \textbf{sufficient} for an experiment $\mathbb{E}$ if $P^{\mathbb{E}}_{\theta}(\omega|T=T(\omega))$ does not depend on $\theta$, or in other words if:
\[ P^{\mathbb{E}}_{\theta}(\omega|T=T(\omega)) = P^{\mathbb{E}}_{\upsilon}(\omega|T=T(\omega)) \mbox{ for all } \theta, \upsilon \in \Theta \]

The following fact is well-known (e.g.,  \citep[Theorem 2] {birnbaum_more_1972}):

\begin{fact}\label{fact:SufficientFactorization}
  The following two assertions are equivalent:
\begin{enumerate}
    \item There is a sufficient statistic $T$ for $\mathbb{E}$ such that $T(\omega_1) = T(\omega_2)$.
    \item There is some $c > 0$ such that  $P^{\mathbb{E}}_{\theta}(\omega_1) = c \cdot P^{\mathbb{E}}_{\theta}(\omega_2)$ for all $\theta \in \Theta$.
\end{enumerate} 
\end{fact}

In other words, for outcomes of the same experiment, \f{lp} and the sufficiency principle are equivalent.

To explore the relevance of sufficient statistics to outcomes of distinct experiments, we must say a few words about \emph{mixed experiments}.  We say an experiment $\mathbb{M}$ is a mixture of $\mathbb{E}$ and $\mathbb{F}$ if (1) a randomizing device is used to choose which of $\mathbb{E}$ and $\mathbb{F}$ to perform and (2) both component experiments have some positive probability of being chosen.  Here is a slightly odd technical principle that will play an important role in ensuing sections.  \\

\noindent Mixture-Sufficiency Principle ($\f{msp}$): Let $\omega_1$ be an outcome of an experiment $\mathbb{E}$ and $\omega_2$ be an outcome of experiment $\mathbb{F}$.  Suppose that, in every mixture $\mathbb{M}$ of $\mathbb{E}$ and $\mathbb{F}$, there is some sufficient statistic $T$ for $\mathbb{M}$ such that $T(\omega_1)=T(\omega_2)$.  Then $\omega_1$ and $\omega_2$ are evidentially equivalent.\\

It is easy to see that that \f{msp} and \f{lp} are equivalent.  Why?  Suppose that \f{lp} entails that $\omega_1$ and $\omega_2$ are evidentially equivalent, i.e., that there is some $c > 0$ such that  $P^{\mathbb{E}}_{\theta}(\omega_1) = c \cdot P^{\mathbb{F}}_{\theta}(\omega_2)$ for all $\theta \in \Theta$.  Let $q \in (0,1)$ be the probability that $\mathbb{E}$ is performed in some mixture $\mathbb{M}$. It follows that
\[P^{\mathbb{M}}_{\theta}(\omega_1) = q \cdot P^{\mathbb{E}}_{\theta}(\omega_1) = q\cdot (c \cdot  P^{\mathbb{F}}_{\theta}(\omega_2)) = \frac{qc}{(1-q)} P^{\mathbb{M}}_{\theta}(\omega_2)\]
Thus, there is some sufficient statistic $T$ \emph{for} $\mathbb{M}$ such that $T(\omega_1)=T(\omega_2)$ (by \autoref{fact:SufficientFactorization}). Hence, \f{msp} likewise entails that $\omega_1$ and $\omega_2$ are evidentially equivalent.  The converse -- that \f{msp} entails \f{lp} -- is verified similarly.

By \autoref{clm:bLPFavoringPosteriorEquivalence}, if $\mathcal{B}$ is the set of all priors, then two experimental outcomes are $\mathcal{B}$ posterior equivalent if and only if \f{lp} entails that they are evidentially equivalent.  It follows from the last argument that the same holds if \f{lp} is replaced by \f{msp}.    This is important because, as we now show, \f{msp} has a natural qualitative analog and can be used to characterize a qualitative notion of posterior equivalence.

\section{Qualitative Robust Bayesianism}
\label{sec:qual}

In this section, we generalize Claims 1-4 to settings in which experiments can justify only qualitative, non-numerical judgments of the form, ``experimental outcome $A$ is at least as likely under supposition $\theta$ as outcome $B$ is under supposition $\eta$.''  Before doing so, we explain why a qualitative framework is important.

\subsection{Motivation for Qualitative Probabilities}
Our qualitative/comparative probabilistic framework is intended to address inference problems with at least one of the following two characteristics; we believe such problems are ubiquitous in science, the law, and corporate settings.

First, when a novel experimental procedure or test is performed, the error rates of the test will generally not be known. Thus, experimenters may not agree upon precise, numerical likelihood functions even if they agree upon how to \emph{order} the likelihood of various experimental outcomes conditional on the point hypotheses.  

For example, imagine a software company has designed a new interview process in which five applicants are assigned several programming tasks.  Let $I= \{1,2,3,4,5\}$ represent the five applicants, and suppose hypotheses are functions $\theta:I \rightarrow \{Y,N\}$ specifying which applicants are qualified ($Y$ for `yes', and $N$ for `no').  Then interviewers might disagree about how \emph{precisely} likely qualified applicants are to complete the various tasks, but they might agree that, for each task, a qualified candidate is more likely to complete the task than an unqualified candidate, i.e., ``$i$ is successful at task $x$ given $\theta(i)=Y$'' is more likely than ``$i$ is successful at task $x$ given $\theta(i)=N$.'' 

Second, even when an experimental procedure is easily repeated, scientists may be able to sample from \emph{only one} ``distribution'' -- namely the one generated by the true hypothesis  $\theta_0$.  Why? Hypotheses in $\Theta$ may describe some feature of the world that is, for all intents and purposes,  unchangeable. Thus, it may be impossible to empirically estimate the likelihood of various experimental outcomes conditional on anything other than the true (yet unknown) parameter. 

For example, in 1610, Galileo observed that four objects move in retrograde motion with Jupiter.   Call those observations $E$.  Galileo believed $E$ was evidence for the hypothesis $\theta_1$ ``The four objects orbit Jupiter''; that hypothesis was in tension with Aristotelian cosmology, which prohibited any object other than Earth from being a source of gravity.  Sadly, Galileo could not travel to alternate universes to estimate the frequency with which the four objects would appear in the same positions that he had observed them.  In other words, he lacked any way to \emph{empirically} estimate the probability of $E$ (or any other telescopic observations) conditional on $\theta_2$ ``The four objects do \emph{not} orbit Jupiter.''  And since the truth of $\theta_1$ was in dispute, Galileo could not simply assume he was sampling from the distribution $P_{\theta_1}$.

We claim that Galileo and his contemporaries, therefore, lacked any basis for agreeing upon numerical probabilities of the form $P_{\theta_1}(E)$ and $P_{\theta_2}(E)$.  As we noted, those probabilities could not be estimated from empirical frequencies.  Galileo's observations did not \emph{logically} refute $\theta_2$:  the retrograde motion of four independent planets could coincide with that of Jupiter.  Nor did $\theta_1$ guarantee that Galileo would observe the four objects precisely where he did in relation to Jupiter. So neither $P_{\theta_1}(E)$ and $P_{\theta_2}(E)$ was identically zero or one. Assigning precise intermediate probabilities also likewise seemed unjustified.  Yet scientists (reasonably) agree that $E$ is more likely if $\theta_1$ is true than if $\theta_2$ were true.

Inference problems with the two above features are ubiquitous.  Further, we often need to weigh probabilistic evidence against qualitative/comparative evidence of the above two types.  
In such cases, one piece of evidence may have precise numerical probabilities conditional on the various hypotheses, but the conjunction of that piece of evidence with other relevant information may lack a precise numerical likelihood. 

For example, consider a United States judge who must determine whether to admit expert testimony that is based on some novel scientific methodology.  The so-called Daubert standard explicitly requires the judge to weigh probabilistic evidence like the ``known or potential error rate'' of a test  (or lack thereof) against qualitative evidence like ``widespread acceptance within a relevant scientific community'' (or lack thereof).  It is, we think, often impossible to justify assigning a precise numerical probability to the ``widespread acceptance'' of some hypothesis or experimental procedure, conditional on the hypothesis being false or the procedure being less reliable than imagined.  It is for this reason and others, we conjecture, that the US Supreme Court advises judges that, when determining whether expert testimony is admissible,  ``The inquiry is a flexible one'', suggesting that a mechanical application of the probability calculus may not be appropriate \citep{daubert_1993}.

Here are two objections that, for reasons of space, we cannot respond to in detail.  First, one might argue that although the above examples motivate developing a comparative, non-numerical framework for evaluating evidence, they do not motivate our specific axioms below.  We agree; in future publications, we hope to show why standard decision-theoretic considerations of coherence motivate our particular axioms.  However, for now, we emphasize that the theorems in ensuing sections provide indirect reason to endorse the axioms, as our results show the axioms entail principles that accord with widely accepted (if controversial) statistical principles like the sufficiency principle.   

Second, some reader may object that, although the above examples require robust statistical analyses -- in which a range of prior distributions \emph{and} likelihood functions are considered -- the examples do not motivate the use of non-numerical probabilities.  Such readers, we conjecture, believe that the only legitimate representation of uncertainty is probability.  We reject that claim.  Typical defenses of that claim fall into two categories:  (1) Dutch Book \citep{de_finetti_foresight:_1937}, forecasting \citep{finetti_theory_1974}, or accuracy arguments \citep{joyce_nonpragmatic_1998, pettigrew_accuracy_2016}, in which degrees of belief are assumed to be representable by real numbers (e.g., the price an agent would be willing to pay for a gamble), and (2) representation theorems in which a numerical probability function is derived from a preference or comparative belief relation (e.g., \citep{savage_foundation_1954} and \citep{krantz_foundations_2006}). The first type of argument simply assumes that rational uncertainty or belief admits a real-valued representation (which we deny); the only question it answers is whether, degrees of belief should obey the probability axioms \emph{if} they are real-valued.  The second type of argument requires the preference relations to be defined over a large/rich algebra of events, and eliciting such preferences (even if they exist) is typically impossible in practice.  

Our framework makes no assumptions about the size or richness of the algebra of events over which an experimenter's beliefs are defined. 

\subsection{Key Concepts}
\begin{table*}[bp]
    \small
    \centering
    \begin{tabular}{p{7cm}|p{7cm}}
        \textbf{Quantitative Probabilistic Notions} & \textbf{Qualitative Analog}  \\
        \hline
        $P^{\mathbb{E}}_{\theta}(E) \leq P^{\mathbb{F}}_{\upsilon}(F)$
         &  $ E |^{\mathbb{E}} \theta \sqsubseteq  F
         |^{\mathbb{F}} \upsilon$ \\
         
         $Q^{\mathbb{E}}(H_1|E) \leq Q^{\mathbb{F}}(H_2|F)$
         &  $H_1 |^{\mathbb{E}} E \preceq H_2 |^{\mathbb{F}} F$ \\
         
         $\f{ll} \Leftrightarrow \omega$ Bayesian favors $H_1$ to $H_2$ & $\f{qll} \Leftrightarrow \omega$ qualitatively favors $H_1$ to $H_2$. \\
         $\f{lp}/\f{msp} \Leftrightarrow \omega_1$ and $\omega_2$ are Bayesian posterior and support equivalent & $\f{qmsp} \Leftrightarrow \omega_1$ and $\omega_2$ are qualitatively posterior and support equivalent. 
    \end{tabular}
\end{table*}

To move from quantitative to qualitative probability, we replace probability functions with two \emph{orderings}. As before, let $\Theta$ be the set of simple hypotheses, and for any experiment $\mathbb{E}$, we let $\Omega^{\mathbb{E}}$ the set of experimental outcomes. The first relation, $\sqsubseteq$, is the qualitative analog of the set of likelihood functions.  
Informally, $A |^{\mathbb{E}} \theta \sqsubseteq B |^{\mathbb{F}} \eta$ represents the claim that ``The outcome $B$ of experiment $\mathbb{F}$ is at least as likely under supposition $\eta$ as outcome $A$ is in experiment $\mathbb{E}$ under supposition $\theta$''; that claim is the qualitative analog of $P^{\mathbb{E}}_{\theta}(A) \leq P^{\mathbb{F}}_{\eta}(B)$.  We write  $A |^{\mathbb{E}} \theta \equiv B |^{\mathbb{F}} \upsilon$  if $A |^{\mathbb{E}} \theta \sqsubseteq B |^{\mathbb{F}} \upsilon$ and vice versa.      As before, we drop the superscripts $\mathbb{E}$ and $\mathbb{F}$ when they're clear from context.

Although we typically consider expressions of the form $A |^{\mathbb{E}}  \theta \sqsubseteq B |^{\mathbb{F}}  \eta$, the $\sqsubseteq$ relation is also defined when experimental outcomes appear to the right of the conditioning bar, i.e., $A|^{\mathbb{E}} \theta \cap E \sqsubseteq B|^{\mathbb{F}}\theta \cap F$ is a well-defined expression if $E \subseteq \Omega^{\mathbb{E}}$ and $F \subseteq \Omega^{\mathbb{F}}$.  However, it is not-well defined when composite hypotheses $H$ appear to the right of the conditioning bar, just as there are no likelihood functions $P_H(\cdot)$ for composite hypotheses in the quantitative case.

Bayesians assume that beliefs are representable by a probability function.  We weaken that assumption and assume beliefs are representable by an ordering $\preceq$. 
The notation is suggestive; $A|B \preceq C|D$ if the experimenter regards $C$ as at least as probable under supposition $D$ as $A$ would be under supposition $B$.  Just as Bayesians condition on both composite hypotheses and experimental outcomes, there are no restrictions on what can appear to the right of the conditioning bar in expressions involving $\preceq$.\footnote{In the body of this paper, we assume $\sqsubseteq$ and $\preceq$ are proper-class relations like $=$, $\subseteq$ and $\in$.  Such relations do not have fixed  sets as their domains, i.e., they can be used to relate \emph{any} two sets. They are unlike, for instance, the standard ordering $\leq$ of real numbers, which is a subset of $\mathbb{R} \times \mathbb{R}$.  In other words, just as $A \subseteq B$ or $A \nsubseteq B$ for any two sets $A$ and $B$, any set $\Omega$ could be used to represent the outcome of some experiment $\mathbb{E}$, and so for simplicity, we assume $\sqsubseteq$ and $\preceq$ are be defined over pairs of all sets.    Such an assumption is not necessary (i.e., we could fix the domains of these relations to a collection of experiments of interest), but it simplifies notation.  
}

As before,  we define $\Delta^{\mathbb{E}}$ to be $\Delta^{\mathbb{E}} = \Theta \times \Omega^{\mathbb{E}}$.  If  $E \subseteq \Omega^{\mathbb{E}}$ and $H \subseteq \Theta$ we write $E|H$ instead of $\Theta \times E| H \times \Omega^{\mathbb{E}}$.  In the special case in which $H=\{\theta\}$ is a singleton, we omit the curly brackets and write $E|\theta$ instead of $E|\{\theta\}$. We write $A | B \sim C|D$ if $A | B \preceq C|D$ and vice versa.  

As in the quantitative setting, we assume the experimenter's ``prior'' does not vary with the experiment performed, and hence, for all subsets $H_1, H_2 \subseteq \Theta$, we assume $H_1 |^{\mathbb{E}} \Delta^{\mathbb{E}} \preceq H_2 |^{\mathbb{E}} \Delta^{\mathbb{E}}$ if and only $H_1  |^{\mathbb{F}} \Delta^{\mathbb{F}} \preceq H_2 |^{\mathbb{F}} \Delta^{\mathbb{F}}$.

Clearly, to prove anything, we must assume that $\sqsubseteq $ and $\preceq$ satisfy certain axioms; those axioms are stated in \autoref{subsec:axioms}.    But before listing out axioms, we state our main results.

\subsection{Main Results}
First, we state a qualitative analog of \f{ll}:

\begin{itemize}
    \item Qualitative law of likelihood (\f{qll}):  $E$ favors $\theta_1$ to $\theta_2$ if $E|\theta_2 \sqsubset E|\theta_1$.  
\end{itemize}

 \autoref{clm:bLLFavoring} asserts that \f{ll} characterizes precisely when $E$ Bayesian favors $\theta_1$ over $\theta_2$, i.e., when all Bayesian agents agree $E$ favors one hypothesis.  So by analogy to the definitions of ``Bayesian support'' and ``Bayesian favoring'', we let $\mathcal{B}$ denote a set of \emph{orderings} satisfying our axioms below and define:
 
 \begin{definition}\label{defn:qSupport}
    $E$ ${\cal B}$-\emph{qualitatively supports} $H_1$ to $H_2$ at least as much as $F$ if $H_1|E \cap (H_1 \cup H_2) \succeq H_1|F \cap (H_1 \cup H_2)$ for all orderings $\preceq$ in ${\cal B}$ for which the expression $\cdot|F \cap (H_1 \cup H_2)$ is well-defined.
\end{definition}

\begin{definition}\label{defn:qFavor}
    $E$ $\mathcal{B}$-\emph{qualitatively favors} $H_1$ to $H_2$ if $E$ $\mathcal{B}$-qualitatively supports $H_1$ to $H_2$ at least as much as $\Omega$.
\end{definition}

Our first major result is the qualitative analog of \autoref{clm:bLLFavoring}.  Just as \autoref{clm:bLLFavoring} entails that $E$ Bayesian favors $\theta_1$ to $\theta_2$ when \f{ll} entails so, our first major result shows that $E$ qualitatively favors $\theta_1$ to $\theta_2$ if and only if \f{qll} entails so. Under mild assumptions, this equivalence can be extended to finite composite hypotheses.  To state the theorem, let $\mathcal{U}$ (again, for ``universal'') now denote that  set of \emph{all} orderings satisfying our axioms below.

\begin{theorem}\label{thm:QLLFavoring}
         Suppose $H_1$ and $H_2$ are finite. Then $E$ $\mathcal{U}$-qualitatively favors $H_1$ over $H_2$ if (1) $\emptyset|\theta \sqsubset E|\theta$ for all $\theta \in H_1 \cup H_2$ and (2) $E|\theta_2 \sqsubseteq E|\theta_1$ for all $\theta_1 \in H_1$ and $\theta_2 \in H_2$. Under \autoref{assum} (below), the converse holds as well. In both directions, the favoring inequality is strict exactly when the likelihood inequality is strict.  If $H_1 = \{\theta_1\}$ and $H_2 = \{\theta_2\}$ are simple, then $E$ qualitatively favors $H_1$ over $H_2$ if and only if \f{qll} entails so. No additional assumptions are required in this case.
\end{theorem}

\noindent  \autoref{assum} below more-or-less says that one's ``prior'' ordering over the hypotheses is unconstrained by the ``likelihood'' relation $\sqsubseteq$.  This is exactly analogous to quantitative probability theory.  In the quantitative case, a joint distribution $Q^{\mathbb{E}}$ on $\Delta = \Theta \times \Omega^{\mathbb{E}}$ is determined by (i) the measures $\langle P_{\theta}(\cdot) \rangle_{\theta \in \Theta}$ over experimental outcomes $\Omega^{\mathbb{E}}$ and (ii) one's prior $Q$ over the hypotheses $\Theta$.  Although the joint distribution $Q^{\mathbb{E}}$ is constrained by $\langle P_{\theta}(\cdot) \rangle_{\theta \in \Theta}$, one's prior $Q$ on $\Theta$ is not, and so for any non-empty hypothesis $H \subseteq \Theta$, one can define some $Q$ such that $Q(H)=1$ and $Q(\theta) > 0$ for all $\theta \in H$.  That's what the following assumption says in the qualitative case.

\begin{assumption}\label{assum}
For all orderings $\sqsubseteq$ satisfying the axioms in \autoref{subsec:axioms} and for all non-empty $H \subseteq \Theta$, there exists an ordering $\preceq$ satisfying the axioms such that (A)  $H|\Delta \sim \Delta|\Delta$ and (B)  $\theta|\Delta \succ \emptyset|\Delta$ for all $\theta \in H$.
\end{assumption}
We conjecture that \autoref{assum} follows from the axioms, but we do not have a proof.

The qualitative analogs of posterior and support equivalence are obvious, but we state them for clarity.

\begin{definition}\label{defn:qFavoringEquivalence}
    $E$ and $F$ are ${\cal B}$-\emph{qualitative support equivalent} if $E$  ${\cal B}$-supports $H_1$ over $H_2$ at least as much as $F$ and vice versa, for any two disjoint hypotheses $H_1$ and $H_2$.
\end{definition}
    
\begin{definition}\label{defn:qPosteriorEquivalence}
    $E$ and $F$ are  ${\cal B}$-\emph{qualitative posterior equivalent} if  for all orderings $\preceq$ in ${\cal B}$ that satisfy the axioms:
    \begin{enumerate}
        \item  the expression $\cdot|E$ is well-defined if and only if $\cdot|F$ is well-defined, and
        \item  $H|E \sim H|F$ for all hypotheses $H \subseteq \Theta$.
    \end{enumerate}
\end{definition}
\noindent As before, these two definitions are equivalent.
\begin{claim}\label{clm:qfavoring_and_posterior_equivalence}
    $E$ and $F$ are qualitative posterior equivalent if and only if they are qualitative support equivalent.
\end{claim}

Our second main result is a qualitative analog of \autoref{clm:bLPFavoringPosteriorEquivalence}, which says that \f{lp}/\f{msp} characterizes when two pieces of evidence are posterior and Bayesian-support equivalent.

To formalize \f{msp} in the qualitative setting, we need to find the qualitative analogs of (1) sufficient statistics and (2) mixed experiments.  Recall, a statistic $T$ is sufficient if $P_{\theta}(\omega|T=T(\omega)) = P_{\upsilon}(\omega|T=T(\omega)) $ for all $\theta$ and $\upsilon$.  That definition generalizes in the obvious way.

\begin{definition}\label{defn:qSufficient}
    A statistic $T:\Omega^{\mathbb{E}} \rightarrow \mathcal{R}$ is \textit{sufficient} if
\[ \omega|^{\mathbb{E}} \{T=T(\omega)\} \cap \theta \equiv  \omega|^{\mathbb{E}} \{T=T(\omega)\} \cap \upsilon \mbox{ for all } \theta, \upsilon \in \Theta\]
\end{definition}

How should we define a ``mixed experiment''?  At the heart of the concept of a mixed experiment is the use of a randomizing device, i.e., a device whose value is independent of the parameter of interest.  A randomizing device, therefore, is an \textit{ancillary statistic}, i.e., a statistic $T$ such that $P_{\theta}(T=t) = P_{\upsilon}(T=t)$ for all values $t$ and all $\theta, \upsilon \in \Theta$.  Again, that notion easily generalizes qualitatively. 

\begin{definition}\label{defn:qAncillary}
A statistic $T:\Omega^{\mathbb{E}} \rightarrow \mathcal{R}$ is called \textit{ancillary} if $\{T=t\}|\theta \equiv \{T=t\}|\upsilon$ for all values $t$ of $T$ and all $\theta, \upsilon \in \Theta$.
\end{definition}

We can now define the qualitative analog of a mixed experiment.
\begin{definition}\label{defn:qMixture}
Let $\{\mathbb{E}_i\}_{i\leq n}$ be a finite set of experiments.  An experiment $\mathbb{M}$ is a \textit{mixture} of $\{\mathbb{E}_i\}_{i\leq n}$ if
\begin{enumerate}
    \item The outcomes of $\Omega^{\mathbb{M}}$ can be represented by the disjoint union $\sqcup_{i \leq n} \Omega^{\mathbb{E}_i}$ of the outcomes of the component experiments.
     \item Let $T_{\pi}$ be the projection statistic $T_{\pi}:\Omega^{\mathbb{M}} \rightarrow I$ that maps each outcome $\omega \in \Omega^{\mathbb{M}}$ to the unique $i \leq n$ such that $\omega \in \Omega^{\mathbb{E}_i}$. Then $T_{\pi}$ is ancillary in $\mathbb{M}$.  
    \item $\omega ~|^{\mathbb{M}}~ \Omega^{\mathbb{E}_i} \cap \theta  \equiv 
    \omega|^{\mathbb{E}_i} \theta$ for all $i \leq n$ and all $\omega \in \Omega^{\mathbb{E}_i}$, and
    \item $\Omega^{\mathbb{E}_i}~|^{\mathbb{M}} \theta 
     \sqsupset \emptyset$ for all $i$ and all $\theta$, and
   
\end{enumerate}
\end{definition}

Condition 1 says that every outcome of $\mathbb{M}$ is an outcome of some component experiment $\mathbb{E}_i$.  We define the outcome space of $\mathbb{M}$ as a disjoint union to ensure that, if some outcome $\omega$ can be obtained in multiple component experiments, then there are multiple corresponding outcomes of $\mathbb{M}$ that encode which component experiment produced $\mathbb{M}$.  For instance, suppose $\mathbb{M}$ is a mixture of (i) an experiment $\mathbb{E}_1$ in which a coin is flipped twice and (ii) an experiment $\mathbb{E}_2$ in which a coin is flipped until one head is observed.  Then the outcome $\omega = \langle$Tails, Heads$\rangle$ can be obtained in both component experiments, and thus, there are two corresponding outcomes in the mixture $\mathbb{M}$ that specify which experiment was conducted to produce $\omega$.  We might naturally represent those two outcomes by $\langle 1, \omega \rangle$ and $\langle 2, \omega \rangle$, but other representations are possible.

Condition 2 says that the component experiment that will be performed is chosen by some randomizing device that is independent of the parameter of interest.  Condition 3 says that, if we know experiment $\mathbb{E}_i$ has been chosen by the randomizing device, then the probability of obtaining $\omega$ is the same as it would have been had $\mathbb{E}_i$ been conducted without the randomizing device being consulted.
Finally, condition 4 says every component experiment has some positive probability of being chosen.

With these definitions, the statement of the \textbf{qualitative mixture sufficiency principle} (\f{qmsp}) is identical to that of \f{msp}, but one replaces the quantitative/probabilistic definitions of sufficient and mixed experiment with the qualitative ones we have just introduced.  \\

\noindent Qualitative Mixture-Sufficiency Principle ($\f{qmsp}$): Let $\omega_1$ be an outcome of a (qualitative) experiment $\mathbb{E}$ and $\omega_2$ be an outcome of (qualitative) experiment $\mathbb{F}$.  Suppose that, in every mixture $\mathbb{M}$ of $\mathbb{E}$ and $\mathbb{F}$, there is some sufficient statistic $T$ for $\mathbb{M}$ such that $T(\omega_1)=T(\omega_2)$.  Then $\omega_1$ and $\omega_2$ are evidentially equivalent.\\

Our second major result is the following:

\begin{theorem}
\label{thm:qmsp}
       If \f{qmsp} entails that 
       $\omega_1$ and $\omega_2$ are evidentially equivalent, then they are $\mathcal{B}$  posterior equivalent for all sets of orderings $\mathcal{B}$ satisfying our axioms.  If  $\mathcal{B}=\mathcal{U}$ is the set of \emph{all} orderings satisfying the axioms below, the converse holds as well.
\end{theorem}

Because  \f{qmsp} is the qualitative analog of \f{lp}, \autoref{thm:qmsp} is the qualitative analog of \autoref{clm:bLPFavoringPosteriorEquivalence}:  both show that some likelihoodist thesis about ``evidential equivalence'' is equivalent to the Robust Bayesian notions of posterior/support equivalence.

 The proof of \autoref{thm:qmsp} relies on the following assumption.

\begin{assumption}\label{assum:qMixturesExist}
For any two experiments $\mathbb{E}$ and $\mathbb{F}$, there is a ``uniform'' mixture $\mathbb{M}$ of the two experiments such that each of the two component experiments has an equal probability of being conducted, i.e., $\Omega^{\mathbb{E}} |^{\mathbb{M}}  \theta \equiv \Omega^{\mathbb{E}} |^{\mathbb{M}}  \theta$ for all $\theta$.
\end{assumption}

We conjecture that assumption is derivable (from the axioms of set theory), but we have not yet produced a proof.\footnote{We say ``\emph{a} uniform mixture'' rather than ``\emph{the} uniform mixture'' because we lack a proof that the mixture $\mathbb{M}$ is unique, even when $\mathbb{E}$ and $\mathbb{F}$ are the only component experiments of $\mathbb{F}$.  In other words, the description ``flip a fair coin to decide whether to conduct $\mathbb{E}$ or $\mathbb{F}$'' may not determine a unique qualitative likelihood ordering $\sqsubseteq$.}

\subsection{Axioms for Qualitative Probability}
\label{subsec:axioms}
We assume that both $\sqsubseteq$ and $\preceq$ satisfy the first set of axioms below.  

To state the axioms, let $\btleq$ be either $\sqsubseteq$ or $\preceq$, and let $\btl$ be the corresponding ``strict'' inequality defined by $x \btl y$ if $x \btleq y$ and $y \nbtleq x$.  Define $A | B \bteq C | D$ if and only if $A | B \btleq C| D$ and vice versa. 

\begin{itemize}
    \item[] \hypertarget{Ax1}{Axiom 1:} $\btleq$ is a weak order (i.e., it is linear/total, reflexive, and transitive).
    \item[] \hypertarget{Ax3}{Axiom 3}: $A|A \bteq B|B$ and $ A|B \btleq \Delta | C$ for all $C \not \in {\cal N}$.
    \item[] \hypertarget{Ax4}{Axiom 4}: $A \cap B | B \bteq A | B$.
    \item[] \hypertarget{Ax5}{Axiom 5}: Suppose $A \cap B = A' \cap B' = \emptyset$.  If $A|C \btleq A'|C'$ and $B|C \btleq B'|C'$, then $A \cup B | C \btleq A' \cup B'|C'$; moreover, if either hypothesis is $\btl$, then the conclusion is $\btl$.
    \item[] \hypertarget{Ax6}{Axiom 6}: Suppose $C \subseteq B \subseteq A$ and $C' \subseteq B' \subseteq A'$.  
    \begin{itemize}
        \item[] \hypertarget{Ax6a}{Axiom 6a:} If $B|A \btleq C'|B'$ and $C|B \btleq B'|A'$, then $C|A \btleq C'|A'$; moreover, if either hypothesis is $\btl$, the conclusion is $\btl$.
        \item[] \hypertarget{Ax6b}{Axiom 6b:} If $B|A \btleq B'|A'$ and $C|B \btleq C'|B'$, then $C|A \btleq C'|A'$; moreover, if either hypothesis is $\btl$ and $C \notin \nul$, the conclusion is $\btl$.
    \end{itemize}	
\end{itemize}
\noindent We discuss \Ax{2} below, as it is different for $\sqsubseteq$ and $\preceq$. The above axioms are due to  \cite[p. 222]{krantz_foundations_2006-1} and enumerated in the same order as in that text.  What we call Axiom 6b is what they call Axiom 6' (p. 227).  Our axioms are necessary for $\preceq$ to be representable by a (conditional) probability measure, but for several reasons  (e.g., there is no Archimedean condition), they are not sufficient for representation by even
a set of probability measures.  See \cite{alon_subjective_2014} for a recent representation theorem for sets of probability measures.
  
Axioms 1, 3, 4, and 5, are fairly analogous to facts of quantitative probability.  \Ax{6} is useful because it allows us to ``multiply'' in a qualitative setting.  To see the motivation for \Ax{6a}, note that if $C \subseteq B \subseteq A$ and $C' \subseteq B' \subseteq A'$, then 
$$P(C|B) = \frac{P(C)}{P(B)} \mbox{ and } P(B|A) = \frac{P(B)}{P(A)},$$ and similarly for the $A', B',$ and $C'$.  So if $P(B|A) \geq  P(C'|B')$ and $P(C|B) \geq P(B'|A')$, then 
$$\frac{P(B)}{P(A)} \geq \frac{P(C')}{P(B')}\mbox{ and } \frac{P(C)}{P(B)} \geq \frac{P(B')}{P(A')}.$$  When we multiply the left and right-hand sides of those inequalities, we obtain $P(C)/P(A)  \geq P(C')/P(A')$, which is equivalent to $P(C|A) \geq P(C'|A')$ given our assumption about the nesting of the sets.  Axiom 6b can be motivated similarly.

In addition to these axioms, in statistical contexts, one typically assumes that experimenters agree upon the likelihood functions of the data, which means that, in discrete contexts, $Q^{\mathbb{E}}(\cdot|\theta) = P^{\mathbb{E}}_{\theta}(\cdot)$ whenever $Q(\theta)>0$.  Similarly, we assume that $\preceq$ extends $\sqsubseteq$ in the following sense:
\begin{itemize}
   \item[] \hypertarget{Ax0}{Axiom 0:}
   If $B, D \not \in {\cal N}_{\preceq}$ and $B,D \in \Theta \times {\cal P}(\Omega)$, then
   $A|B \preceq C|D$ if and only if $A|B \sqsubseteq C| D$. 
\end{itemize}

We now return to Axiom 2, which concerns probability zero events.    For any given fixed $\theta$, the conditional probability $P_{\theta}(\cdot|E)$  is undefined if and only if $P_{\theta}(E) = 0$.  Similarly, we define $\nul_{\sqsubseteq}$ to be all and only sets of that form $\{\theta\} \times E$ such that $E|\theta \sqsubseteq \emptyset|\theta$. We call such events $\sqsubseteq$-\emph{null}.  Notice that an experimenter may have a prior that assigns a parameter $\theta \in \Theta$ zero probability, even if $P_{\theta}(E) > 0$ for all events $E$.  So the set of null events, in the experimenter's joint distribution, will contain all of the null events for each $P_{\theta}$ plus many others. Accordingly, we assume the following about $\preceq$-null events.

\begin{itemize}
 \item[]  \hypertarget{Ax2'}{Axiom 2} (for $\sqsubseteq$):
 $\theta \times \Omega \not \in \nul_{\sqsubseteq}$ for all $\theta$, and $A \in \nul_{\sqsubseteq}$ if and only if $A= \theta \times E$ and $E|\theta \sqsubseteq \emptyset|\theta$.

 \item[] \hypertarget{Ax2}{Axiom 2} (for $\preceq$): $\Delta \not \in {\cal N}$, and $A \in {\cal N}$ if and only if $A|\Delta \preceq \emptyset| \Delta$.
\end{itemize}

We say an expression of the form $\cdot|A$ is \emph{undefined} with respect to $\preceq/\sqsubseteq$ if $A$ is null with respect to the appropriate relation.  Because it is typically clear from context, we do not specify with respect to which ordering an expression is undefined.

The final axiom just says that that permissible priors do not vary with the experiment.

\begin{itemize}
   \item[] \hypertarget{Ax7}{Axiom 7:}
   $\theta|^\mathbb{E}\Delta^\mathbb{E} \sim \theta|^\mathbb{F}\Delta^\mathbb{F}$  
   for any $\theta \in \Theta$ and for all experiments $\mathbb{E}$ and $\mathbb{F}$ 
   containing $\Theta$ as their parameter spaces.
\end{itemize}

\Ax{7} is plausible if a prior distribution represents a researcher's degrees of belief, and the researcher believes her choice of experiment does not influence the parameter of interest.  However, \Ax{7} does conflict with ``objective Bayesian'' approaches in which, for instance, reference priors \citep{bernardo_reference_1979} or maximum entropy priors \citep{jaynes_probability_2003} are used.

Axiomatic systems for non-numerical, comparative conditional probabilities are not new; see \citep{keynes_treatise_1921, koopman_axioms_1940} for instance.  Although there are a number of small technical differences between the axioms of the various systems proposed over the last hundred years, the main difference between our framework and previous systems is that we axiomatize \emph{two} orderings, i.e., we distinguish qualitative likelihoods from qualitative priors and joint distributions.

\subsection{Favoring and QLL}
Over the next two sections, we sketch proofs of  \autoref{thm:QLLFavoring} and \autoref{thm:qmsp} so the reader can see how our axioms work in ways analogous to quantitative probability theory. For ease of reading, we have stated all the propositions and lemmata below without qualification, even though several require the assumption that one or more events are non-null.  Again, for detailed proofs that handle the case of null events, see the supplementary materials.

Our goal in this section is to sketch a proof of  \autoref{thm:QLLFavoring} specifically.   The theorem follows immediately from the following two propositions:

\begin{proposition}\label{prop:QBayesFactorFavoring}
    Suppose $H_1 \cap H_2 = \emptyset$. Then $E|H_1 \succeq E|H_2$ if and only if $H_1|E\cap (H_1 \cup H_2) \succeq H_1|H_1\cup H_2$. Further, if either side of the biconditional contains a strict inequality $\succ$, then so does the other.
\end{proposition}

\begin{proposition}\label{prop:qBayesFactorLikelihoods}
        Suppose $H_1$ and $H_2$ are finite and that $H_1 \cap H_2 = \emptyset$.  Then $E|H_2 \preceq E|H_1$ for all orderings $\preceq$ (satisfying the axioms above) if $E|\theta_2 \sqsubseteq E|\theta_1$ for all $\theta_1 \in H_1$ and $\theta_2 \in H_2$.  Further, under \autoref{assum}, if $E|H_2 \preceq E|H_1$ for all orderings $\preceq$, then $E|\theta_2 \sqsubseteq E|\theta_1$ for all $\theta_1 \in H_1$ and $\theta_2 \in H_2$. 
        
        In both directions, the left-hand side contains the strict inequality $\prec$ if and only if the right-hand side contains the strict relation $\sqsubset$.
\end{proposition}
Note that if $H_1$ and $H_2$ are simple hypotheses, i.e. $H_1 = \theta_1, H_2 = \theta_2$, then we get that $E|H_1 \preceq E|H_2$ if and only if $E|\theta_1 \sqsubseteq E|\theta_2$ by \Ax{0}. The additional assumption is required only when $H_1$ or $H_2$ are composite.

The proofs of these propositions require the following lemmata; the first two are  analogs of the fact that $P(E|H_1 \cup H_2)$ must be between $P(E|H_1)$ and $P(E|H_2)$.
\begin{lemma}\label{lm:qllcf}
    Suppose $H_1 \cap H_2 = \emptyset$. Then $E|H_1 \succeq E|H_2$ if and only if $E|H_1 \succeq E|(H_1 \cup H_2)$. Further, if either side of the biconditional is strict, the other side is strict too.
\end{lemma}

\begin{lemma}\label{lm:qbfl2} Suppose $H_1 \cap H_2 = \emptyset$.
   \begin{enumerate}
       \item If $E|H_1, E|H_2 \preceq E|H_3$, then $E|(H_1 \cup H_2) \preceq E|H_3$.
       \item If $E|H_3 \preceq E|H_1, E|H_2$, then $E|H_3 \preceq E|(H_1 \cup H_2)$.
   \end{enumerate}
   If the premise is not strict \emph{and} $E|H_1 \sim E|H_2$, then the conclusion is not strict. Otherwise, the conclusion is strict. 
\end{lemma}

Finally, we often use the following variant of \Ax{6a}: 
\begin{lemma}\label{lm:Aditya}
    Suppose $A\supseteq B\supseteq C$ and $A\supseteq B'\supseteq C$. If $B|A \succeq C|B'$ and $B \not \in \nul_{\preceq}$, then $B'|A \succeq C|B$. Further, if the antecedent is $\succ$, the consequent is $\succ$. 
\end{lemma}

\noindent \textbf{Proof of \pref{QBayesFactorFavoring}:} 
First we prove the left-to-right direction. Suppose $E|H_1 \succeq E|H_2$, where $H_1 \cap H_2 = \emptyset$.
By \autoref{lm:qllcf},  we get $E|H_1 \succeq E|(H_1 \cup H_2)$.  Applying \Ax{4} to both sides of that inequality yields 
\begin{equation}\label{eqn:qbff1}
    E\cap H_1|H_1 \succeq E\cap (H_1 \cup H_2)|H_1 \cup H_2
\end{equation}
Define:
\begin{alignat*}{3}
    &A = H_1 \cup H_2 & & \\ 
    &B =  E \cap (H_1 \cup H_2) & \quad \quad & B' = H_1\\
    &C = E\cap H_1 & &
\end{alignat*}
Note that equation \eqref{eqn:qbff1} says that $C|B' \succeq B|A$. So \autoref{lm:Aditya} tells us that $C|B \succeq B'|A$, i.e.,
    $E\cap H_1|E\cap (H_1 \cup H_2) \succeq H_1|H_1\cup H_2$
Applying \Ax{4} to the left-hand side of that  inequality yields the desired result. Note that if we had assumed $E|H_1 \succ E|H_2$, then our conclusion would contain $\succ$ because both  \autoref{lm:Aditya} and \autoref{lm:qllcf}  yield strict comparisons. 

In the right to left direction, suppose $H_1|E\cap (H_1 \cup H_2) \succeq H_1|H_1\cup H_2$. By \Ax{4}:
\begin{align*}
    H_1&|E\cap (H_1 \cup H_2) \\
    &\sim H_1\cap E\cap (H_1 \cup H_2)|E\cap (H_1 \cup H_2) \\
    &\sim E \cap H_1|E\cap (H_1 \cup H_2)
\end{align*}
So, we know $E \cap H_1|E\cap (H_1 \cup H_2) \succeq H_1|H_1\cup H_2$. Now we apply  \autoref{lm:Aditya} to
\begin{alignat*}{3}
    &A = H_1 \cup H_2 &  \\
    &B =  H_1& \quad \quad 
    &B' = E \cap (H_1 \cup H_2)  \\
    &C = E\cap H_1 & &
\end{alignat*}
The containment relations are satisfied, and $C|B' \succeq B|A$. Thus, we get $C|B \succeq B'|A$, i.e.,   $E\cap H_1|H_1 \succeq E\cap (H_1\cup H_2)|(H_1 \cup H_2).$
Applying \Ax{4} to both sides of the last inequality yields $E|H_1 \succeq E|H_1 \cup H_2$. So by \autoref{lm:qllcf}, we get $E|H_1 \succeq E|H_2$, as desired. As before, if the premise were $\succ$, the conclusion would also be $\succ$ because the necessary lemmas would yield strict comparisons.
\qed
\\

\noindent \textbf{Proof of \pref{qBayesFactorLikelihoods}:} 
In the right-to-left direction, suppose $E|\theta_2 \sqsubseteq E|\theta_1$ for all $\theta_1 \in H_1$ and $\theta_2 \in H_2$.  We want to show $E|H_2 \preceq E|H_1$ for all orderings $\preceq$ satisfying the axioms.  So let $\preceq$ be any such ordering.  We show $E|H_2 \preceq E|H_1$ by induction on the maximum of the number of elements in $H_1$ or $H_2$.

In the base case, suppose $H_1 = \{\theta_1\}$ and $H_2=\{\theta_2\}$ both have one element.  Then by \Ax{0}, it immediately follows that $E|H_2 \preceq E|H_1$.

For the inductive step, suppose the result holds for all natural numbers $m \leq n$, and assume that $H_1 = \{\theta_{1,1}, \ldots \theta_{1,k}\}$ and $H_2=\{\theta_{2,1} \ldots, \theta_{2,l}\}$ where either $k$ or $l$ (or both) is equal to $n+1$.  Define $H_1'=\{\theta_{1,1}\}$ and $H_1'' = H_1 \setminus H_1'$, and similarly, define $H_2'=\{\theta_{2,1}\}$ and $H_2'' = H_2 \setminus H_2'$.  Then $H_1', H_1', H_2',$ and $H_2''$ all have $n$ or fewer  elements, and by assumption, $E|\theta_2 \sqsubseteq E|\theta_1$ for all $\theta_2 \in H_2', H_2''$ and all $\theta_1 \in H_1', H_1''$.  By inductive hypothesis, it follows that $E|H_1', E|H_1'' \preceq E|H_2', E|H_2''$.  By repeated application of \autoref{lm:qbfl2}, it follows that $E|H_1 = E|(H_1' \cup H_1'') \preceq E|(H_2' \cup H_2'') = E|H_2$.

Note that if the premise were $\sqsubset$, \Ax{0} would prove the strict version of the base case. In the inductive step, every inequality would also be strict because \autoref{lm:qbfl2} preserves the strictness. Thus, we would get $E|H_2 \prec E|H_1$.

In the left-to-right direction, suppose $E|H_2 \preceq E|H_1$ for all orderings $\preceq$.  We want to show that $E|\theta_2 \sqsubseteq E|\theta_1$ for all $\theta_1 \in H_1$ and $\theta_2 \in H_2$.  So fix $\theta_1 \in H_1$ and $\theta_2 \in H_2$.  We must show $E|\theta_2 \sqsubseteq E|\theta_1$.  Suppose for the sake of contradiction that $E|\theta_2 \not \sqsubseteq E|\theta_1$, and thus, by totality of $\sqsubseteq$, we know $E|\theta_1 \sqsubset E|\theta_2$.   We must find at least one ordering $\preceq$ that (i) satisfies Axioms 0-6 and (ii) entails that $E|H_2 \not \preceq E|H_1$.

Using \autoref{assum} with $H = \{\theta_1,\theta_2\}$, define $\preceq$ so that it satisfies all of the axioms and:
    \begin{eqnarray*}
        \theta_1 |\Delta, \theta_2| \Delta &\succ& \emptyset|\Delta \\
        A|\Delta  &\sim& \emptyset| \Delta \mbox{ if } \theta_1, \theta_2 \not \in A
    \end{eqnarray*}

Because  $\theta_1 |\Delta, \theta_2| \Delta \succ \emptyset|\Delta$, \Ax{2} entails that $\{\theta_1\}, \{\theta_2\} \not \in \nul$. Since $\{\theta_1\}, \{\theta_2\} \not \in \nul$ and $E|\theta_1 \sqsubset E|\theta_2$, by \Ax{0} we obtain that $E|\theta_1 \prec E|\theta_2$.

To finish the proof, we need one final lemma, which is the analog of the claim that $P(A|B \cup C) = P(A|B)$ whenever $P(B) > P(C)=0$.

\begin{lemma}\label{lm:krantz9}
Suppose $C \in \nul$ and $B \not \in \nul$.  Then $A|B \sim A|B \cup C$ for all $A$.
\end{lemma}

Now consider $H_1' = H_1 \setminus \{\theta_1\}$ and $H_2' = H_2 \setminus \{\theta_2\}$.  Since $H_1$ and $H_2$ are disjoint, neither $H_1'$ nor $H_2'$ contain either $\theta_1$ or $\theta_2$.  Thus by construction of $\preceq$, it follows that $H_1', H_2' \in \nul$.  So by \autoref{lm:krantz9}, we obtain that 
\begin{eqnarray*}
E|H_1 &=& E|H_1' \cup \{\theta_1\} \sim E|\theta_1 \mbox{ and } \\
E|H_2 &=& E|H_2' \cup \{\theta_2\} \sim E|\theta_2
\end{eqnarray*}
Since $E|\theta_1 \prec E|\theta_2$, it follows that $E|H_1 \prec E|H_2$, as desired.

If the premise were $\prec$, the conclusion would clearly be $\sqsubset$, by simply swapping all strict and weak inequalities in the preceding argument.

\begin{flushright}
    $\Box$
\end{flushright}

\subsection{Posterior Equivalence and QMSP}
In this section, we prove \autoref{thm:qmsp}, which says that \f{qmsp} is equivalent to posterior equivalence.  We use the following lemmata to prove this theorem.
\begin{lemma}\label{lm:qAncillaryNoUpdate}
    If $T$ is ancillary, then for all $t \in {\cal R}$, the events $\{T=t\}$ and $\Omega$ are posterior/support equivalent.
\end{lemma}

The previous lemma captures a simple intuition. Because the distribution of an ancillary statistic does not depend on the parameter of interest, learning the value of an ancillary statistic should not change one's beliefs about the parameter of interest:  ancillary statistics are uninformative, and so one's posterior beliefs (i.e., those conditional on   $\{T=t\}$) should equal one's prior beliefs (i.e., those conditional on the sure event $\Omega$).  For instance, if a callous experimenter flips a coin to select which patients will receive a treatment in a drug trial, then the sequence of coin flips should not change her beliefs about the efficacy of the drug!

\newcommand{\ce}{|^{\mathbb{E}}}
\newcommand{\cf}{|^{\mathbb{F}}}
\newcommand{\cm}{|^{\mathbb{M}}}

\begin{lemma}\label{lm:qBayesLTPGeneral}
    Let $A, B_1 \ldots B_n$ be events in some experiment $\mathbb{E}$, and let $C_1, \ldots C_n, B_1' \ldots B_n'$ be events in an experiment $\mathbb{F}$ (which may or may not be identical to $\mathbb{E}$).  Let $G = \bigcup_{i \leq n} B_i$ and $G' = \bigcup_{i \leq n} B_i'$. 
    Further, suppose
    \begin{enumerate}
        \item $B_1, \ldots, B_n$ are disjoint,
        \item $B_1', \ldots, B_n'$ are disjoint,
        \item $B_i|G \sim B_i'|G'$ for all $i \leq n$, and
        \item  $A \ce B_i \sim C_i \cf  B'_i$  for all $i \le n$.
   \end{enumerate}
 If $A \cap G \not \in \nul$, then  $B_i \ce A \cap G \sim B_i \cap C_i \cf \bigcup_{j \le n} B_j' \cap C_j$ for all $i \le n$.
\end{lemma}

\autoref{lm:qBayesLTPGeneral} can be thought of as an application of Bayes' theorem combined with the Law of Total Probability.  Consider the simple case in which $C_i = C$ and $B_i = B_i'$ for all $i$.  Further, suppose that $G=G' = \Delta$ is the sure event, i.e., the $B_i$s partition the entire space.  In that case, the quantitative version of the lemma says that, if $P(A|B_i)=P(C|B_i)$ for all $i$, then $P(B_i|A)=P(B_i|C)$ for all $i$. That claim follows easily from an application of Bayes' theorem and the Law of Total Probability.

\renewcommand{\oe}{\omega^{\mathbb{E}}}
\newcommand{\of}{\omega^{\mathbb{F}}}
\newcommand{\ome}{\omega^{\mathbb{M},\mathbb{E}}}
\newcommand{\omf}{\omega^{\mathbb{M},\mathbb{F}}}

Earlier, we noted that we may sometimes need to write $\langle \mathbb{E}, \omega \rangle$ to indicate that the outcome $\omega$ was observed in experiment $\mathbb{E}$, as the same outcome could appear in multiple experiments. That distinction is now relevant. Let $\oe = \langle \mathbb{E}, \omega \rangle$.  Importantly, $\oe$ must be distinguished from the situation in which (i) one conducts a mixed experiment $\mathbb{M}$ containing $\mathbb{E}$ as a component, (ii) $\mathbb{E}$ is selected (by one's randomizing device) as the component experiment to be performed, and (iii) one observes $\omega$.    The latter situation is represented by the pair $\ome := \langle \mathbb{M}, \langle \mathbb{E}, \omega \rangle \rangle$.


\begin{lemma} \label{lem:mixtureComponentEquivalence}
    Let $\omega$ be an outcome of some experiment $\mathbb{E}$ and let $\mathbb{M}$ be an arbitrary mixture including $\mathbb{E}$.  Then  $\ome$  is posterior equivalent to $\oe$.
\end{lemma}
\begin{proof}
    This follows directly from \autoref{lm:qBayesLTPGeneral}, with the appropriate substitutions.  Enumerate the elements of $\Theta = \{\theta_1, \ldots \theta_n\}$ and let:
    \begin{align*}
        A &= \oe \\
        B_i &= \theta_i \\
        B'_i &= \theta_i \cap \Omega^\mathbb{E} \\
        C_i &= \ome 
    \end{align*}
    We now show that the four conditions necessary to apply \autoref{lm:qBayesLTPGeneral} hold. Obviously, disjointness holds. Condition 4 says that $\oe \ce \theta_i \sim \ome \cm \theta_i \cap  \Omega^\mathbb{E}$, which is the third condition in the definition of a mixture.
    
    Condition 3 says that $\theta_i|^\mathbb{E}\Theta \sim \theta_i|^\mathbb{M}\Theta \cap \Omega^\mathbb{E}$. We prove this in two steps. First, recall that $\theta_i|^\mathbb{E} \Theta \sim \theta_i|^\mathbb{M} \Theta$, as we assume that priors don't vary across experiments (\Ax{7}). Next, the fourth condition in the definition of a mixture says that the projection statistic is ancillary. Thus, we can apply \autoref{lm:qAncillaryNoUpdate} to the projection statistic, giving us that $\Omega^\mathbb{E}$ is posterior equivalent to $\Omega^\mathbb{M}$ (and recall that $\Omega^\mathbb{M}$ and $\Theta$ are both shorthand for the same event, $\Delta^\mathbb{M}$). This means that $\theta_i|^\mathbb{M} \Theta \sim \theta_i|^\mathbb{M} \Theta, \Omega^\mathbb{E}$. So, we've shown Condition 3 by transitivity.
    
    Now, \autoref{lm:qBayesLTPGeneral} says that $\theta_i|^\mathbb{E} \ome \sim \theta_i|^\mathbb{M} \ome, \Omega^\mathbb{E} \sim \theta_i|^\mathbb{M} \ome$, as desired.
\end{proof}

\begin{proposition}\label{prop:qSufficientWeakPosteriorEquivalence}
    Suppose $\Theta$ is finite and let $\omega_1, \omega_2 \in \Omega^\mathbb{E}$ be arbitrary. Then the following claims are equivalent:
    \begin{enumerate}
        \item $\omega_1$ and $\omega_2$ are posterior/support equivalent.
        \item There is a sufficient statistic $T$ such that $T(\omega_1) = T(\omega_2)$ 
    \end{enumerate}
\end{proposition}

\begin{proof}[Proof of \autoref{thm:qmsp}]
    In the forward direction we prove that if \f{qmsp} entails that $\oe_1$ and $\of_1$ are evidentially equivalent, then they are $\mathcal{U}$ posterior equivalent, where $\mathcal{U}$ is the set of all orderings satisfying our axioms. So assume that, in every mixture $\mathbb{M}$ of $\mathbb{E}$ and $\mathbb{F}$, there is some sufficient statistic $T$ for $\mathbb{M}$ such that $T(\ome_1) = T(\omf_2)$. By \autoref{prop:qSufficientWeakPosteriorEquivalence}, this means that $\ome_1$ and $\omf_2$ are posterior equivalent. From \autoref{lem:mixtureComponentEquivalence}, we know that $\ome_1$ is posterior equivalent to $\oe_1$ and that $\omf_2$ is posterior equivalent to $\of_2$. By the transitivity of posterior equivalence, we conclude that $\oe_1$ and $\of_2$ are posterior equivalent, as desired.
    
    The reverse direction follows by symmetry; each step of the above argument is reversible. In more detail, first assume that $\oe_1$ and $\of_2$ are posterior equivalent; by the same steps as above, we conclude that $\ome_1$ and $\omf_2$ are posterior equivalent in any mixture $\mathbb{M}$. Then we use \autoref{prop:qSufficientWeakPosteriorEquivalence} to conclude that there is a sufficient statistic $T$ such that $T(\ome_1) = T(\omf_2)$, and thus $\oe_1$ and $\of_2$ are evidentially equivalent by \f{qmsp}.
\end{proof}

\section{Major Differences between the Qualitative and Quantitative Frameworks}
\label{sec:qualdiff}

In the previous section, we showed that some likelihoodist principles extend naturally to the qualitative setting. But one might ask, ``Which aspects of quantitative probabilistic reasoning do \textit{not} extend to the qualitative setting?'' In this section, we answer that question. 

In \autoref{subsec:noQLL+} we show that there is no natural qualitative analog of $\f{ll}^+$.  Recall, in the quantitative setting, the likelihood function fully captures our notion of robustness:  one can determine when all Bayesians will agree on some evidential statement about support, favoring, or posterior equivalence using only facts about the likelihood function.  
In the qualitative setting, however, the positions of $E$ and $F$ in the likelihood \emph{ordering} are no longer sufficient to determine which hypotheses the two observations support and to what degree.  Thus, to explore which hypotheses enjoy robust confirmation from the data, one must resort to our definition of support, which does easily translate to the qualitative setting.

In \autoref{subsec:JointNotUnique}, we highlight one possible explanation for why $\f{ll}^+$ does not extend qualitatively. In the quantitative setting, an experimenter's prior distribution and the likelihood function together determine a unique posterior distribution. In contrast, in the qualitative setting, an agent's prior ordering and the the likelihood ordering fail to determine a unique posterior ordering. Intuitively, the failure may explain why $\f{ll}^+$ cannot be extended qualitatively -- the likelihood ordering is simply not rich enough. But we have no formal argument for this.   Future research should investigate the following. Let ${\cal A}$ be an axiomatic system for qualitative probability extending our axioms, such that priors and likelihoods always determine a unique joint distribution in $\mathcal{A}$. Is there a qualitative version of $\f{ll}^+$ in $\mathcal{A}$?

In \autoref{subsec:qSupportNotTotal} we show that qualitative support is not total order over simple hypotheses. 

\subsection{No Qualitative Analog of LL+}
\label{subsec:noQLL+}

\citet[p. 22]{royall_statistical_1997} emphasizes that one virtue of $\f{ll}^+$ is that it entails the ``irrelevance of the sample space.''  Royall's point is that to assess whether $P^{\mathbb{E}}_{\theta_1}(\omega_1)/P^{\mathbb{E}}_{\theta_2}(\omega_1) 
\geq P^{\mathbb{F}}_{\theta_1}(\omega_2)/P^{\mathbb{F}}_{\theta_2}(\omega_2)$ or vice versa, one need not know what other observations one \emph{might} have made in the experiments $\mathbb{E}$ and $\mathbb{F}$: one needs to know only the likelihood functions of $\omega_1$ and $\omega_2$.  Birnbaum notices that \f{lp} has a similar implication:  to determine whether $P^{\mathbb{E}}_{\theta}(\omega_1) = c \cdot P^{\mathbb{F}}_{\theta}(\omega_2)$ for all $\theta$ -- and hence to determine whether $\omega_1$ and $\omega_2$ are ``evidentially equivalent'' -- one need know only the likelihood functions of $\omega_1$ and $\omega_2$.  \citet[p. 271]{birnbaum_foundations_1962} concludes  that \f{lp} entails the ``irrelevance of outcomes not actually observed.''  

In this section, we argue that Royall's and Birnbaum's remarks are misguided:  if one endorses the Robust Bayesian viewpoint, then one cannot in general determine whether $E$ supports one hypothesis over another as much as $F$ if one knows only how $E|\theta$ and $F|\theta'$ are ordered across all $\theta, \theta' \in \Theta$.    Thus, we also reject both \citep{birnbaum_foundations_1962}'s claim that  ``the evidential \emph{meaning} of any outcome $x$ of any experiment $E$ is characterized fully by giving the likelihood function [of $x$]'' and \citep{berger_likelihood_1988}'s similar thesis that ``all evidence, which is obtained from an experiment, about an unknown quantity $\theta$, is contained in the likelihood function of $\theta$ for the given data.''

Before stating our main result, it is important to notice that Royall undersells the convenience of $\f{ll}^+$.  In the quantitative setting, one can determine whether $\omega_1$ Bayesian supports $\theta_1$ over $\theta_2$ at least as much as $\omega_2$ by comparing only \emph{four} numbers: $P_{\theta_1}(\omega_1), P_{\theta_2}(\omega_1), P_{\theta_1}(\omega_2),$ and $P_{\theta_2}(\omega_2)$. That is,  $\f{ll}^+$ entails that it is unnecessary to know the \emph{full} likelihood functions of $\omega_1$ and $\omega_2$ to assess whether the former supports a particular hypothesis over another. 

Our main result is the following:

\begin{theorem}\label{thm:qLL+NoAnalog}
    There exist two qualitative experiments such that 
        \begin{itemize}
            \item  $E|\theta$ and $F|\theta'$ are ordered identically in the two experiments, for all $\theta, \theta'$.
            \item In the first experiment, $E$ qualitatively supports $\theta_1$ over $\theta_2$ strictly more than $F$, but in the second experiment, the support relation is reversed. 
        \end{itemize} 
\end{theorem}

The proof of \autoref{thm:qLL+NoAnalog}  relies on another substantial theorem.  To state the theorem, we must state a qualitative notion of conditional independence.  Given an ordering $\btleq$ satisfying the above axioms (so $\btleq$ could be either $\preceq$ or $\sqsubseteq$), define:

\begin{definition}\label{defn:qConditionalIndependence}
For any $C \not \in \nul$, the events $A$ and $B$ are said to be \textbf{conditionally independent} given $C$ if $A|B \cap C \bteq A|C$ or $B \cap C$ is null.   In this case, we write $A \ind_C B$.  When $C$ is the sure event, we say $A$ and $B$ are \textbf{independent} (full stop); in that case, we drop the subscript $C$ and write $A \ind B$.
\end{definition}

The relation $\ind$ so-defined satisfies the semi-graphoid properties and is thus an appropriate qualitative analog of conditional probability. See supplemental materials for proofs.\footnote{See \citep{pearl_causality:_2000} for statements of the semi-graphoid properties, some of which were studied earlier in \citep{dawid_conditional_1979}. When  $\btleq$ is $\preceq$, the relation $\ind$ satisfies the semi-graphoid properties without qualification.  When $\btleq$ is $\sqsubseteq$, however,  minor adjustments to the statements of the properties are necessary because some events can appear on only one side of the conditioning bar in expressions involving $\sqsubseteq$.  For instance,  $\theta$ can appear to the right but not left of the conditioning bar in expressions involving $\sqsubseteq$, and this limitation means $\ind$ is not fully symmetric, as the semi-graphoid properties require.}

 \begin{theorem}
    \label{thm:qLPEntailsPosteriorEquivalence}
    Suppose $\Theta$ is finite and that there exist events $\{C_{\theta} \}_{\theta \in \Theta}$ such that
    \begin{enumerate}
            \item $E| \theta \equiv F \cap C_{\theta} | \theta$ for all $\theta \in \Theta$, 
            \item $F \ind_\theta C_{\theta}$ with respect to $\sqsubseteq$ for all $\theta \in \Theta$,
            \item $C_{\theta} | \theta \equiv C_{\eta} | \eta$ for all $\theta, \eta \in \Theta$.
            \item $\emptyset|\theta \sqsubset C_\theta|\theta$ for all $\theta$.
    \end{enumerate}	 
    If $\Theta$ is finite, then $E$ and $F$ are  $\mathcal{U}$-posterior/support equivalent.
\end{theorem}

\autoref{thm:qLPEntailsPosteriorEquivalence} is a qualitative analog of \f{lp} that has a fairly intuitive quantitative analog (one can think of the events $C_{\theta}$ as encoding the constant $c > 0$ in \f{lp}), but we omit its proof and a discussion of its motivation.  See the supplementary materials for more detail.

\begin{proof}[Proof of \autoref{thm:qLL+NoAnalog}]
Perhaps surprisingly, the experiments can be constructed over a parameter space $\Theta$ containing only two elements, $\theta_1$ and $\theta_2$.   Consider two experiments with likelihood orderings summarized by the numerical probabilities in the table below.  In both experiments, we pull \emph{twice}  (with replacement) from an unknown urn.  The urn may be of two types, represented by $\theta_1$ and $\theta_2$ respectively in the tables below.  

    \begin{table}[H]
    \centering
    \begin{tabular}{l|llll}
               & Red & White & Cobalt &  Black \\ 
               \hline
    $\theta_1$ & 2/200 & 1/200  & 100/200 & 97/200   \\
    $\theta_2$ & 20/200  & 9/200 & 100/200 & 70/200 
    \end{tabular}
    \caption{Experiment $\mathbb{E}_1$}
    \end{table}
    \begin{table}[H]
    \centering
    \begin{tabular}{l|llll}
               & Red  & White  & Cobalt &  Black \\ 
               \hline
    $\theta_1$ & 2/200 & 1/200  & 100/200 & 97/200   \\
    $\theta_2$ & 20/200  & 11/200 & 100/200 & 69/200 
    \end{tabular}
    \caption{Experiment $\mathbb{E}_2$}
    \end{table}
    
Now consider the qualitative likelihood ordering generated by the above two experiments.  In other words, define $\sqsubseteq$ so that $E |^{\mathbb{E}_i} F, \theta \sqsubseteq E' |^{\mathbb{E}_j} F', \theta'$ if and only if $P^{\mathbb{E}_i}_{\theta}(E|F) \leq P^{\mathbb{E}_j}_{\theta'}(E'|F')$, where $i, j \in \{1,2\}$ range over the two possible experiments.

For example, let White and Red respectively denote the events that we draw a white/red ball on the first draw. Let Cobalt$_2$ denote the event we draw a Cobalt ball on the second draw.  Then since 
\[
P^{\mathbb{E}_1}_{\theta_1}(\mbox{Red}) \leq P^{\mathbb{E}_1}_{\theta_2}(\mbox{White}) \leq  P^{\mathbb{E}_2}_{\theta_2}(\mbox{White})\]
we have that
\[
\mbox{Red}|^{\mathbb{E}_1}~ \theta_1 \sqsubseteq \mbox{White} |^{\mathbb{E}_1}~ \theta_2 \sqsubseteq \mbox{White}|^{\mathbb{E}_2}~ \theta_2 \]
Similarly, since the draws are assumed to be conditionally independent (given the choice of urn), we have 
\[ P^{\mathbb{E}_1}_{\theta_1}(\mbox{Red} \cap  \mbox{Cobalt}_2) = 2/100 \cdot 1/2  = P^{\mathbb{E}_1}_{\theta_1}(\mbox{White}),\]
which translates to the qualitative condition that
\[ \mbox{Red} \cap  \mbox{Cobalt}_2|^{\mathbb{E}_1} \theta_1 \equiv  \mbox{White}|^{\mathbb{E}_1} \theta_1 \]

Now that the reader understands how to use the table to find the relevant qualitative conditions, we proceed with the proof.

First, note that in both experiments we have:
    \begin{align*}
        \text{White}|\theta_1 \sqsubset \text{Red}|\theta_1 \sqsubset \text{White}|\theta_2 \sqsubset \text{Red}|\theta_2
    \end{align*}
This gives us the first claim in the theorem above: the events White and Red are ordered identically by $\sqsubseteq$ in both experiments.

We now show Red qualitatively supports $\theta_1$ over $\theta_2$ strictly more than White in experiment 1, while the reverse is true is experiment 2. 
The following is the key proposition, as it describes sufficient conditions for (strict) qualitative support.
    \begin{proposition}\label{prop:qLL+NoAnalogMain}
        Suppose the following three conditions hold:
        \begin{enumerate}
            \item $B_1 \cap B_2 = \emptyset$,
            \item $A_1 | B_1 \succ A_2 | B_1 \succ \emptyset$, and
            \item $A_2 | B_2 \succeq A_1 |B_2 \succ \emptyset$.
        \end{enumerate}
        Then 
        \[ B_1 | A_1 \cap (B_1 \cup B_2) \succ B_1 | A_2 \cap (B_1 \cup B_2)\]
        If the three conditions hold for any prior that agrees with our axioms, then $A_1$ qualitatively supports $B_1$ over $B_2$ strictly more than $A_2$.
    \end{proposition}
To apply the proposition, let $B_i = \theta_i$, $A_1 =\text{Red} \cap \text{Cobalt}_2$, and $A_2 =$ White. All three conditions of \autoref{prop:qLL+NoAnalogMain} hold for any prior that agrees with the likelihood ordering in experiment 1; thus, $A_1 \sbaf{}{B_1}{}{B_2} A_2$. In experiment 2, the same conditions obtain if we swap $A_1$ and $A_2$; thus we get $A_2 \sbaf{}{B_1}{}{B_2} A_1$. However, $A_1$ and $A_2$ are not ordered identically by the likelihoods in the two experiments. 
    
We now claim that $A_1 =$ Red, Cobalt$_2$ is qualitative support equivalent to the event Red. \autoref{thm:qLPEntailsPosteriorEquivalence} entails exactly that. We can apply that theorem with $E = R$ and $F = A_1$. Let $C_{\theta} =$ Cobalt$_2$ for all $\theta$. Conditions 1, 3, and 4 follow obviously from the likelihoods in both experiments. For Condition 2, note that:
    \begin{align*}
        \text{Red} \cap \mbox{Cobalt}_2 | \text{Cobalt}_2, \theta \equiv \text{Red}|\text{Cobalt}_2, \theta \equiv \text{Red}|\theta
    \end{align*}
where the first equality follows from \Ax{4} and the second equality follows from the fact that the first draw is conditionally independent of the second draw, given $\theta$. Thus, $F \ind_\theta C_\theta$ for all $\theta$.
    
With the four conditions satisfied, the theorem tells us that $A_1$ and Red are posterior/support equivalent. Combining this with our earlier application of \autoref{prop:qLL+NoAnalogMain} means that Red$\sbaf{}{\theta_1}{}{\theta_2}$White in experiment 1, while White$\sbaf{}{\theta_1}{}{\theta_2}$Red in experiment 2.
\end{proof}

Let's summarize the upshot of the above theorem.  Although the two experiments order four important events (namely, $\text{Red}|\theta_1, \text{Red}|\theta_2, \text{White}|\theta_1$ and $\text{White}|\theta_2$) in the same way, they do \emph{not} order \emph{all} events identically.  Other events (e.g., Cobalt$_2$) influence the evidential impact of Red and White.  This is why, contra Royall's and Birnbaum's claims, the sample space is not ``irrelevant.''

\subsection{Qualitative Likelihoods and Priors do not Determine a Unique Posterior}
\label{subsec:JointNotUnique}
When working with numerical probabilities, the likelihood function of $E$ and prior $P$ together determine the posterior $P(H|E)$ by Bayes Rule and the law of total probability:
\begin{align*}
    P(E|H) &= \sum_{\theta \in H} P(E|\theta) P(\theta|H) \\
    P(H|E) &= \frac{P(E|H)P(H)}{P(E)}
\end{align*}
However, in the qualitative setting, a prior and likelihood ordering do not determine the posterior.

Consider the following scenarios, in which one flips a coin of unknown bias a single time. The parameter space $\Theta = \{\theta_1, \theta_2\}$ denotes the unknown bias of the coin, and the set of outcomes is $\Omega = \{\omega_H, \omega_T\}$.\\

\noindent \textbf{Scenario 1:} Your prior is defined by $P(\theta_1) = 0.49$ and $P(\theta_2) = 0.51$.   The likelihood functions of $\omega_1$ and $\omega_2$ are determined by two numbers:  $P_{\theta_1}(\omega_H) = 0.99$ and $P_{\theta_2}(\omega_H)  = 0.01$. In this case, $P(\theta_1| \omega_H) \approx 0.99$ while $P(\theta_2|\omega_H) \approx 0.01$.\\

\noindent \textbf{Scenario 2:} Your prior is defined by $P(\theta_1) = 0.01$ and $P(\theta_2) = 0.99$.   Again, the likelihood functions of $\omega_1$ and $\omega_2$ are determined by two numbers:  $P_{\theta_1}(\omega_H) = 0.51$ and $P_{\theta_2}(\omega_H)  = 0.49$. In this case, $P(\theta_1| \omega_H) \approx 0.01$ while $P(\theta_2|\omega_H) \approx 0.99$.\\

Intuitively, in scenario 1 your evidence ($\omega_H$) very strongly favors the hypothesis $\theta_1$ that you had a slight prior bias against. In scenario 2, your evidence very weakly favors the hypothesis that you had a large prior bias against. 

Define a joint ordering $\preceq$ using the numerical probabilities above, i.e., define $A|^{i}B \preceq A'|^{j} B'$ if and only if $P^i(A|B) \leq P^j(A|B)$ where $i, j = 1,2$ range over the two scenarios.

Notice that the qualitative likelihood orderings in both scenarios are identical:   $\omega_H|\theta_1 \equiv \omega_T|\theta_2 \sqsupset \omega_T|\theta_1 \equiv \omega_H|\theta_2$. Furthermore, both scenarios have the same qualitative prior orderings over hypotheses:  $\theta_1 \preceq \theta_2$. But in scenario 1, the posterior ordering is $\theta_1 | \omega_H \succ \theta_2|\omega_H$, while in scenario 2, we get $\theta_1 | \omega_H \prec \theta_2 | \omega_H$. Thus, a qualitative likelihood ordering and prior do not determine a unique qualitative posterior.

The example highlights the distinction between two questions that \citep{royall_statistical_1997} carefully distinguishes:  (1) What does the evidence support? and (2) What should I believe? Because the likelihood orderings are identical in the two scenarios, the (qualitative) favoring and posterior/support equivalence relations \textit{are} the same in the two experiments. 
For example, in this example, $\omega_H$ favors $\theta_1$ over $\theta_2$ in both scenarios, and there is no non-trivial support equivalence in either scenario. In both scenarios, the answers to Royall's first question is the same.  But the answer to the second differs:  the strength of your prior convictions determines what you should believe, and that strength might fail to be captured in the qualitative prior ordering.


\subsection{Qualitative Support is not Total over Simple Hypotheses}
\label{subsec:qSupportNotTotal}
In the quantitative case, for any set of outcomes $E$ and $F$ and any simple hypotheses $\theta_1$ and $\theta_2$ that are compatible with both $E$ and $F$,\footnote{In other words, $P_\theta(E), P_\theta(F) > 0$ for $\theta \in \{\theta_1, \theta_2\}$.} it must be that $E$ Bayesian supports $\theta_1$ over $\theta_2$ at least as much as $F$, or vice versa. Succinctly, Bayesian support between two simple hypotheses is total over compatible pieces of evidence. In the qualitative case, this is not true. 

Consider the following experiments:
\begin{table}[H]
\centering
\begin{tabular}{l|lll}
           & $\omega_1$  & $\omega_2$  & $\omega_3$ \\ \hline
$\theta_1$ & 0.02 & 0.15 & 0.965 \\
$\theta_2$ & 0.01 & 0.05 & 0.985 
\end{tabular}
\caption{Experiment A}
\end{table}
\begin{table}[H]
\centering
\begin{tabular}{l|lll}
           & $\omega_1$ & $\omega_2$  & $\omega_3$ \\ \hline
$\theta_1$ & 0.04 & 0.15 & 0.945 \\
$\theta_2$ & 0.01 & 0.05 & 0.985 
\end{tabular}
\caption{Experiment B}
\end{table}

Notice that, \emph{quantitatively}, in Experiment A, $\omega_2$ strictly Bayesian supports $\theta_1$ over $\theta_2$ more than $\omega_1$, while in Experiment $B$, the reverse is true. 

However, both experiments induce the same qualitative likelihood orderings.\footnote{The full qualitative likelihood ordering is $\emptyset \sqsubset F|\theta_2 \sqsubset E|\theta_2 \sqsubset E,F|\theta_2 \equiv F|\theta_1 \sqsubset E|\theta_1 \sqsubset E,F|\theta_1 \sqsubset G|\theta_1 \sqsubset G,F|\theta_1 \sqsubset G|\theta_2 \equiv E,G|\theta_1 \sqsubset G,F|\theta_2 \sqsubset G,E|\theta_2 \sqsubset \Omega$. }  Thus, observing $\omega_i$ in Experiment A is \emph{qualitatively} support equivalent to observing $\omega_i$ in Experiment B.  Thus, neither $\omega_1$ nor $\omega_2$  strictly \emph{qualitatively} supports $\theta_i$ over $\theta_j$ more than the other, as we have just seen that there are two numerical probability functions compatible with the (qualitative) likelihood ordering in which the quantitative support relations differ.


\section{The Stopping Rule Principle}
\label{sec:sps}
One of the most controversial consequences of \f{lp} is that it entails the ``irrelevance'' of stopping rules (e.g., see Savage's contribution to \cite{savage_foundations_1962}).  This is sometimes called the \emph{stopping rule principle} (\f{srp})  \citep[p. 76]{berger_likelihood_1988}.   Many statisticians and scientists reject \f{srp}, arguing that one's method for analyzing data should reflect the rule one uses to halt sampling; see \citep[pp. 42-51]{mayo_statistical_2018}  and references therein.  For example, according to critics of \f{srp}, a scientist who samples until she acquires a statistically significant result should analyze her data differently than a scientist who acquired the exact same data but had planned -- prior to the experiment -- to stop sampling after a fixed number of observations \citep{anscombe_fixed-sample-size_1954, fletcher_stopping_2022}.    

Our final major result is that \f{rb} likewise entails the irrelevance of a class of so-called ``non-informative stopping rules'' \citep{raiffa_applied_1961}, even in the qualitative setting.  Thus, one cannot object to \f{srp} simply by rejecting \f{lp} or orthodox Bayesianism:  one must reject substantially weaker principles about rational belief and evidence.

Let us first clarify what it means for a stopping rule to be ``irrelevant.'' From the robust Bayesian perspective, the choice between two non-informative stopping rules is \emph{irrelevant} if, whenever an outcome $\omega$ can be obtained from two experiments $\mathbb{E}$ and $\mathbb{F}$ that \emph{differ only} in their respective non-informative stopping rules, 
observing $\omega$ in $\mathbb{E}$ is posterior equivalent to observing $\omega$ in $\mathbb{F}$.  Below, we define ``stopping rule'' and ``non-informative''; we also explain what it means for experiments to ``differ only'' in their stopping rules.

\subsection{Non-Informative Stopping Rules}

Suppose $\theta_1 \in \Theta_1$ is the parameter of interest but that the probability of various experimental outcomes may be partially determined by a nuisance parameter $\theta_2 \in \Theta_2$.  Nonetheless, often the data-generating process in an experiment $\mathbb{E}$ is decomposable into two parts \citep[\S 2.3]{raiffa_applied_1961}.  The first part $$h^{\mathbb{E}}(\vec{x}; \theta_1,\theta_2) = f^{\mathbb{E}}(x_1; \theta_1, \theta_2) \cdot  f^{\mathbb{E}}(x_2|x_1; \theta_1, ,\theta_2) \cdots  f^{\mathbb{E}}(x_n|x_1, \ldots x_{n-1}; \theta_1,\theta_2)$$ determines the chances that a measurement device takes a sequence of values  $\vec{x} = \langle x_1, \ldots x_n \rangle$, assuming precisely $n$ data points are sampled. 
The second  component $\phi(k+1| x_1, \ldots, x_k; \theta_1, \theta_2)$, which is called the \emph{stopping rule},  determines the probability that an additional measurement is made at all, and its values might be affected by innumerable factors in addition to the value of the parameter of interest (e.g., budget, time, etc.); we can represent those factors by the nuisance parameter $\theta_2$.  Given an experimenter has observed a sequence of $n$ many points $\vec{x}$, the probability of stopping is 
$$s^{\mathbb{E}}(n|\vec{x}; \theta_1, \theta_2)= \phi(1| \theta_1, \theta_2) \cdots \phi(n| x_1, \ldots , x_{n-1}; \theta_1, \theta_2) \cdot (1-\phi(n+1| \vec{x}; \theta_1, \theta_2)),$$
and so letting $\theta = \langle \theta_1, \theta_2 \rangle$, one can factor the likelihood function as follows:
$$P_\theta^{\mathbb{E}}(\vec{x}) = h^{\mathbb{E}}(\vec{x}; \theta_1, \theta_2) \cdot s^{\mathbb{E}}(n|\vec{x}; \theta_1, \theta_2)$$

With this decomposition, define two experiments $\mathbb{E}$ and $\mathbb{F}$ to \textbf{differ only in stopping rule} if $f^{\mathbb{E}} = f^{\mathbb{F}}$, or equivalently, if $h^{\mathbb{E}} = h^{\mathbb{F}}$.

Everyone agrees that two experiments may differ only in stopping rule, and yet a common outcome $\omega$ of the two experiments may carry different evidential weight depending on the experiment in which it is obtained.  Why?  Some stopping rules are informative:  if $s^{\mathbb{E}}$ depends upon the parameter of interest $\theta_1$, then the point at which the experiment halts may provide one with information about the parameter.   This motivates the following definition:

\begin{definition}\label{defn:NonInformative}
$\mathbb{E}$ has a $\Theta_1$ \textbf{non-informative} stopping rule if
\begin{enumerate}
    \item $h^{\mathbb{E}}$ depends only on $\theta_1$, and
    \item $s^{\mathbb{E}}$ depends only on $\theta_2$.
\end{enumerate}
\end{definition}

Our definition is inspired by \citep[pp. 37-38]{raiffa_applied_1961}'s definition, but it differs in two ways.  First, our first condition is implicitly assumed in Raiffa and Schlaifer's presentation, as they define $f$ and $h$ to depend only on the parameter of interest $\theta_1$.  Second, we omit Raiffa and Schlaifer's requirement that ``$\theta_1$ and $\theta_2$ are independent \emph{a priori}.'' We omit that requirement because (1) we want to \emph{try} to formulate the stopping rule principle in a manner that does not presuppose a particular view on foundations of statistics (i.e., and so does not presuppose Bayesianism, likelihoodism, or some classical/frequentist approach), and (2) Raiffa and Schlaifer's additional requirement is Bayesian in nature, as ``independent a priori'' means that $\theta_1$ and $\theta_2$ are independent with respect to the experimenter's prior probability function.  Nonetheless, their second condition is critically important for reasons to be discussed momentarily.

The \textbf{stopping rule principle} (\f{srp}) asserts that if $\omega$ is a common outcome of two experiments $\mathbb{E}$ and $\mathbb{F}$ that differ only in  non-informative stopping rules, then $\langle \mathbb{E}, \omega \rangle$ and $\langle \mathbb{F}, \omega \rangle$ are evidentially equivalent.  Like \f{ll}, \f{lp}, and the related likelihoodist principles above, the statement of \f{srp} makes reference only to likelihood functions. 

The robust Bayesian perspective partially vindicates \f{srp}.  Why only a ``partial'' vindication?  If a decision-maker believes that learning about $\theta_2$ can tell one about $\theta_1$ or vice versa, then virtually no stopping rule is going to be uninformative to her, unless the rule depends on neither $\theta_1$ nor $\theta_2$.     Thus, from the robust Bayesian perspective, there are no stopping rules that \emph{all} decision-makers will regard as uninformative.  Nonetheless, the robust Bayesian can provide a non-trivial characterization of which rules are uninformative for decision makers who regard $\theta_1$ and $\theta_2$ as independent \emph{a priori}.  Let ${\cal B}(\ind)$ denote that set of priors.  

\begin{theorem}\label{thm:SRPEntailsPosteriorEquivalence}
\cite[p. 38]{raiffa_applied_1961}
   Suppose two experiments $\mathbb{E}$ and $\mathbb{F}$ differ only in their $\Theta_1$-non-informative stopping rules and that $\omega$ is a common outcome of both experiments.  For any hypothesis $H \subseteq \Theta_1$ about the parameter of interest and $\overline{H} = H \times \Theta_2$.  Then $Q^{\mathbb{E}}(\overline{H}| \omega) = Q^{\mathbb{F}}(\overline{H}| \omega)$  for all $H \subseteq \Theta_1$ and all priors $Q \in {\cal B}(\ind)$. In other words, if \f{srp} entails that two outcomes are evidentially equivalent, then the outcomes are ${\cal B}(\ind)$-posterior equivalent with respect to all hypotheses about the parameter of interest.
\end{theorem}

A special - but still controversial - case of \f{srp} is when two experiments $\mathbb{E}$ and $\mathbb{F}$ differ only in non-informative stopping rules that (a) depend on \emph{neither} parameter and (b) are always identically zero or one.    We call such stopping rules \textbf{data-driven} since they depend only the observed sample.

\begin{definition}\label{defn:DataDriven}
An experiment $\mathbb{E}$ has a \textbf{data-driven}  stopping rule if $s^{\mathbb{E}}(n|\vec{x})$ depends on neither parameter and is always identically zero or one. 
\end{definition}

If a sample $\vec{x}$ can be obtained with positive probability in two experiments with data-driven stopping rules, then
$$P^{\mathbb{E}}_{\theta}(\vec{x}) = h(\vec{x}; \theta) \cdot s^{\mathbb{E}}(n|\vec{x})  = h(\vec{x}; \theta) = h(\vec{x}; \theta) \cdot s^{\mathbb{F}}(n|\vec{x})= P^{\mathbb{F}}_{\theta}(\vec{x}).$$
Thus, the likelihood function of $\vec{x}$ is the \emph{exactly the same} (not just proportional!) in both experiments, and so observing $\vec{x}$ in $\mathbb{E}$ is posterior equivalent to observing $\vec{x}$ in $\mathbb{F}$, as claimed.  Notice that a likelihoodist who rejects a fully Bayesian viewpoint could also argue that, in this special case, the choice between such non-informative stopping rules is irrelevant because, by the likelihood principle, outcomes obtained in experiments that differ only by such stopping rules are evidentially equivalent.\\

\noindent \textbf{Example 3:} Suppose you are interested whether the fraction $\theta$ of people who survive three months of a new cancer treatment exceeds the survival rate of the conventional treatment, which is known to be $94\%$.  Two experimental designs are considered:  $\mathbb{E}$, which consists in treating 100 patients, and $\mathbb{F}$, which consists in treating patients until two deaths are recorded.  In both experiments, it is possible to treat 100 patients and record precisely two deaths, the second of which is the 100th patient.  Let $\vec{x}$ be a binary sequence of length 100 representing such a data set, and note that $P_{\theta}^{\mathbb{E}}(\vec{x}) = P_{\theta}^{\mathbb{F}}(\vec{x}) = \theta^2 \cdot (1-\theta)^{98}$ no matter the survival rate $\theta$.  Here, the stopping rules of the respective experiments are given by:
$$\phi^{\mathbb{E}}(n|x_1, \ldots x_{n-1} ;\theta) = \left\{ 
\begin{array}{l}
     1 \mbox{ if }  n < 100 \\
     0 \mbox{ otherwise}
\end{array}
\right.
$$
and 
$$\phi^{\mathbb{F}}(n|x_1, \ldots x_{n-1} ;\theta) = \left\{ 
\begin{array}{l}
     1 \mbox{ if }  \sum_{k \leq n-1} x_k < 2 \\
     0 \mbox{ otherwise}
\end{array}
\right.
$$
Thus, both stopping rules are data-driven.

By the arguments above, the choice between the two stopping rules are irrelevant from both a robust Bayesian and likelihoodist perspective.  However, a quick calculation shows that under the null hypothesis -- that the new drug is no more effective than the conventional treatment --  the chance of two or fewer deaths in $\mathbb{E}$ is $.0566$, which is not significant at the $.05$ level.  In contrast, in $\mathbb{F}$, the chance of treating $100$ or more patients under the null is $.014$, which is significant at the $.05$ level.  Advocates of classical testing sometimes conclude, therefore, that the latter experiment provides better evidence against the null than the former.  This indicates a way in which \f{srp} is at odds with some classical testing procedures.

\subsection{The qualitative setting}
Our final major result is that, even in the qualitative setting, when two experiments differ only in a data driven stopping rules, every experimental outcome is posterior equivalent to itself, regardless of which experiment it is obtained in.

\begin{theorem}\label{thm:qDataDrivenEntailsPosteriorequivalence}
If two experiments $\mathbb{E}$ and $\mathbb{F}$  differ only in data-driven stopping rules and $\omega$ is a common outcome of both, then observing $\omega$ in $\mathbb{E}$ is posterior equivalent to observing  $\omega$ in $\mathbb{F}$.
\end{theorem}

To rigorously state and prove the theorem, we must define ``differs only in stopping rule''  and ``data-driven'' in our qualitative framework.

Assume $\Theta = \Theta_1 \times \Theta_2$, where $\Theta_1$ contains the parameter of interest and $\Theta_2$ is the set of parameters that influence when the experiment stops.  In this section, we consider only experiments $\mathbb{E}$ such that (i) the outcome space $\Omega^{\mathbb{E}}$ contains finite and countable sequences and (ii) the set containing the empty sequence and all countable sequences of  $\Omega^{\mathbb{E}}$ is $\sqsubseteq$-null. The first condition says outcomes of the experiment represent multiple trials, i.e., the outcome $\langle \omega_1, \omega_2 \ldots \rangle \in \Omega^{\mathbb{E}}$ represents an outcome in which one observes $\omega_1$ in one trial, $\omega_2$ in a second, and so on. The second condition says that at least one trial will occur and that the experiment is sure to stop after a finite number of trials.    We let $\Omega^{\mathbb{E}}_{>n}$ represent all outcome sequences of length greater than $n$. As always, when the experiment $\mathbb{E}$ is clear from context, we again drop the superscript $\mathbb{E}$ from $\Omega^{\mathbb{E}}$.


We let $\Sigma$ denote the set of all prefixes/initial-segments of experimental outcomes.  In other words,  $\sigma \in \Sigma$ if and only if there is some $n \in \mathbb{N}$ and some $\omega = \langle \omega_1, \omega_2, \ldots \rangle \in \Omega$ such that $\sigma= \langle \omega_1, \omega_2, \ldots \omega_n \rangle$.  Given $\sigma \in \Sigma$ and $n \in \mathbb{N}$, we define  $\sigma_{\leq n}$ to be the set of experimental outcomes that agree with $\sigma$ through the first $n$ trials.  In symbols:
\[ \sigma_{\leq n} := \{\omega \in \Omega: \omega_k = \sigma_k \mbox{ for all } k \leq n \}\]
Notice when $n=0$, then $\sigma_{\leq n} = \Omega$.

In probability theory, it is standard to use a comma rather than the intersection symbol in some settings, e.g., one typically writes $P(X=x, Y=y)$ not $P(\{X=x\} \cap \{Y=y\})$.  We use the same abbreviation below to save horizontal space.

\begin{definition}\label{defn:qDifferStoppingRule}
Experiments $\mathbb{E}$ and $\mathbb{F}$ \textbf{differ only in stopping rule} if for any $n \geq 0$ and  any $\sigma \in \Sigma^{\mathbb{E}} \cap \Sigma^{\mathbb{F}}$ of length $n+1$:
\[ \sigma_{\leq n+1} |^{\mathbb{E}} \sigma_{\leq n}, \Omega^{\mathbb{E}}_{>n}, \langle \theta_1, \theta_2 \rangle   \equiv  \sigma_{\leq n+1} |^{\mathbb{F}}  \sigma_{\leq n} , \Omega^{\mathbb{F}}_{>n}, \langle \theta_1, \theta_2 \rangle \mbox{ for all } \theta_1 \in \Theta_1 \mbox{ and } \theta_2 \in \Theta_2\]
\end{definition}

This condition says that the data-generating process of the two experiments is identical.  Why?  It says that, given that another sample point will certainly be drawn after $\langle \sigma_1, \ldots \sigma_n \rangle$ has been observed, then the probability that the next observation takes a specific value is unaffected by whether one is conducting $\mathbb{E}$ or $\mathbb{F}$.

Our goal is to characterize which stopping rules are ``uninformative'' with respect to the parameter of interest.   The qualitative analog of \autoref{defn:NonInformative} is

\begin{definition}\label{defn:qNonInformative}
We say an $\mathbb{E}$ has a $\Theta_1$- \textbf{non-informative} stopping rule if for all $n \geq 0$ and all $\theta_1, \theta_1' \in \Theta_1$  and  all  $\theta_2, \theta_2' \in \Theta_2$, we have:
\begin{enumerate}
    \item For all for all non-null $\sigma$ of length $n+1$:
      \[ \sigma_{\leq n+1} |^{\mathbb{E}} \sigma_{\leq n}, \Omega^{\mathbb{E}}_{> n}, \langle \theta_1, \theta_2' \rangle \equiv \sigma_{\leq n+1}|^{\mathbb{E}} \sigma_{\leq n}, \Omega^{\mathbb{E}}_{> n},\langle \theta_1, \theta_2 \rangle \]
      \item For all for all non-null $\sigma$ of length $n$
    \[ \Omega_{>n}|^{\mathbb{E}} \sigma_{\leq n},\langle \theta_1, \theta_2 \rangle \equiv \Omega_{>n}|^{\mathbb{E}} \sigma_{\leq n},\langle \theta_1', \theta_2 \rangle \]
\end{enumerate}
\end{definition}

The first condition says that the data generating process depends only on $\theta_1$, and the second condition says that the probability of drawing another sample point depends only on $\theta_2$. 

Again, we focus on a special subclass of non-informative stopping rules:

\begin{definition}\label{defn:qDataDrive}
We say an $\mathbb{E}$ has a $\Theta_1$- \textbf{data-driven} stopping rule if
\begin{enumerate}
    \item The stopping rule depends on neither parameter.  In other words, for all $\theta_1, \theta_1' \in \Theta_1$, all $\theta_2, \theta_2' \in \Theta_2$, and all sequences $\sigma \in \Sigma$:
    \[ \Omega_{>n} \ce \sigma_{\leq n},\langle \theta_1, \theta_2 \rangle \equiv \Omega_{>n} \ce \sigma_{\leq n},\langle \theta_1', \theta_2' \rangle \]
    \item  The chance of stopping is always zero or one .  In other words, for all $\theta_1, \theta_1' \in \Theta_1$, all $\theta_2, \theta_2' \in \Theta_2$, and all all sequences $\sigma \in \Sigma$, either
   $\Omega_{>n} \ce \sigma_{\leq n},\langle \theta_1, \theta_2 \rangle \equiv \Omega^{\mathbb{E}} \ce \theta$ for all $\theta$
    or
   $\Omega_{>n} \ce \sigma_{\leq n},\langle \theta_1, \theta_2 \rangle \equiv \emptyset \ce \theta$ for all $\theta$.
\end{enumerate}
\end{definition}

By \Ax{0} and \Ax{3}, Condition 2 in the definition of ``data-driven'' immediately implies that, for all joint orderings $\preceq$ over the entire space $\Delta$, either $\Omega_{>n} \ce \sigma_{\leq n},\langle \theta_1, \theta_2 \rangle \sim  \Delta^{\mathbb{E}} \ce \theta$ for all $\theta$ or $\Omega_{>n} \ce \sigma_{\leq n},\langle \theta_1, \theta_2 \rangle \sim \emptyset \ce \Delta$.

With these definitions, we can now prove \autoref{thm:qDataDrivenEntailsPosteriorequivalence}. We'll need the following lemma, which is an analog of the chain rule for numerical probabilities.

\begin{lemma}\label{lm:chainrule}
    Let $1 \le n_0 \le n$ be arbitrary. Suppose, for all $m \le n_0$, we have that $A_m |^{\mathbb{E}} \bigcap_{k > m} A_k \sim B_m |^{\mathbb{F}} \bigcap_{k > m} B_k$. Then we know that $\bigcap_{j \le n_0} A_j |^{\mathbb{E}} \bigcap_{k > n_0} A_k \sim \bigcap_{j \le n_0} B_j |^{\mathbb{F}} \bigcap_{k > n_0} B_k$.
\end{lemma}
We also need the following lemma, which allows us to apply \f{qmsp} more easily.

\begin{lemma}\label{lm:qSufficientStatisticExamples}
Suppose $\mathbb{E}$ is an experiment with a finite set of outcomes $\Omega$.  Let $S$ be an equivalence relation such that $S(\omega, \omega')$ holds precisely when $\omega | \theta \equiv \omega'|\theta$ for all $\theta$. 
Let $T:\Omega \rightarrow \Omega/S$ be the quotient map $\omega \mapsto [\omega]_S$ that takes each outcome to its $S$-equivalence class. Then $T$ is sufficient.
\end{lemma}

\begin{proof}[Proof of \autoref{thm:qDataDrivenEntailsPosteriorequivalence}]
    It will help to prove a slightly stronger theorem.  Specifically, we show the theorem holds if instead of quantifying over all common outcomes $\omega$ of $\mathbb{E}$ and $\mathbb{F}$, we quantify over all \emph{prefixes/initial-segments} $\sigma$ of common outcomes of  $\mathbb{E}$ and $\mathbb{F}$, i.e., we show $H \ce \sigma_{\le n} \sim H \cf \sigma_{\le n}$ for all $n \in \mathbb{N}$ and $\sigma_{\le n}$ such that there is $\omega \in \Omega^{\mathbb{E}} \times \Omega^{\mathbb{F}}$ such that $\sigma_{\le n}$ is the first $n$ trials of $\omega$. Let $\oe = \langle \mathbb{E}, \omega \rangle$ and $\of = \langle \mathbb{F}, \omega \rangle$.
    
    We first handle case in which $\oe$ is null: we show that if $\oe$ is null, then so is $\of$.
    Assume $\oe$ is null, and suppose $\omega = \sigma_{\le n+1}$ (i.e., $\omega$ is of length $n+1$). Then $\emptyset \equiv \sigma_{\le n+1} \ce \sigma_{\le n}, \Omega^{\mathbb{E}}_{> n}, \theta$. 
    We also know that $\sigma_{\le n+1} \cf \sigma_{\le n}, \Omega^{\mathbb{F}}_{> n}, \theta \equiv \sigma_{\le n+1} \ce \sigma_{\le n}, \Omega^{\mathbb{E}}_{> n}, \theta$ since the experiments differ only in stopping rule. 
    Thus, $\of$ is null. For the remainder of the proof, we assume both outcomes are non-null.
    
    The following is the key proposition for this proof. 
    \begin{proposition}
        For all $n$ and for all $\theta$, let $\sigma_{\le n}$ be an arbitrary non-null shared prefix. Then  $\sigma_{\le n} \ce \theta \equiv \sigma_{\le n} \cf \theta$.
    \end{proposition}
    \begin{proof}
        The proof is by induction on the length of $\sigma$.
        
        For the base case, $\sigma_{\le 0} = \Delta$. In this case, we just want to show that $\Delta \ce \theta \sim \Delta \cf \theta$ for all $\theta$. This which follows from \Ax{3} and \Ax{4}. 
        
        Now, assume the result for all prefixes $\sigma_{\le n}$ of length $n$; we prove it for an arbitrary prefix $\sigma_{\le n+1}$ of length $n+1$. Notice that
        \[\sigma_{\le n+1} = \sigma_{\le n+1} \cap \Omega_{>n} \cap \sigma_{\le n}.\]
        Thus $\sigma_{\le n+1}^{\mathbb{E}} \ce \theta \equiv \sigma_{\le n+1}, \Omega_{>n}^{\mathbb{E}}, \sigma_{\le n} \ce \theta$ (and similarly for $\mathbb{F}$). We now apply \autoref{lm:chainrule} with:
        \begin{alignat*}{3}
            A_1 &= B_1 = \sigma_{\le n+1}\\
            A_2 &= \Omega_{> n}^{\mathbb{E}} \qquad &B_2 &= \Omega_{> n}^{\mathbb{F}} \\
            A_3 &= B_3 = \sigma_{\le n} \\
            A_4 &= B_4 = \theta\\
            n_0 &= 3
        \end{alignat*}
        We now show that the conditions of the lemma hold.
        \begin{enumerate}
            \item $A_1 \ce A_2, A_3, A_4 \equiv B_1 \cf B_2, B_3, B_4$. This is because experiments $\mathbb{E}$ and $\mathbb{F}$ differ only in stopping rule.
            \item $A_2 \ce A_3, A_4 \equiv B_2 \cf B_3, B_4$. Here, we need to use the fact that both experiments have a data-driven stopping rule. By condition 2 of that definition, either $\Omega_{> n}^{\mathbb{E}} \ce \sigma_{\le n}, \theta \sim \Delta$ or $\emptyset$. If the latter is true, then $\sigma_{\le n+1} \in \nul$, as $\sigma_{\le n+1} \subseteq \Omega_{> n}^{\mathbb{E}}$. Thus, we know the former is true: $\Omega_{> n}^{\mathbb{E}} \ce \sigma_{\le n}, \theta \equiv \Delta$.  By similar reasoning,  $\Omega_{> n}^{\mathbb{F}} \cf \sigma_{\le n}, \theta \sim \Delta$.  We then get the desired conclusion from \Ax{3}, which says that $\Delta \ce \Delta \sim \Delta \cf \Delta$.
            \item $A_3 \ce A_4 \equiv B_3 \cf B_4$. This is the inductive hypothesis. 
        \end{enumerate}
        With the conditions satisfied, the lemma says that $\sigma_{\le n+1}, \Omega_{> n}^{\mathbb{E}}, \sigma_{\le n} \ce \theta \equiv \sigma_{\le n+1} \ce \theta \equiv \sigma_{\le n+1} \cf \theta$, as desired.
    \end{proof}
    The proposition tells us that $\oe \ce \theta \equiv \of \cf \theta$ for all shared outcomes $\omega$. Now, let $\mathbb{M}$ be a ``uniform'' mixture in which the experiments $\mathbb{E}$ and $\mathbb{F}$ are each conducted with the same, ``positive'' probability; i.e., any mixture in which $\Omega^{\mathbb{E}}|\theta \equiv \Omega^{\mathbb{F}}|\theta$ for all $\theta$ and $\Omega^{\mathbb{E}}|\theta \sqsupset \emptyset$ for all $\theta$.  Such a mixture exists by \autoref{assum:qMixturesExist}. Let $\ome = \langle \mathbb{M}, \langle \mathbb{E}, \omega \rangle \rangle$, and $\omf = \langle \mathbb{M}, \langle \mathbb{F}, \omega \rangle \rangle$.
    From \autoref{lem:mixtureComponentEquivalence}, we know that $\oe$ is posterior equivalent to $\ome$. 
    Similarly, $\of$ is posterior equivalent to $\omf$. So it remains to show that $\ome$ is posterior equivalent to $\omf$.
    
    First, we apply \Ax{6b} with:
    \begin{alignat*}{3}
        X &= \theta \qquad &X' &= \theta \\
        Y &= \Omega^{\mathbb{E}} \cap \theta \qquad &Y' &= \Omega^{\mathbb{F}} \cap \theta \\
        Z &= \ome \cap \Omega^{\mathbb{E}} \cap \theta \qquad &Z' &= \omf \cap \Omega^{\mathbb{F}} \cap \theta
    \end{alignat*}
    By the proposition above and the definition of a mixture, we know that $\ome \cm \theta, \Omega^{\mathbb{E}} \equiv \omf \cm \theta, \Omega^{\mathbb{F}}$, i.e., $Z|Y \equiv Z'|Y'$. Since $\mathbb{M}$ is a uniform mixture, we know that $\Omega^{\mathbb{E}}|\theta \equiv \Omega^{\mathbb{F}}|\theta$, i.e., $Y|X \equiv Y'|X'$. By \Ax{6b}, we get that $Z|X \equiv Z'|X'$, i.e., $\ome, \Omega^{\mathbb{E}} | \theta \equiv \omf, \Omega^{\mathbb{F}} | \theta$. Since $\ome \cap \Omega^{\mathbb{E}} = \ome$ (and similarly for $\mathbb{F}$), this simplifies to $\ome|\theta \equiv \omf|\theta$.
    
    Now, let $T$ be the statistic that maps $\omega$ to its $S$-equivalence class, i.e., 
    \[T(\omega) = [\omega] = \{\omega':   \omega' | \theta \equiv \omega | \theta \mbox{ for all } \theta \}\]
    
    The previous paragraph shows that $T(\ome) = T(\omf)$, and \autoref{lm:qSufficientStatisticExamples} says that $T$ is sufficient. Thus, by \autoref{thm:qmsp}, we know $\ome$ is posterior equivalent to $\omf$ in $\mathbb{M}$, as desired.
\end{proof}

We end with a conjecture, namely, that the
qualitative analog of \autoref{thm:SRPEntailsPosteriorEquivalence} is likewise true.  Recall, we are interested in characterizing which stopping rules are uninformative for decision-makers whose beliefs are representable by an ordering for which $\theta_1 \ind \theta_2$ for all $\theta_1 \in \Theta_1$ and $\theta_2 \in \Theta_2$.   Let 
${\cal B}(\ind)$ be the set of all such orderings. 


\begin{conjecture}\label{con:SRPEntailsPosteriorEquivalence}
Suppose $\mathbb{E}$ and $\mathbb{F}$ differ only in stopping rule and that both have $\Theta_1$ non-informative stopping rules.    Let $n \geq 0$ and $\sigma \in \Sigma^{\mathbb{E}} \cap \Sigma^{\mathbb{F}}$ be a non-null sequence of observations of length $n$ that can be obtained in both experiments (i.e., there are outcomes $\omega^{\mathbb{E}}$ and $\omega^{\mathbb{F}}$ of $\mathbb{E}$ and $\mathbb{F}$ respectively for which $\sigma$ is a common prefix).  For an $H \subseteq \Theta_1$, let $\overline{H} = H \times \Theta_2$.  Then for all  $H \subseteq \Theta_1$, and all orderings $\preceq$ in ${\cal B}(\ind)$, it follows that $\overline{H}|^{\mathbb{E}} \sigma_{\leq n} \sim \overline{H}|^{\mathbb{F}} \sigma_{\leq n}$.
\end{conjecture}

\section{Conclusions and Future Work}


We began this project with the goal of defending likelihoodism, by showing that likelihoodist theses like \f{ll}, \f{lp}, the sufficiency principle, etc. can be articulated easily in qualitative probabilistic frameworks.  However, we eventually realized that only \textit{some} likelihoodist principles had qualitative analogs, and we developed a new theory of evidence to support this discovery. Our theory, Robust Bayesianism (\f{rb}), is  intended to formalize the idea that evidence is what would convince a large group of scientists, with varying priors, to change their minds in similar ways. 
Our theory remains close to likelihoodism; we show that, in the quantitative setting, likelihoodist principles are equivalent to our Robust Bayesian principles. However, our theory is compatible with weaker assumptions about (i) rational belief and (ii) what types of likelihood judgments are widely shared by scientists.
In this broader setting, only some likelihoodist principles have analogs. Notably, $\f{ll}^+$ does not, which suggests that it may be an artifact of modeling likelihoods as numerical probabilities. 

\f{rb} provides a novel justification for the controversial stopping rule principle, which roughly says that that the evidential strength of some datum $\omega$ does not depend on the stopping rule of the experiment in which $\omega$ is observed. 
Importantly, our argument requires only very weak assumptions about rational belief and the precision of agreement on likelihood functions.  Further, our argument also does not depend upon \f{lp} or the assumption that unobserved outcomes lack evidential importance. Thus, classical statisticians cannot reject \f{srp} by simply rejecting \f{lp}, or by dismissing quantitative versions of Bayesianism, or by insisting that unobserved outcomes sometimes have evidential importance.

Although we hope that \f{rb} will eventually be of use to applied statisticians and empirical scientists, our work thusfar is important primarily to theorists.  Our results do, however, provide a useful litmus test for helping resolve debates about which statistical principles should guide data analysis; if a statistical principle cannot be articulated in our qualitative setting, then one should be suspicious of applying the principle in settings in which (i) experimenters' priors cannot be precisely elicited (and may be probabilistically ``incoherent''), or (ii) experimenters share only loose, comparative judgments about likelihood functions/orderings.

We conclude by highlighting three promising directions for future research; all three focus on applying our theory to practice.

The first two directions concern elicitation. Thus far, we have assumed that the relation $\sqsubseteq$ and the set of permissible orderings $\mathcal{B}$ are given. In practice, however, one will need to elicit these orderings from experts.  Elicitation of numerical probabilistic judgments is notoriously subtle \citep{ohagan_uncertain_2006}.
Nonetheless, we conjecture that, if a statistician's goal is merely to elicit a qualitative judgment of the form $A|B \preceq C|D$,  she may be able to avoid certain problems that belie elicitation in the quantitative setting.\footnote{An anonymous referee has suggested that, in cases of state dependent utility \citep{karni_decision_1985, schervish_state-dependent_1990}, eliciting \emph{conditional} judgments of the form $A|B \preceq C|D$ is no easier in the qualitative setting.  If that is right, then the framework we have introduced will be difficult to apply in practice unless  so-called `act-state-independence' holds.  However, we think the examples in Section 2.1 above motivate thinking that such cases are ubiquitous.} No matter how we elicit expert judgments, an expert's time is valuable; we want to elicit their ordering as efficiently as possible. These remarks motivate the following two research questions.
\begin{enumerate}
    \item How should we elicit an expert's judgment as to whether $A|B \preceq C|D$ or vice versa?
    \item Given an answer to 1., how many comparisons must we ask experts to make so that we can quickly determine the expert's full ordering? Is there an efficient adaptive algorithm that determines the best question to ask the expert next, given her previous answers?
\end{enumerate}

The elicitation questions raise a central computational difficulty: we currently lack an efficient algorithm for computing what a \emph{subset} of comparisons of the form $A|B \btleq C|D$ entail about the \emph{full} order $\btleq$, assuming our axioms hold. Even if we were presented data from a survey of experts indicating that (i) the experts agree upon some unique likelihood ordering $\sqsubseteq$, (ii) their priors are representable by a very small set of orderings, and (iii) they all agree that $E$ constitutes their full data, we would have no efficient way of enumerating all posterior orders compatible with that survey data, except to use a brute force algorithm. The difficulty is serious because the number of all orderings grows superexponentially in the size of $\Theta$ and $\Omega$, and this could hold even after pruning orderings that are impermissible or incompatible with our axioms.

The quantitative setting suggests a potential strategy for dealing with this computational issue.  In the quantitative setting, it is easy to compute which hypotheses are $\mathcal{U}$-supported over others (i.e., supported no matter the prior)  because one only needs to compare the likelihood ratios (by \autoref{clm:BSupportCharacterized}).   Thus, if the  $\mathcal{B}$-support relation is equivalent to $\mathcal{U}$-support relation for the set of priors $\mathcal{B}$ of interest, then one can resort to comparing likelihood ratios instead of computing all posteriors with respect to $\mathcal{B}$.  So, our third question is the following:



\begin{enumerate}
    \item[3.]  What properties must a set of prior probability functions $\mathcal{B}$ possess for $\mathcal{B}$-support to equal $\mathcal{U}$-support?  And the analogous question in the qualitative setting is, what properties must a set of \emph{orderings} $\mathcal{B}$ possess for $\mathcal{B}$-support to equal $\mathcal{U}$-support?   When such conditions obtain, one can at least determine what hypotheses are qualitatively \emph{favored} by the data using \f{qll}.
\end{enumerate}

Positive answers to these research questions would let us directly compute Bayesian and/or qualitative support relations in practice, which could lead to the development of new statistical methods. 

\newgeometry{margin=1in}
\appendix
\setcounter{theorem}{0}
\setcounter{lemma}{0}
\setcounter{assumption}{0}
\setcounter{proposition}{0}
\setcounter{definition}{0}
\setcounter{claim}{0}

\section{Robust Bayesianism and Likelihoodism}

\subsection{LP is Equivalent to Bayesian posterior and support equivalence}

\begin{claim}\label{clm:bSupportPosteriorEquivalence}
    Let $\mathcal{B}$ be any set of priors. Then the following are equivalent:  
        \begin{enumerate}
            \item $E$ and $F$ are $\mathcal{B}$-Bayesian posterior equivalent, 
            \item $E$ and $F$ are $\mathcal{B}$-Bayesian support equivalent
        \end{enumerate}
\end{claim}
\begin{proof}
 Suppose that $E$ and $F$ are posterior equivalent.  We must show that they are support equivalent, i.e., that \[Q^{\mathbb{E}}(H_1 | E \cap (H_1 \cup H_2))= Q^{\mathbb{F}}(H_1 | F \cap (H_1 \cup H_2))\] 
 for all disjoint hypotheses $H_1, H_2 \subseteq \Theta$ and all $Q$ for which those conditional probabilities are well-defined.

    To do so, let $H_1$ and $H_2$ be disjoint hypotheses, and let $Q$ be a prior such that $Q^{\mathbb{E}}(\cdot | E \cap (H_1 \cup H_2))$ and $Q^{\mathbb{F}}(\cdot| F \cap (H_1 \cup H_2))$ are both well-defined.  It follows that $Q^{\mathbb{E}}(E)$ and $Q^{\mathbb{F}}(F)$ are both positive (since $Q^{\mathbb{E}}(E \cap (H_1 \cup H_2))$ and $Q^{\mathbb{F}}(F \cap (H_1 \cup H_2))$ are).  Since $E$ and $F$ are posterior equivalent, we can infer that  $Q(H_1|E) = Q(H_1|F)$ and  $Q(H_2|E) = Q(H_2|F)$.  Thus:
    \begin{equation}\label{eqn:bLPFPE3}
        (Q(H_1|E) \cdot Q(E)) \cdot (Q(H_2|F) \cdot Q(F)) =  (Q(H_1|F) \cdot Q(E)) \cdot (Q(H_2|E) \cdot Q(F))
    \end{equation}
    We obtain the desired result via the following sequences of arithmetic operations:
    \begin{eqnarray*}
    \mbox{ Equation \ref{eqn:bLPFPE3} holds } &\Leftrightarrow& \\
     Q(E \cap H_1) \cdot Q(F \cap H_2) &=&  Q(H_1 \cap F) \cdot Q(E \cap H_2) \\ &\Leftrightarrow& \\
     Q(E \cap H_1) \cdot Q(F \cap H_2)  + Q(E \cap H_1) \cdot Q(F \cap H_1) &=&   Q(F \cap H_1) \cdot Q(E \cap H_2)  + Q(E \cap H_1) \cdot Q(F \cap H_1)\\   
     &\Leftrightarrow& \\
      Q(E \cap H_1) \cdot  \left[ Q(F \cap H_1) + Q(F \cap H_2) \right] &=&    Q(F  \cap H_1)\left[ Q(E \cap H_1) + Q(E \cap H_2) \right]  \\ &\Leftrightarrow& \\
      Q(E \cap H_1)/ \left[ Q(E \cap H_1) + Q(E \cap H_2) \right]     &=&    Q(F  \cap H_1)/ \left[ Q(F \cap H_1) + Q(F \cap H_2) \right] \\
      &\Leftrightarrow& \\
      Q(H_1 | E \cap (H_1 \cup H_2) )     &=&     Q(H_1 | F \cap (H_1 \cup H_2) )  \mbox{ as } H_1 \cap H_2 = \emptyset\\
    \end{eqnarray*}
    
Next we prove the converse. Suppose $E$ and $F$ are Bayesian support equivalent.  We must show that $E$ and $F$ are Bayesian posterior equivalent, i.e., that (1) $Q(\cdot | E)$ is well-defined if and only if $Q(\cdot | F)$ is and (2)  $Q(H|E)=Q(H|F)$ for all hypotheses $H$ and all $Q$ for which those conditional probabilities are well-defined.
    
    To show (1), pick any $H_1$ and $H_2$ such that $H_1 \cup H_2 = \Theta$.  If $Q(\cdot | E)$ is well-defined, then $Q(E) > 0$, and note that $Q(E) = Q(E \cap (H_1 \cup H_2))$ because $H_1 \cup H_2 = \Theta$.  Since $E$ and $F$ are Bayesian support equivalent, it follows that $Q(F) = Q(F \cap (H_1 \cup H_2)) \geq  Q(E \cap (H_1 \cup H_2)) > 0$, and hence, $Q(\cdot|F)$ is all well-defined.   Thus, we've shown that if $Q(\cdot | E)$ is well-defined, so is $Q(\cdot|F)$.  The converse is proven in the exact same manner.
    
   To show (2), let $Q$ be any prior for which $Q(\cdot|E)$ and $Q(\cdot|F)$ are well-defined.
    Let $H \subseteq \Theta$ by some arbitrary hypothesis.  We must show $Q(H|E)=Q(H|F)$.
    Then define $H_1 = H$ and $H_2 = \Theta \setminus H$.  Since $E$ and $F$ are Bayesian support equivalent, we know that 
    $$Q(H_1 | E \cap (H_1 \cup H_2))= Q(H_1 | F \cap (H_1 \cup H_2))$$ 
    from which it follows that
    $$Q(H|E) = Q(H_1 |E) = Q(H_1 | E \cap (H_1 \cup H_2))= Q(H_1 | F \cap (H_1 \cup H_2)) = Q(H|F)$$
    as desired.
\end{proof}

\subsection{$\f{LL}^+$ is Equivalent to Bayesian Support}

The following claim links Bayesian support to $\f{ll}^+$ by showing that $E \bafd F$ exactly when $\f{ll}^+$ says that $E$ provides stronger evidence for $\theta_1$ over $\theta_2$ than $F$ does.

\begin{claim}\label{clm:BSupportCharacterized}
Suppose $H_1$ and $H_2$ are finite and disjoint.  Let $E$ and $F$ be outcomes of experiments $\mathbb{E}$ and $\mathbb{F}$ respectively, and let $\mathcal{B} = \mathcal{U}$ be the universal set of priors. Then  $E \baf{{\cal B}}{H_1}{}{H_2}{} F$ if and only if (1) for all $\theta \in H_1 \cup H_2$, if $P^{\mathbb{F}}_{\theta}(F) > 0$, then $P^{\mathbb{E}}_{\theta}(E) > 0$, and (2) for all $\theta_1 \in H_1$ and $\theta_2 \in H_2$:
\begin{equation}
    P^{\mathbb{E}}_{\theta_1}(E) \cdot  P^{\mathbb{F}}_{\theta_2}(F) \geq P^{\mathbb{E}}_{\theta_2}(E) \cdot P^{\mathbb{F}}_{\theta_1}(F)
\end{equation}
Similarly, $E \sbaf{{\cal B}}{H_1}{}{H_2}{} F$ if and only if the inequality in the above equation is always strict.
\end{claim}
\begin{proof}
    In both directions, we will need the following basic arithmetic fact.\\  
    
    \noindent \textbf{Observation 1:}  Suppose $\langle a_j \rangle_{j \leq n} ,\langle b_k \rangle_{k \leq m}, \langle c_j \rangle_{j \leq n},$ and $\langle d_k \rangle_{k \leq m}$ are sequences of non-negative real numbers.  Then $a_j \cdot d_k \geq b_k  \cdot c_j$ for all $j \leq n$ and $k \leq m$ if and only if
    \begin{equation}\label{eqn:bcfc2}
    \frac{\sum_{j \leq n}r_j \cdot a_j}{\sum_{j \leq n}r_j \cdot a_j + \sum_{k \leq m}s_k \cdot b_k} \geq \frac{\sum_{j \leq n}r_j \cdot c_j}{\sum_{j \leq n}r_j \cdot c_j + \sum_{k \leq m}s_k \cdot d_k} 
    \end{equation}
    for all sequences $\langle r_j \rangle_{j \leq n}$ and $\langle s_k \rangle_{k \leq m}$ of non-negative real numbers such that $\sum_{j \leq n}r_j  a_j + \sum_{k \leq m}s_k  b_k$ and $\sum_{j \leq n}r_j  c_j + \sum_{k \leq m}s_k d_k$ are both positive.\\

    \noindent \textbf{Proof:} First, note that if $\langle r_j \rangle_{j \leq n}$ and $\langle s_k \rangle_{k \leq m}$ are sequences of non-negative real numbers such that $\sum_{j \leq n}r_j a_j + \sum_{k \leq m}s_k  b_k$ and $\sum_{j \leq n}r_j  c_j + \sum_{k \leq m}s_k d_k$ are positive, then both sides of \autoref{eqn:bcfc2} are well-defined and
    \begin{eqnarray*}
    \mbox{\autoref{eqn:bcfc2} holds} &\Leftrightarrow& \left(\sum_{j \leq n}r_j a_j \right) \cdot \left(\sum_{j \leq n}r_j c_j + \sum_{k \leq m}s_k d_k \right) \geq   \left( \sum_{j \leq n}r_j c_j \right) \cdot \left(\sum_{j \leq n}r_j a_j + \sum_{k \leq m}s_k b_k \right) \\
        ~ &~&\mbox{ Cross multiplying} \\
        &\Leftrightarrow& \left(\sum_{j \leq n}r_j a_j \right) \cdot \left(\sum_{k \leq m}s_k d_k \right) \geq   \left( \sum_{j \leq n}r_j c_j \right) \cdot \left(\sum_{k \leq m}s_k b_k \right) \\
        ~ &~&\mbox{ Cancelling common terms} \\
         &\Leftrightarrow& \left(\sum_{j \leq n, k \leq m }r_j a_j s_k d_k \right) \geq   \left( \sum_{j \leq n, k \leq m}r_j c_j s_k b_k  \right)  [\mbox{Condition } \dag] \\
        ~ &~&\mbox{ Rules of double sums}
    \end{eqnarray*}
Now we prove the right-to-left direction.  Assume  \autoref{eqn:bcfc2} holds \emph{for all} sequences of $r_j$s and $s_k$s for which the sums specified above are positive.  Pick arbitrary $j_0 \leq n$ and $k_0 \leq m$.  Our goal is to show $a_{j_0}d_{k_0} \geq c_{j_0}b_{k_0}$.   Consider the specific sequence of $r_j$'s and $s_k$'s such that $r_{j_0} = s_{k_0} = 1$ and $r_j=s_k=0$ for all $j \neq j_0$ and $k \neq k_0$.    There are two cases two consider.

If $\sum_{j \leq n}r_j a_j + \sum_{k \leq m}s_k  b_k$ and $\sum_{j \leq n}r_j  c_j + \sum_{k \leq m}s_k d_k$ are both positive, then  Condition $\dag$ holds and reduces to the claim $a_{j_0} d_{k_0} \geq c_{j_0} b_{k_0}$, as desired.

If $\sum_{j \leq n}r_j a_j + \sum_{k \leq m}s_k  b_k = a_{j_0} + b_{k_0}$ and $\sum_{j \leq n}r_j  c_j + \sum_{k \leq m}s_k d_k = c_{j_0} + d_{k_0}$ are not both positive, then because all the $a$'s, $b$'s, etc. are non-negative, it follows that either (i)  $a_{j_0} = b_{k_0} = 0$  or  (i)  $c_{j_0} = d_{k_0} = 0$.  In either case $a_{j_0}d_{k_0} = 0 \geq 0 = c_{j_0}b_{k_0}$.

The left-to-right direction merely reverses the direction of the above calculations.   In greater detail, assume $a_j \cdot d_k \geq b_k  \cdot c_j$ for all $j \leq n$ and $k \leq m$.  Let $\langle r_j \rangle_{j \leq n}$ and $\langle s_k \rangle_{k \leq m}$ be sequences of non-negative real numbers such that $\sum_{j \leq n}r_j  a_j + \sum_{k \leq m}s_k  b_k$ and $\sum_{j \leq n}r_j  c_j + \sum_{k \leq m}s_k d_k$ are both positive.  Our goal is to show \autoref{eqn:bcfc2}  holds.  By the reasoning above, it suffices to show Condition $\dag$ holds.  Note that because $a_j \cdot d_k \geq b_k  \cdot c_j$ for all $j \leq n$ and $k \leq m$ by assumption, every term in the sum on the left-hand-side of Condition $\dag$ is greater than the corresponding term on the right hand side.  So $\dag$ holds as desired.
    \begin{flushright}
        (End Proof of Observation).
    \end{flushright}

The central claim follows from Observation 1 by appropriate substitutions.  Specifically, enumerate $H_1 = \{\theta_{1,1}, \theta_{1,2}, \ldots \theta_{1,n} \}$ and $H_2 = \{\theta_{2,1}, \theta_{2,2}, \ldots \theta_{2,m} \}$.  To apply the observation, let $a_j = P^{\mathbb{E}}_{\theta_{1,j}}(E)$, $c_j = P^{\mathbb{F}}_{\theta_{1,j}}(F)$, $b_k = P^{\mathbb{E}}_{\theta_{2,k}}(E)$, and $d_k = P^{\mathbb{F}}_{\theta_{2,k}}(F)$.  Those terms are clearly all non-negative.  So Observation 1 entails that
\begin{equation}\label{eqn:bcfc3}
P^{\mathbb{E}}_{\theta_{1,j}}(E) \cdot  P^{\mathbb{F}}_{\theta_{2,k}}(F)  = a_j d_k \geq 
   c_j b_k
    \geq P^{\mathbb{E}}_{\theta_{2,k}}(E) \cdot P^{\mathbb{F}}_{\theta_{1,j}}(F) \mbox{ for all } j, k
\end{equation}
if and only if
 \begin{equation}\label{eqn:bcfc4}
    \frac{\sum_{j \leq n}r_j \cdot P^{\mathbb{E}}_{\theta_{1,j}}(E)}{\sum_{j \leq n}r_j \cdot P^{\mathbb{E}}_{\theta_{1,j}}(E) + \sum_{k \leq m} s_k \cdot P^{\mathbb{E}}_{\theta_{2,k}}(E)} \geq \frac{\sum_{j \leq n}r_j \cdot P^{\mathbb{F}}_{\theta_{1,j}}(F)}{\sum_{j \leq n}r_j \cdot P^{\mathbb{F}}_{\theta_{1,j}}(F) + \sum_{k \leq m}s_k \cdot P^{\mathbb{F}}_{\theta_{2,k}}(F)} 
    \end{equation}
for all sequences $\langle r_j \rangle_{j \leq n}$ and $\langle s_k \rangle_{k \leq m}$ of non-negative real numbers such that $\sum_{j \leq n}r_j  \cdot P^{\mathbb{E}}_{\theta_{1,j}}(E) + \sum_{k \leq m}s_k \cdot P^{\mathbb{E}}_{\theta_{2,k}}(E)$ and $\sum_{j \leq n}r_j \cdot P^{\mathbb{F}}_{\theta_{1,j}}(F) + \sum_{k \leq m}s_k \cdot P^{\mathbb{F}}_{\theta_{2,k}}(F)$ are both positive.\\

Now, we prove the right-to-left direction of \autoref{clm:BSupportCharacterized}.  Suppose (1)  for all $\theta \in H_1 \cup H_2$, if $P^{\mathbb{F}}_{\theta}(F) > 0$, then $P^{\mathbb{E}}_{\theta}(E) > 0$, and (2) \autoref{eqn:bcfc3} holds.  Let $Q$ be some prior such that $Q^{\mathbb{F}}(\cdot|F \cap (H_1 \cup H_2))$ is well-defined.  We claim $Q^{\mathbb{E}}(\cdot|E \cap (H_1 \cup H_2))$ is also well-defined.   Why? Since  $Q^{\mathbb{F}}(\cdot|F \cap (H_1 \cup H_2))$, we have by definition that:
\begin{equation}\label{eqn:bcfc5}
      Q^{\mathbb{F}}(H_1|F \cap (H_1 \cup H_2))  = \frac{\sum_{j \leq n} Q(\theta_{1,j}) \cdot   P_{\theta_{1,j}}(F)}{\sum_{j \leq n} Q(\theta_{1,j}) \cdot P^{\mathbb{F}}_{\theta_{1,j}}(F) + \sum_{k \leq m} Q(\theta_{2,k}) \cdot  P^{\mathbb{F}}_{\theta_{2,k}}(F)}
\end{equation}
That fraction is well-defined only if the denominator is positive,  i.e.,
\begin{equation}\label{eqn:bcfc6}
       \sum_{j \leq n} Q(\theta_{1,j}) \cdot P^{\mathbb{F}}_{\theta_{1,j}}(F) + \sum_{k \leq m} Q(\theta_{2,k}) \cdot  P^{\mathbb{F}}_{\theta_{2,k}}(F) > 0
\end{equation}
Therefore, there is some $\theta \in H_1 \cup H_2$ such that $Q(\theta) \cdot P^{\mathbb{F}}_{\theta}(F) >0$.    By assumption (1), it follows that $Q(\theta) \cdot P^{\mathbb{E}}_{\theta}(E) >0$.  Hence:
\begin{equation}\label{eqn:bcfc7}
      \sum_{j \leq n} Q(\theta_{1,j}) \cdot P^{\mathbb{E}}_{\theta_{1,j}}(E) + \sum_{k \leq m} Q(\theta_{2,k}) \cdot  P^{\mathbb{E}}_{\theta_{2,k}}(E) > 0
\end{equation}
from which it follows that  $Q^{\mathbb{E}}(\cdot|E \cap (H_1 \cup H_2))$ is also well-defined.  It thus remains to be shown that  $Q^{\mathbb{E}}(H_1|E \cap (H_1 \cup H_2)) \geq  Q^{\mathbb{F}}(H_1|F \cap (H_1 \cup H_2))$.

Letting $r_j = Q(\theta_{1,j})$ and $s_k = Q(\theta_{2,k})$, the inequalities in  \autoref{eqn:bcfc6} and \autoref{eqn:bcfc7} entail that $\sum_{j \leq n}r_j  a_j + \sum_{k \leq m}s_k  b_k$ and $\sum_{j \leq n}r_j  c_j + \sum_{k \leq m}s_k d_k$ are both positive.  Since we've assumed (2) that \autoref{eqn:bcfc3} holds, the left-to-right direction of Observation 1 entails that \autoref{eqn:bcfc4} must hold as desired.

In the left-to-right direction, suppose that  $Q^{\mathbb{E}}(H_1|E \cap (H_1 \cup H_2)) \geq Q^{\mathbb{F}}(H_1|F \cap (H_1 \cup H_2))$ for all priors $Q$ for which $Q^{\mathbb{F}}(H_1|F \cap (H_1 \cup H_2))$ is well-defined.  We must show that  (1)  for all $\theta \in H_1 \cup H_2$, if $P^{\mathbb{F}}_{\theta}(F) > 0$, then $P^{\mathbb{E}}_{\theta}(E) > 0$, and (2) \autoref{eqn:bcfc3} holds.

To show (1), suppose $P^{\mathbb{F}}_{\upsilon}(F) > 0$ for some $\upsilon \in H_1 \cup H_2$.  Define $Q$ so that $Q(\upsilon) = 1$.  Then 
\begin{equation}\label{eqn:bcfc8}
    Q^{\mathbb{F}}(F \cap (H_1 \cup H_2)) = \sum_{\theta \in H_1 \cup H_2} Q(\theta) \cdot  P^{\mathbb{F}}_{\theta}(F) = Q(\upsilon) \cdot  P^{\mathbb{F}}_{\upsilon}(F) = P^{\mathbb{F}}_{\upsilon}(F) > 0
\end{equation}
and so $Q^{\mathbb{F}}(\cdot |F \cap (H_1 \cup H_2))$ is well-defined.  Thus, $Q^{\mathbb{E}}(H_1|E \cap (H_1 \cup H_2)) \geq Q^{\mathbb{F}}(H_1|F \cap (H_1 \cup H_2))$ by our assumption.  Hence, $Q^{\mathbb{E}}(\cdot|E \cap (H_1 \cup H_2))$ is also well-defined, which entails $Q^{\mathbb{E}}(E \cap (H_1 \cup H_2)) > 0$.  But using the same reasoning that lead to \autoref{eqn:bcfc8}, we obtain
$P^{\mathbb{E}}_{\upsilon}(E) =Q^{\mathbb{E}}(E \cap (H_1 \cup H_2))$, and so $P^{\mathbb{E}}_{\upsilon}(E) > 0$ as desired.

Next, we must show that \autoref{eqn:bcfc3} holds.  Since $Q^{\mathbb{E}}(H_1|E \cap (H_1 \cup H_2)) \geq Q^{\mathbb{F}}(H_1|F \cap (H_1 \cup H_2))$,  we make the same substitutions as in the left-to-right direction and infer that \autoref{eqn:bcfc4} holds.  We then apply Observation 1 to obtain that  \autoref{eqn:bcfc3} holds, as desired.
\end{proof}

\begin{claim}\label{cor:BLLFavoring}  Let $\mathcal{B} = \mathcal{U}$ be the universal set of priors.  Then  $E$ $\mathcal{B}$-Bayesian favors $H_1$ to $H_2$ if and only if (1) $P_{\theta}(E) > 0$ for all $\theta \in H_1 \cup H_2$ and (2) $P_{\theta_1}(E) \geq P_{\theta_2}(E)$ for all $\theta_1 \in H_1$ and $\theta_2 \in H_2$.  If $H_1=\{\theta_1\}$ and $H_2=\{\theta_2\}$ are  simple, $E$  strictly $\mathcal{B}$-Bayesian favors $H_1$ to $H_2$ if and only if \f{ll} entails $E$ favors $H_1$ to $H_2$.
\end{claim}
\begin{proof}
    This follows immediately from \autoref{clm:BSupportCharacterized}, since the likelihood of $P_{\theta}(\Omega)=1$ for all $\theta$.
\end{proof}

\begin{claim}\label{clm:bLPPosteriorEquivalence}
    If \f{lp} entails $E$ and $F$ are evidentially equivalent, then $E$ and $F$ are $\mathcal{B}$-Bayesian posterior equivalent.  If $\mathcal{B}=\mathcal{U}$ is the universal set of priors, the converse holds as well.
\end{claim}
\begin{proof}
 Suppose there is a $c > 0$ such that $P^{\mathbb{E}}_{\theta}(E) = c \cdot P^{\mathbb{F}}_{\theta}(F)$.  We must show $E$ and $F$ are posterior equivalent.  To do so, we must first show $Q^{\mathbb{E}}(\cdot |E)$ is well-defined if and only if $Q^{\mathbb{F}}(\cdot |F)$ is.  To do that, we must show $Q^{\mathbb{E}}(E)$ is positive if and only if $Q^{\mathbb{F}}(F)$  is.
    
   By definition of $Q^{\mathbb{E}}$, we have
    \begin{equation}\label{eqn:blPFPE1}
    Q^{\mathbb{E}}(E) = \sum_{\theta \in \Theta} Q(\theta) \cdot P^{\mathbb{E}}_{\theta}(E).
    \end{equation}
    Similarly,
   \begin{equation}\label{eqn:blPFPE2}
   Q^{\mathbb{F}}(F) =  \sum_{\theta \in \Theta} Q(\theta) \cdot P^{\mathbb{F}}_{\theta}(F)
    \end{equation}
    Since $P^{\mathbb{E}}_{\theta}(E) = c \cdot P^{\mathbb{F}}_{\theta}(F)$ for some $c>0$, each term in the sum in \autoref{eqn:blPFPE1} is positive if and only if the corresponding term in the sum \autoref{eqn:blPFPE2} is positive.
    Hence, $Q^{\mathbb{E}}(E) > 0$ if and only if $Q^{\mathbb{F}}(F) > 0$, as desired.
    
   The second thing we must show is that $Q^{\mathbb{E}}(H|E) = Q^{\mathbb{F}}(H|F)$ for all $H$, whenever those conditional probabilities are well-defined.  To that end, let $H \subseteq \Theta$ be any (simple or composite) hypothesis.  Then:
    
     \begin{alignat*}{3}
     		Q^{\mathbb{E}}(H| E) &= \frac{Q^{\mathbb{E}}(E \cap H) }{Q^{\mathbb{E}}(E)} && \mbox{definition of conditional probability} \\
     					&= \frac{ \sum_{\theta \in H \cap \Theta} Q^{\mathbb{E}}(E \cap \{\theta\}) }{\sum_{\theta \in \Theta} Q^{\mathbb{E}}(E \cap \{\theta\})} && \mbox{law of total probability} \\
     					&= \frac{ \sum_{\theta \in H \cap \Theta} P^{\mathbb{E}}_{\theta}(E)  \cdot Q(\theta)}{ \sum_{\theta \in \Theta} P^{\mathbb{E}}_{\theta}(E) \cdot Q(\theta)}  && \mbox{ because } Q(\cdot| \theta) = P^{\mathbb{E}}_{\theta}(\cdot)  \mbox{ for all } \theta \in \Theta \mbox{ and } Q(\theta) \mbox{ does not depend upon } \mathbb{E} \\
     					&= \frac{ \sum_{\theta \in H \cap \Theta} c \cdot P^{\mathbb{F}}_{\theta}(F)  \cdot Q(\theta)}{ \sum_{\theta \in \Theta} c \cdot P^{\mathbb{F}}_{\theta}(F) \cdot Q(\theta)}  &&  \mbox{ as } P^{\mathbb{E}}_{\theta}(E) = c \cdot P^{\mathbb{F}}_{\theta}(F) \\
     					&= \frac{ \sum_{\theta \in H \cap \Theta}  P^{\mathbb{F}}_{\theta}(F)  \cdot Q(\theta)}{ \sum_{\theta \in \Theta} P^{\mathbb{F}}_{\theta}(F) \cdot Q(\theta) } &&  \mbox{ canceling } c  \\
     					&= Q^{\mathbb{F}}(H| F)  && \mbox{reversing the above steps}
     \end{alignat*}
    
Now we prove the converse.  Assume $E$ and $F$ are posterior equivalent.   That is, assume that, for all $Q$, we have (i) $Q^{\mathbb{E}}(\cdot|E) $ is well-defined if and only if $Q^{\mathbb{F}}(\cdot|F)$ is and (ii)
   $Q^{\mathbb{E}}(H|E) = Q^{\mathbb{F}}(H|F)$ for all non-empty hypotheses $H \subseteq \Theta$ and all $Q$ for which those conditional probabilities are well-defined.  We must find some $c > 0$ such that $P^{\mathbb{E}}_{\theta}(E)= c \cdot P^{\mathbb{F}}_{\theta}(F)$ for all $\theta$.

    We first assume there is some $\theta \in \Theta$  such that $P^{\mathbb{E}}_{\theta}(E) > 0$ or $P^{\mathbb{F}}_{\theta}(F) > 0$.  Otherwise, $P^{\mathbb{E}}_{\theta}(E)=P^{\mathbb{F}}_{\theta}(F) = 0$ for all $\theta \in \Theta$, and so $P^{\mathbb{E}}_{\theta}(E)= c \cdot P^{\mathbb{F}}_{\theta}(F)$ for all $c > 0$ trivially.

    So pick any $\upsilon \in \Theta$ such that either  $P^{\mathbb{E}}_{\upsilon}(E) > 0$ or $P^{\mathbb{F}}_{\upsilon}(F) > 0$.  Without loss of generality, assume that
    $P^{\mathbb{F}}_{\upsilon}(F) > 0$.  We claim $P^{\mathbb{E}}_{\upsilon}(E) > 0$ as well.  Why? Let $Q$ be the prior such that $Q(\upsilon) = 1$.  Note $Q \in {\cal B}={\cal U}$ the universal set of priors.  Then $Q^{\mathbb{F}}(F) = Q(\upsilon) \cdot P^{\mathbb{F}}_{\upsilon}(F)  = P^{\mathbb{F}}_{\upsilon}(F)  > 0$.  Hence, $Q^{\mathbb{F}}(\cdot|F)$ is well-defined.  Since $E$ and $F$ are posterior equivalent, $Q^{\mathbb{E}}(\cdot|E)$ is also well-defined.  Because $Q^{\mathbb{F}}(E) = Q(\upsilon) \cdot P^{\mathbb{E}}_{\upsilon}(E) = P^{\mathbb{E}}_{\upsilon}(E)$, it follows that $P^{\mathbb{E}}_{\upsilon}(E) > 0$ as desired.
    
    Thus, if we let $c = P^{\mathbb{E}}_{\upsilon}(E)/P^{\mathbb{F}}_{\upsilon}(F)$, it follows that $c > 0$.  We want to show that  $P^{\mathbb{E}}_{\theta}(E)= c \cdot P^{\mathbb{F}}_{\theta}(F) $ for all $\theta$.  It suffices, therefore, to show $P^{\mathbb{E}}_{\upsilon}(E)/P^{\mathbb{F}}_{\upsilon}(F) = P^{\mathbb{E}}_{\theta}(E)/P^{\mathbb{F}}_{\theta}(F)$ for all $\theta$.
    
    To do so, let $\theta$ be arbitrary and 
    now define $Q$ to be any prior such that $Q(\upsilon)$ and $Q(\theta)$ are both positive. 
    Again, it's easy to check that $Q^{\mathbb{E}}(\cdot|E)$ and $Q^{\mathbb{F}}(\cdot|F)$ are well-defined.  
    So Bayes rule entails:
    $$ \frac{Q^{\mathbb{E}}(\theta|E)}{Q^{\mathbb{E}}(\upsilon|E)} = \frac{P^{\mathbb{E}}_{\theta}(E)}{P^{\mathbb{E}}_{\upsilon}(E)} \cdot \frac{Q(\theta)}{Q(\upsilon)}$$
    And similarly for $F$.  By assumption, $Q^{\mathbb{E}}(\{\theta\}|E) = Q^{\mathbb{E}}(\{\theta\}|F)$, and hence:
    $$ \frac{P^{\mathbb{E}}_{\theta}(E)}{P^{\mathbb{E}}_{\upsilon}(E)} \cdot \frac{Q(\theta)}{Q(\upsilon)} = \frac{P^{\mathbb{F}}_{\theta}(F)}{P^{\mathbb{F}}_{\upsilon}(F)} \cdot \frac{Q(\theta)}{Q(\upsilon)}$$
    It immediately follows that $P^{\mathbb{E}}_{\upsilon}(E)/P^{\mathbb{F}}_{\upsilon}(F) = P^{\mathbb{E}}_{\theta}(E)/P^{\mathbb{F}}_{\theta}(F)$ as desired.
    
\end{proof}

\section{Technical Lemmas}

We begin with the analog of the claim that, if $P(C)>0$, then $P(A|C) \leq P(B|C)$ if and only if $P(A \cap C) \leq P(B \cap C)$.  For brevity, we write $A$ instead of $A|\Delta$ when the conditioning event is the sure event $\Delta$.

\setcounter{lemma}{9}

\begin{lemma}\label{lm:qConditionalIntersection} 
If $A| C \preceq B|C$, then $A \cap C \preceq B \cap C$.  The converse holds if $C \not \in \nul$.
\end{lemma}
\begin{proof}
    To prove the left to right direction, we apply \Ax{6b}.  Let $\mathfrak{A}=\mathfrak{A}'=\Delta$; $\mathfrak{B}=\mathfrak{B}'=C$; $\mathfrak{C}=A \cap C$ and $\mathfrak{C}' = B \cap C$.  Notice that   $\mathfrak{C} \subseteq \mathfrak{B} \subseteq \mathfrak{A}$ and $\mathfrak{C}' \subseteq \mathfrak{B}' \subseteq \mathfrak{A}'$.  Because $\mathfrak{A}=\mathfrak{A}'$ and $\mathfrak{B}=\mathfrak{B}'$, it immediately follows from the reflexivity of $\preceq$ that $\mathfrak{B}|\mathfrak{A} \preceq \mathfrak{B}'|\mathfrak{A}'$.  Further:
    \begin{align*}
            \mathfrak{C}|\mathfrak{B} &= A \cap C | C &&\mbox{by definition of } \mathfrak{C} \& \mathfrak{B} \\
            &\sim A | C &&\mbox{by \Ax{4}} \\
            &\preceq B|C &&\mbox{by assumption} \\
            &\sim B \cap C | C &&\mbox{by \Ax{4}} \\
            &= \mathfrak{C}' | \mathfrak{B}' &&\mbox{by definition of } \mathfrak{C}' \mbox{ and } \mathfrak{B}'
    \end{align*}
    So by \Ax{6b}, it follows that $\mathfrak{C}|\mathfrak{A} \preceq \mathfrak{C}'|\mathfrak{A}'$, i.e., that $A\cap C| \Delta \preceq B \cap C | \Delta$, i.e., that $A \cap C \preceq B \cap C$, as desired.

    In the right to left direction, we prove the contrapositive.  That is, suppose that $A| C \not \preceq B | C$.  Because $C \not \in \nul$ and the relation $\preceq$ is total (\Ax{1}), it follows that $B|C \prec A | C$. One can then apply \Ax{6b} in the same way we just did to show that $B \cap C \prec A \cap C$ (since $B|C \prec A | C$, we get the strict conclusion).  By the definition of the $\prec$, it follows that $A \cap C \not \preceq B \cap C$, as desired.
\end{proof}
 
The set $\nul$ in \Ax{2} is the analog of probability zero events.  For technical purposes, it will be useful to define $\nul(C) = \{A:  A|C \preceq \emptyset|C\}$ to be the set of events that are probability zero \emph{conditional} on some non-null event $C \not \in \nul$, i.e., $\nul(C)$ is the analog of the set of events $A$ such that $P(A|C)=0$.  Thus, $\nul = \nul(\Delta)$.

\begin{lemma}\label{lm:qConditionalNull}
If $C \not \in \nul$, then $A \in \nul(C)$ if and only if $A \cap C \in \nul$.
\end{lemma}
\begin{proof}
This follows directly from \autoref{lm:qConditionalIntersection}.  In detail, in the left-to-right direction, suppose $A \in \nul(C)$.  By definition of $\nul(C)$, that means $A|C \preceq \emptyset|C$.  By \autoref{lm:qConditionalIntersection}, it follows that $A \cap C| \Delta \preceq \emptyset|\Delta$, and so $A \cap C \in \nul$ by \Ax{2}.

In the right to left direction, suppose  $A \cap C \in \nul$.  Then $A \cap C| \Delta \preceq \emptyset|\Delta$ by \Ax{2}.  Since $C \not \in \nul$, \autoref{lm:qConditionalIntersection} entaails $A|C \preceq \emptyset|C$ as desired.
\end{proof}

The following results are either exact copies of lemmata in \cite[p. 229]{krantz_foundations_2006} or very slight generalizations thereof; some of Krantz et. al.'s results pertain only to unconditional null sets.  When the results are copied from  \cite[p. 229]{krantz_foundations_2006}, we omit a proof.  When we generalize the original lemma to conditional null sets, we provide a proof but our arguments are more-or-less identical to those of the results for unconditional null sets.

\begin{lemma}\label{lm:krantz6-1}
    $A|B \cap C \succeq \emptyset|\Delta$ whenever $B \not \in \nul(C)$.  It follows that $A|B \succeq \emptyset|\Delta$ whenever $B \not \in \nul$.
\end{lemma}
\begin{proof}
    Suppose $B \not \in \nul(C)$.  Assume for the sake of contradiction that  $A|B \cap C \not \succeq \emptyset|\Delta$.  Since $B \not \in \nul(C)$, by \Ax{1} (specifically, the totality of $\preceq$), it follows that 
        \begin{equation}\label{eqn:krantz6-1_1}
            A|B \cap C \prec \emptyset|\Delta
        \end{equation}
     Further, by \Ax{3}, we have that
    \begin{equation}\label{eqn:krantz6-1_2}
            ((B \cap C) \setminus A )|B \cap C \preceq \Delta|\Delta.
    \end{equation}
    Applying \Ax{5} to \autoref{eqn:krantz6-1_1} and \autoref{eqn:krantz6-1_2} yields
   $A \cup ((B \cap C) \setminus A)|B \cap C \prec \emptyset \cup \Delta | \Delta$.
   Simplicfying that inequality, yields
   $B \cap C|B \cap C \prec \Delta | \Delta$, which contradicts \Ax{3}.
\end{proof}

\begin{lemma}\label{lm:krantz6-2} 
    $\emptyset|A \sim \emptyset|B$ for all $A,B \not \in \nul$.
\end{lemma}

\begin{lemma}\label{lm:krantz7}
    If $A \supseteq B$ and $C \not \in \nul$, then $A|C \succeq B|C$.
\end{lemma}

\begin{lemma}\label{lm:krantz8}
   For all $C \not \in \nul$:
   \begin{enumerate}
       \item $\emptyset \in \nul(C)$ and $C \not \in \nul(C)$.
       \item If $A \in \nul(C)$, then $C \setminus A \not \in \nul(C)$.  It follows that If $A \in \nul$, then $\Delta \setminus A \not \in \nul$.
       \item If $A \in \nul(C)$ and $B \subseteq A$, then $B \in \nul(C)$.
       \item If $A, B \in \nul(C)$, then $A \cup B \in \nul(C)$.
   \end{enumerate}
\end{lemma}
\begin{proof}
~ \\
    \noindent[Part 1]. By \Ax{1} (specifically, reflexivity of $\preceq$), we know $\emptyset|C \preceq \emptyset|C$, which implies $\emptyset \in \nul(C)$. 
    
    For the second statement, suppose for the sake of contradiction that $C \in \nul(C)$.  Then by definition of $\nul(C)$ we obtain that $C|C = C \cap C|C \preceq \emptyset|C$.   But $C|C \sim \Delta|\Delta$ by \Ax{3} and $\emptyset|C \sim \emptyset|\Delta$ by \autoref{lm:krantz6-2}.  So if $C|C \preceq \emptyset|C$, then $\Delta|\Delta \preceq \emptyset|\Delta$, contradicting \Ax{2}. \\
    
     \noindent[Part 2] Suppose for the sake of contradiction that both $A \in \nul(C)$ and $C \setminus \in \nul(C)$.  Then by definition of $\nul(C)$, we know both $A|C \preceq \emptyset|C$ and  $(C \setminus A)|C \preceq \emptyset|C$.  Applying \Ax{5} yields $C|C \preceq \emptyset|C$. By \autoref{lm:qConditionalIntersection}, that entails $C|\Delta \preceq \emptyset|\Delta$, which means $C \in \nul$ by definition of $\nul$.  But this contradicts the assumption that $C \not \in \nul$.\\
     
     \noindent[Part 3] Suppose for the sake of contradiction that $A \in \nul(C)$ and $B \subseteq A$ but that $B \not \in \nul(C)$.  Then by definition of $\nul(C)$, we know both $B|C \not \preceq \emptyset|C$.  By \Ax{1} (specifically, totality) and the fact that $C \not \in \nul$, it follows that 
     \begin{equation}\label{eqn:krantz8_1}
         B|C \succ \emptyset| C.
     \end{equation}
     Further, by \autoref{lm:krantz6-1}, we know that
     \begin{equation}\label{eqn:krantz8_2}
         A \setminus B |C \succeq \emptyset| C
     \end{equation}
     Applying \Ax{5} to \autoref{eqn:krantz8_1} and \autoref{eqn:krantz8_1} yields $A|C \succ \emptyset|C$, contradicting the assumption that $A \in \nul(C)$.\\
     
      \noindent[Part 4] This is a straightforward consequence of \Ax{5} and the previous part of this lemma.  Since $A \in \nul(C)$ and $(A\setminus B) \subseteq A$, it follows from Part 3 of this lemma that $(A \setminus B) \in \nul(C)$.  By definition of $\nul(C)$, that means 
      \[(A \setminus B)|C \preceq \emptyset|C.\]  
      Similarly, because $B \in \nul(C)$, we know $B|C \preceq \emptyset|C$.  Applying \Ax{5} to to previous two inequalities yields $A \cup B|C \preceq \emptyset|C$, i.e., that $A \cup B \in \nul(C)$ as desired.
\end{proof}

In standard probability theory, calls an event ``$P$-almost sure'' if $P(A) = 1$.  By analogy, define:

\begin{definition}
Given $C \not \in \nul$, an event $A$ is called $C$-\emph{almost sure}  if $C \setminus A \in \nul_{\preceq}(C)$.  Let ${\cal S}_{\preceq}(C)$ denote the set of all $\preceq$-almost sure events.
\end{definition}

As before, we drop the reference to the ordering $\preceq$ when it's clear from context, and we let ${\cal S} = {\cal S}(\Delta)$ be the almost-sure events conditional on $\Delta$.

\begin{lemma}\label{lm:qSureIsNotNull} 
Let $A \in {\cal S}(C)$ be an almost sure event.  Then:
    \begin{enumerate}
        \item $C \in {\cal S}(C)$ and $A \not \in \nul(C)$
        \item $A \cap B \not \in \nul(C)$ for any $B \not \in \nul)(C)$.
        \item $A \not \in \nul$.
    \end{enumerate}
    
\end{lemma}
\begin{proof} For brevity, we let $\neg A$ abbreviate $C \setminus A$ below.\\

\noindent [Part 1:] These two statements follow immediately from Parts 1 and 2 respectively of \autoref{lm:krantz8}.\\

\noindent [Part 2:] Suppose for the sake of contradiction that $A \cap B \in \nul(C)$.  Since $A \in {\cal S}(C)$, we know $\neg A \in \nul(C)$.  Thus, by Part 2 of \autoref{lm:krantz8}, since $\neg A \cap B \subseteq \neg A$, it follows that $\neg A \cap B \in \nul(C)$.  Since both  $A \cap B$ and $\neg A \cap B$ are members of $\nul(C)$, it follows from Part 4 of \autoref{lm:krantz8} that  $(A\cap B) \cup (\neg A \cap B) \in \nul(C)$.  However, $(A\cap B) \cup (\neg A \cap B) = B$, and so it follows that $B \in \nul(C)$, contradicting assumption.\\

\noindent [Part 3:] By contradiction. Suppose $A \in \nul$.  Since $A \cap C \subseteq A$, Part 3 of \autoref{lm:krantz8} entails that $A \cap C \in \nul$, i.e., that $A \cap C|\Delta \preceq \emptyset|\Delta$.  Since $C \not \in \nul$, \autoref{lm:qConditionalIntersection} entails $A|C \preceq \emptyset|C$.  Thus, $A \in \nul(C)$, contradicting Part 1 of this lemma.
\end{proof}

\noindent In addition to \autoref{lm:krantz8} and \autoref{lm:qSureIsNotNull}, we will frequently need to refer to additional facts about null and almost-sure events.  The first three parts of the following lemma slightly strengthen \cite[p. 230]{krantz_foundations_2006-1}'s Lemma 9. The remaining parts of the are the lemma are analogous to facts from standard probability theory that characterize when measure zero sets can be ignored, and hence, when one can focus on the almost sure events. For instance, in standard probability theory, if $P(A)=1$, then (1) $P(C|B) = P(C \cap A |B)$ and  (2) $P(C|B) = P(C |B \cap A)$ whenever $P(B) > 0$.   The last parts of the lemma are the analog of those two facts.

We separate the following lemma from the previous ones since its proof uses \Ax{6b}.  As we have said, \cite[p. 230]{krantz_foundations_2006-1}, prove the first parts of the lemma, but in doing so, they use a ``structural'' axiom (specifically, what they call Axiom 8) that requires the algebra of events to be sufficiently rich.  What we show is that similar facts can be derived from \Ax{6b}, which poses no constraint on the algebra of events itself.

\setcounter{lemma}{3}

\begin{lemma}\label{lm:krantz9}
Suppose $D \not \in \nul$.
\begin{enumerate}
    \item  If $A \in \nul(D)$, then $A|B \cap D \sim \emptyset|D$ for all $B \not \in \nul(D)$. 
    \item If $A \in \nul(D)$, then $A|B \cap D \sim \emptyset|C$ for all $B \not \in \nul(D)$ and $C \not \in \nul$.
    \item If $A \subseteq B$ and $A \in \nul(B)$, then $A \in \nul$.
     \item  If $A \in {\cal S}(D)$, then $A|A \sim A|B \cap D$ for all $B \not \in \nul(D)$.
    \item If $A \in {\cal S}(D)$, then $C|B \cap D \sim (A \cap C)|B \cap D$ for all $B \not \in \nul(D)$.    Similarly, if $A \in \nul(D)$, then $C|B \cap D \sim C \cup A| B \cap D$  for all $B \not \in \nul(D)$.
    \item If $A \in {\cal S}(D)$, then $C|B \cap D \sim C|(B \cap A) \cap D$ for all $B \not \in \nul(D)$.  Similarly, if $A \in \nul(D)$, then $C|B \cap D \sim C| (B \cup A) \cap D$  for all $B \not \in \nul(D)$.
\end{enumerate}
\end{lemma}
\setcounter{lemma}{2}

\begin{proof} For brevity, we let $\neg A$ abbreviate $D \setminus X$ below.\\

\noindent [Part 1:] As $A \in \nul(D)$, we have $A|D \preceq \emptyset|D$.  By \autoref{lm:qConditionalIntersection}, it follows that $A \cap D|\Delta \preceq \emptyset|\Delta$.  Since $A \cap B \cap D \subseteq A \cap D$, \autoref{lm:krantz7} entails that $A\cap B \cap D |\Delta \preceq A \cap D|\Delta$.  By transitivity, we obtain that $A \cap B \cap D|\Delta \preceq \emptyset|\Delta$.  Now since $B \not \in \nul(D)$ by assumption, \autoref{lm:qConditionalNull} entails that $B \cap D \not \in \nul$, and so $S|B \cap C$ is well-defined for all $S$.  Because $A \cap B \cap D|\Delta \preceq \emptyset|\Delta$ and $B \cap D \not \in \nul$, \autoref{lm:qConditionalIntersection} entails that $A|B \cap D \preceq \emptyset|B \cap D$.  But $\emptyset|B \cap D \sim \emptyset|D$ by \autoref{lm:krantz6-2}, and so $A|B \cap D \preceq \emptyset|D$ by transitivity.  Finally,  $A|B \cap D \succeq \emptyset|\Delta \sim \emptyset|D$ by \autoref{lm:krantz6-1} and \autoref{lm:krantz6-2}.  Since the inequality holds in both directions, we have  $A|B \cap D \sim \emptyset|D$ as desired.

\noindent [Part 2:] Follows immediately from Part 1 and \autoref{lm:krantz6-2}.

\noindent [Part 3:]  Suppose $A \subseteq B$ and $A \in \nul(B)$.  Since  $A \in \nul(B)$, we have $A|B \preceq \emptyset|B$.  By \autoref{lm:qConditionalIntersection}, it follows that $A \cap B|\Delta \preceq \emptyset|\Delta$.  Since $A \subseteq B$, we know $A \cap B = A$ and so 
\[A|\Delta = A \cap B|\Delta \preceq \emptyset |\Delta\]
Thus, $A \in \nul$, as desired.\\

\noindent [Part 4:]  Suppose for the sake of contradiction that $A|A \not \sim A|B \cap D$.
That means one of the following three options holds:  (1) $A \in \nul$, (2) $B \cap D \in \nul$, or (3) either $A|A \prec A|B \cap D$ or $A|A \succ A|B \cap D$, where this last option follows from \Ax{1} (specifically, the totality of $\preceq$). Option 1 is impossible because $A \in {\cal S}(D)$ by assumption and therefore $A \not \in \nul$ by Part 3 of \autoref{lm:qSureIsNotNull}.  Option 2 is likewise impossible because $B \not \in \nul(D)$, and thus $B \cap D \not \in \nul$ by \autoref{lm:qConditionalNull}.  So option 3 is the one that must obtain.

By \Ax{3} and \Ax{4}, we have $A|A \sim \Delta|\Delta \succeq A|B \cap D$.  Thus, the first disjunct of option 3 ( that $A|A \prec A|B \cap D$) is impossible.  Thus, it must be the case that the second disjunct holds, i.e., that $A|B \cap D \prec A|A$.  Since $A \in {\cal S}(D)$, by definition of ${\cal S}(D)$, we know $\neg A \in \nul(D)$.  Hence, $\neg A|B \cap D \preceq \emptyset|A$ by Part 2 of this lemma. Since (1)  $A|B \cap D \prec A|A$ and (2) $\neg A|B  \cap D \preceq \emptyset|A$, it follows from \Ax{5} that $A \cup \neg A|B \prec A|A$.  But $A \cup \neg A = D$, and so we've shown $D|B \cap D \prec A|A$.  Since $D|B \cap D \sim B \cap D|B \cap D$ by \Ax{4}, it follows that $B \cap D|B \cap D \prec A|A$, contradicting the \Ax{3}.\\

\noindent [Part 5:]  First, note that because $B \not \in \nul(D)$, we know $B \cap D \not \in \nul$ by \autoref{lm:qConditionalNull}.  Thus, all expressions of the form $\cdot|B \cap D$ below are well-defined.

Next, since $A \in {\cal S}(D)$, we have $\neg A \in \nul(D)$.  Thus, by Part 3 of \autoref{lm:krantz8}, we have $(\neg A \cap C) \in \nul(D)$ since $\neg A \cap C \subseteq \neg A$.  Thus, $\neg A \cap C|B \cap D \sim \emptyset|B \cap D$ by Part 2 of this Lemma.  Since (1) $A \cap C|B \cap D \sim A \cap C|B \cap D$ by \Ax{1} (specifically, the reflexivity of $\preceq$) and (2) $\neg A \cap C|B \cap D \sim \emptyset|B \cap D$ (as just shown),  \Ax{5} entails:
\[C|B \cap D = (A \cap C) \cup (\neg A \cap C) |B \cap D \sim (A \cap C) \cup \emptyset|B \cap D  = A \cap C|B \cap D \]
as desired.

To show the second claim, suppose $A \in \nul(D)$ so that $\neg A \in {\cal S}(D)$ by definition.  Because $\neg A \in {\cal S}(D)$, what we've just shown entails both that $C|B \cap D  \sim C \cap \neg A|B \cap D $ and that $(C \cup A)|B \cap D \sim (C \cup A) \cap \neg A|B \cap D$.  Notice that $C \cap \neg A = (C \cup A) \cap \neg A$, and so by transitivity we get  $C|B \cap D \sim C\cup A|B \cap D$ as desired.  \\

\noindent [Part 6:] As before, note that because $B \not \in \nul(D)$, we know that $B \cap D \not \in \nul$  \autoref{lm:qConditionalNull}.  Thus, all expressions of the form $\cdot|B \cap D$ below are well-defined.

Suppose for the sake of contradiction that $C|B \cap D \not \sim C|(B \cap D) \cap A$.  As we just noted, $B \cap D \not \in \nul$, and so the left-hand-side is well-defined. The right-hand side is also well-defined.  Why? By Part 1 of \autoref{lm:qSureIsNotNull}  $D \in {\cal S}(D)$, and by assumption, $B \not \in \nul(D)$.  Thus, $B \cap D \not \in \nul(D)$ by  Part 2 of \autoref{lm:qSureIsNotNull}. Applying the same logic again, since $A \in {\cal S}(D)$ and  $B \cap D \not \in \nul(D)$, Part 2 of \autoref{lm:qSureIsNotNull} entails $A \cap B \cap D \not \in \nul(D)$.

Since  $C|B \cap D \not \sim C|(B \cap A) \cap D$ and both sides are well-defined, by \Ax{1}, either $C|B \cap D  \succ C|(B \cap A) \cap D $ or vice versa.  Without loss of generality, assume the former.  Define:
    \begin{alignat*}{3}
        &\mathfrak{A}=D  & \quad \quad & \mathfrak{A}' = D  \\
        & \mathfrak{B} = B \cap D& \quad \quad & \mathfrak{B}'= (B \cap A) \cap D \\
        & \mathfrak{C}= C \cap B \cap D & \quad \quad & \mathfrak{C}'= C \cap B \cap A \cap D
    \end{alignat*}
    By \Ax{4} and assumption, $\mathfrak{C}|\mathfrak{B} \succ \mathfrak{C}'|\mathfrak{B}'$.  Because $A \in {\cal S}(D)$, Part 5 of this lemma entails $\mathfrak{B}|\mathfrak{A} = B \cap D|D \sim \mathfrak{B}'|\mathfrak{A}' = (B \cap A) \cap D|D$.  So by \Ax{6b}, it follows that 
    \[ \mathfrak{C}|\mathfrak{A} \succ \mathfrak{C}'|\mathfrak{A}' = (C \cap B \cap A \cap D)|D.\]
    
    But again, by Part 4 of this lemma, $\mathfrak{C}'|\mathfrak{A}' = C \cap B \cap A \cap D|D \sim (C \cap B)|D = \mathfrak{C}|\mathfrak{A}$.  So we've shown $\mathfrak{C}|\mathfrak{A} \succ \mathfrak{C}'|\mathfrak{A}'$ and $\mathfrak{C}|\mathfrak{A} \sim \mathfrak{C}'|\mathfrak{A}'$, which is a contradiction.

The second claim is proven in the same way as the second claim in Part 5.
\end{proof}

The final important lemma is special case of \Ax{6a} that we use frequently; it's proof relies on \Ax{6b} in cases in which null sets are involved.
\begin{lemma}\label{lm:Aditya}
    Suppose $A\supseteq B\supseteq C$ and $A\supseteq B'\supseteq C$. If $B|A \succeq C|B'$ and $B \not \in \nul$, then $B'|A \succeq C|B$. Further, if the antecedent is $\succ$, the consequent is $\succ$.   
\end{lemma}
\begin{proof}
    Assume $A\supseteq B\supseteq C$, $A\supseteq B' \supseteq C$, and $B|A \succeq C|B'$. 
    
   First note that if $C \in \nul$, then it follows from \autoref{lm:krantz9} that $C|B \sim \emptyset|\Delta$.  Hence, by \autoref{lm:krantz6-1}, we get $B'|A \succeq C|B$. 
    In this case, we also get the stronger result that  $B'|A \succ C|B$ without any further assumption.  Why?  Since by hypothesis $B|A \succeq C|B'$, the expression $C|B'$ is defined and $B' \not \in \nul$.  Since $B' \subseteq A$, it follows from Part (3) of \autoref{lm:krantz9}  that $B'|A \not \preceq \emptyset |\Delta$.  By \Ax{1} (specifically, totality), we obtain that $B'|A \succ \emptyset | \Delta \sim C|B$ as desired.

    So suppose $C \not \in \nul$.  Suppose, for the sake of contradiction, that $B'|A \not \succeq C|B$. Since $B|A \succeq C|B'$, we know $A \not \in \nul$ and so the expression $B|A$ is well-defined.  Further, $B \not \in \nul$ by assumption.  Thus, if $B'|A \not \succeq C|B$, it follows from by \Ax{1} (specifically, totality of the ordering) that $C|B \succ B'|A$.  Now we apply \Ax{6a} with $A' = A$, $C' = C$. We thus get $C|A \succ C'|A' = C|A$, which contradicts the reflexivity of the weak ordering (i.e., \Ax{1}). Thus, $B'|A \succeq C|B$.
   
   Finally, we must show that if the antecedent is strict (i.e., $B|A \succ C|B'$), then so is the conclusion (i.e., $B'|A \succ C|B$).  Again, we suppose for the sake of contradiction that $B'|A \not \succ C|B$.  Because $A, B \not \in \nul$ by the above reasoning, it follows from  by \Ax{1} (specifically, totality of the ordering), $C|B \succeq B'|A$.   Again, we apply \Ax{6a} with $A' = A$ and $C' = C$ to obtain $C|A \succ C|A$, which again contradicts the reflexivity of the weak ordering.
\end{proof}

\setcounter{lemma}{16}

\section{QLL and Qualitative Favoring}

\begin{theorem}\label{thm:QLLFavoring}
    Suppose $H_1$ and $H_2$ are finite. Then $E$ qualitatively favors $H_1$ over $H_2$ if (1) $\emptyset|\theta \sqsubset E|\theta$ for all $\theta \in H_1 \cup H_2$ and (2) $E|\theta_2 \sqsubseteq E|\theta_1$ for all $\theta_1 \in H_1$ and $\theta_2 \in H_2$. Under \autoref{assum} (below), the converse holds as well. In both directions, the favoring inequality is strict exactly when the likelihood inequality is strict.  
    
    If $H_1 = \{\theta_1\}$ and $H_2 = \{\theta_2\}$ are simple, then $E$ qualitatively favors $H_1$ over $H_2$ if and only if \f{qll} entails so. No additional assumptions are required in this case.
\end{theorem}

\begin{assumption}\label{assum}
For all orderings $\sqsubseteq$ satisfying the axioms above and for all non-empty $H \subseteq \Theta$, there exists an ordering $\preceq$ satisfying the axioms such that (A)  $H \in {\cal S}_{\preceq}$ and (B)  $\theta|\Delta \succ \emptyset|\Delta$ for all $\theta \in H$.
\end{assumption}

We believe \autoref{assum} is provable, but we have not yet produced a proof.  The assumption more-or-less says that one's ``prior'' ordering over the hypotheses is unconstrained by the ``likelihood'' relation $\sqsubseteq$.  This is exactly analogous to the quantitative case.  In that case, a joint distribution $Q^{\mathbb{E}}$ on $\Delta = \Theta \times \Omega^{\mathbb{E}}$ is determined (i) the measures $\langle P_{\theta}(\cdot) \rangle_{\theta \in \Theta}$ over $\Omega^{\mathbb{E}}$ and (ii) one's prior $Q$ over $\Theta$.  Although the joint distribution  $Q^{\mathbb{E}}$ is constrained by $\langle P_{\theta}(\cdot) \rangle_{\theta \in \Theta}$, one's prior $Q$ on $\Theta$ is not, and so for any non-empty hypothesis $H \subseteq \Theta$, one can define some $Q$ such that $Q(H)=1$ and $Q(\theta) > 0$ for all $\theta \in H$.  That's what the assumption says.
 
The theorem follows from the following three propositions, which we prove in the ensuing sections.

\begin{lemma}\label{lm:QRegularLikelihood}
    Suppose $E \cap (H_1 \cup H_2) \not \in \nul$ for all orderings satisfying the axioms above and for which $H_1 \cup H_2 \not \in \nul_{\preceq}$.  If \autoref{assum} holds, then $\emptyset|\theta \sqsubset E|\theta$ for all $\theta \in H_1 \cup H_2$.
\end{lemma}

\begin{proposition}\label{prop:QBayesFactorFavoring}
    Suppose $H_1 \cap H_2 = \emptyset$.  
        \begin{enumerate}
            \item Suppose $E \cap (H_1 \cup H_2) \not \in {\cal N}$.  If $E|H_1 \succeq E|H_2$, then $H_1|E\cap (H_1 \cup H_2) \succeq H_1|H_1\cup H_2$.
            \item Suppose  $H_1, H_2 \not \in {\cal N}$.  If $H_1|E \cap (H_1 \cup H_2) \succeq H_1|H_1\cup H_2$, then   $E|H_1 \succeq E|H_2$.
        \end{enumerate}    
    Further, if the antecedent of either conditional contains a strict inequality $\succ$, then so does the consequent.
\end{proposition}

\begin{proposition}\label{prop:qBayesFactorLikelihoods}
    Suppose $H_1$ and $H_2$ are finite and that $H_1 \cap H_2 = \emptyset$.  If $E|\theta_2 \sqsubseteq E|\theta_1$ for all $\theta_1 \in H_1$ and $\theta_2 \in H_2$, then $E|H_2 \preceq E|H_1$ for all orderings $\preceq$ (satisfying the axioms above) for which $H_1, H_2 \not \in \nul$.   The converse holds under \autoref{assum}.
    In both directions, $\preceq$ can be replaced by the strict inequality $\prec$ if and only if the $\sqsubseteq$ can be replaced by the strict relation $\sqsubset$.
\end{proposition}
Notice that the second proposition holds in both directions when $H_1$ and $H_2$ are simple hypotheses by \Ax{0}. The additional assumption is necessary only if $H_1$ and $H_2$ are composite.\\

\noindent \emph{Proof of \autoref{thm:QLLFavoring}:}  We consider the right-to-left direction first.  Suppose (1) $\emptyset|\theta \sqsubset E|\theta$ for all $\theta \in H_1 \cup H_2$ and (2) $E|\theta_2 \sqsubseteq E|\theta_1 $ for all $\theta_1 \in H_1$ and $\theta_2 \in H_2$.  We must show $E$ qualitatively favors $H_1$ over $H_2$.   To do so, let $\preceq$ be any ordering satisfying the axioms such that $H_1 \cup H_2 \not \in \nul$.  We must show $H_1|E \cap (H_1 \cup H_2) \succeq H_1 | H_1 \cup H_2$.  
To do so, we must first show that $E \cap (H_1 \cup H_2) \not \in \nul$, so that the left-hand side of the last inequality is well-defined.  Since $H_1 \cup H_2$ is finite and not a member of $\nul$, it follows from Part 4 of \autoref{lm:krantz8} that there is some $\eta \in H_1 \cup H_2$  such that $\{\eta\} \not \in \nul$.  By assumption (1) that $\emptyset|\theta \sqsubset E|\theta$ for all $\theta \in H_1 \cup H_2$, we know $\emptyset|\eta \sqsubset E|\eta$ in particular.  By \Ax{0} and the fact that $\{\eta\} \not \in \nul$, we can infer $\emptyset|\eta \prec E|\eta$. \autoref{lm:qConditionalIntersection} then entails that $\emptyset \cap \{\eta\} | \Delta \prec E \cap \{\eta\}|\Delta$.  So by \Ax{2}, $ E \cap \{\eta\} \not \in \nul$.  Because  $E \cap \{\eta\} \subseteq E \cap (H_1 \cup H_2) \not \in \nul$ as desired.

Now that we know $E \cap (H_1 \cup H_2) \not \in \nul$, we argue that $H_1|E \cap (H_1 \cup H_2) \succeq H_1 | H_1 \cup H_2$ via proof by cases.  Since $H_1 \cup H_2 \not \in \nul$, it follows from \autoref{lm:krantz8} that at least one of $H_1$ and $H_2$ is not a member of $\nul$.  Thus, there are three cases to consider:  (A) Neither is in $\nul$, (B) $H_1 \in \nul$ but $H_2 \not \in \nul$, and (C) $H_2 \in \nul$ but $H_1 \not \in \nul$.\\

\noindent Case A:  Suppose $H_1, H_2 \not \in \nul$.  By \autoref{prop:qBayesFactorLikelihoods} and our first assumption (that $\emptyset|\theta \sqsubset E|\theta$ for all $\theta \in H_1 \cup H_2$), it follows that $E|H_2 \preceq E|H_1$.  Since $E \cap (H_1 \cup H_2) \not \in \nul$ and  $E|H_2 \preceq E|H_1$, Part 1 of \autoref{prop:QBayesFactorFavoring} allows us to conclude that $H_1|E \cap (H_1 \cup H_2) \succeq H_1 | H_1 \cup H_2$, as desired. \\

\noindent Case B:  Suppose $H_1 \in \nul$ but $H_2 \not \in \nul$.  Since $H_1 \in \nul$, Part 1 of \autoref{lm:krantz9} entails both that $H_1|E \cap (H_1 \cup H_2) \sim \emptyset|\Delta$ and $H_1|H_1 \cup H_2 \sim \emptyset|\Delta$.  By transitivity of $\sim$ (which follows from \Ax{1}), we know that $H_1|E \cap (H_1 \cup H_2) \sim H_1|H_1 \cup H_2$, and hence, $H_1|E \cap (H_1 \cup H_2) \succeq H_1|H_1 \cup H_2$ as desired. \\

\noindent Case C:  Suppose $H_2 \in \nul$ but $H_1 \not \in \nul$.  First note that because $H_2 \in \nul$, we obtain that $\neg H_2 := \Delta \setminus H_2 \in {\cal S}$ by definition of ${\cal S}$.  Further, because $H_1 \cap H_2 = \emptyset$, we know $E \cap H_1 \subseteq H_1 \subseteq \neg H_2$, and hence $(E \cap (H_1 \cup H_2)) \cap \neg H_2 = E \cap H_1$.
Thus:
\begin{eqnarray*}
H_1 | E \cap (H_1 \cup H_2) &\sim& H_1 | (E \cap (H_1 \cup H_2)) \cap \neg H_2 \mbox{ by Part 6 of \autoref{lm:krantz9} since } \neg H_2 \in {\cal S} \\
 &=& H_1 | E \cap H_1 \mbox{ as just shown } \\
 &\sim& E \cap H_1 | E \cap H_1 \mbox{ by \Ax{4}} \\
 &\sim& H_1 | H_1 \mbox{ by \Ax{3}} \\
 &\sim& H_1 | H_1 \cup H_2  \mbox{ by Part 6 of \autoref{lm:krantz9} since } H_2 \in \nul 
\end{eqnarray*}

In the left to right direction, suppose $E$ qualitatively favors $H_1$ over $H_2$, or in other words, $H_1|E \cap (H_1 \cup H_2) \succeq H_1|H_1 \cup H_2$ for all orderings $\preceq$ for which $H_1 \cup H_2 \not \in \nul_{\preceq}$.  Further, suppose \autoref{assum} holds.  We must show  (1) $\emptyset|\theta \sqsubset E|\theta$ for all $\theta \in H_1 \cup H_2$ and (2) $E|\theta_2 \sqsubseteq E|\theta_1$ for all $\theta_1 \in H_1$ and $\theta_2 \in H_2$. 

To show (1), recall we've assumed $H_1|E \cap (H_1 \cup H_2) \succeq H_1|H_1 \cup H_2$ for all orderings $\preceq$ satisfying the axioms and for which $H_1 \cup H_2 \not \in \nul$.  Thus, $E \cap (H_1 \cup H_2) \not \in \nul$ for all orderings $\preceq$ satisfying the axioms and for which $H_1 \cup H_2 \not \in \nul$. By \autoref{lm:QRegularLikelihood} and \autoref{assum}, we're done.

To show (2), recall we've assumed $H_1|E \cap (H_1 \cup H_2) \succeq H_1|H_1 \cup H_2$ for all orderings $\preceq$ for which $H_1 \cup H_2 \not \in \nul$.  Thus, it is also the case theat $H_1|E \cap (H_1 \cup H_2) \succeq H_1|H_1 \cup H_2$ for all orderings $\preceq$ for which \emph{both} $H_1 \not \in \nul$ and $H_2 \not \in \nul$.  It follows from Part 2 of \autoref{prop:QBayesFactorFavoring} that $E|H_2 \preceq E|H_1$ for all orderings $\preceq$ for which \emph{both} $H_1 \not \in \nul$ and $H_2 \not \in \nul$. \autoref{prop:qBayesFactorLikelihoods} then entails our desired conclusion that $E|\theta_2 \sqsubseteq E|\theta_1$ for all $\theta_1 \in H_1$ and $\theta_2 \in H_2$.
\begin{flushright}
  $\Box$
\end{flushright}

\subsection{Proof of \autoref{lm:QRegularLikelihood}}

\noindent \emph{Proof of \autoref{lm:QRegularLikelihood}:}  By contraposition.  Assume that it's not the case that $\emptyset|\theta \sqsubset  E|\theta$ for all $\theta$.  We show there is an ordering $\preceq$ satisfying the axioms such that $H_1 \cup H_2 \not \in \nul_{\preceq}$ and $E \cap (H_1 \cup H_2) \in \nul_{\preceq}$. 

Because it's not the case that $\emptyset|\theta \sqsubset  E|\theta$ for all $\theta \in H_1 \cup H_2$, there exists $\theta \in H_1 \cup H_2$ such that $\emptyset|\theta \not \sqsubset  E|\theta$.  By \Ax{2} (for $\sqsubseteq$), $\theta \not \in \nul_{\sqsubseteq}$ and so the expression $\cdot|\theta$ is well-defined (with respect to $\sqsubseteq$).  Since $\emptyset|\theta \not \sqsubset  E|\theta$ and $\cdot|\theta$ is well-defined, it follows from \Ax{1} (specifically totality of $\sqsubseteq$) that $ E|\theta \sqsubseteq \emptyset|\theta$.

Using \autoref{assum}, pick any ordering $\preceq$ satisfying the axioms such that (1) $H=\{\theta\} \in {\cal S}_{\preceq}$ and (2) $\theta|\Delta \succ \emptyset|\Delta$.  By 2 and \Ax{2}, we know $\{\theta\} \not \in \nul_{\preceq}$, and since $\{\theta\} \subseteq H_1 \cup H_2$, it follows from Part 3 of \autoref{lm:krantz8} that $H_1 \cup H_2 \not \in \nul_{\preceq}$.

Because $E|\theta \sqsubseteq \emptyset|\theta$, \Ax{0} entails that $E|\theta \preceq \emptyset|\theta$.  By \autoref{lm:qConditionalIntersection}, it follows that $E \cap \{\theta\}|\Delta \preceq \emptyset|\Delta$, and hence, $E \cap \{\theta\} \in \nul_{\preceq}$ by \Ax{2}.  By assumption, $\{\theta\} \in {\cal S}_{\preceq}$, and so $\neg \{\theta\} \in \null_{\preceq}$. Because $E \cap \neg \{\theta\} \subseteq \neg \{\theta\}$, Part 3 of \autoref{lm:krantz8} allows us to conclude $E \cap \neg \{\theta\} \in \nul$.  Because both  $E \cap \{\theta\}$ and $ E \cap \neg \{\theta\}$ are null sets, Part 4 of \autoref{lm:krantz8} entails $E = (E \cap \neg \{\theta\}) \cup (E \cap \{\theta\})$ is in $\nul$.  Again, because $E \cap (H_1 \cup H_2) \subseteq E$ and $E \in \nul$, we know from Part 3 of \autoref{lm:krantz8} that $E \cap (H_1 \cup H_2) \in \nul$, and we're done.
\begin{flushright}
  $\Box$
\end{flushright}

\subsection{Proof of \autoref{prop:QBayesFactorFavoring}}
The proof of \autoref{prop:QBayesFactorFavoring} requires the following lemmas.
\setcounter{lemma}{0}
\begin{lemma}\label{lm:qllcf} 
    Suppose $H_1 \cap H_2 = \emptyset$.  If $E|H_1 \succeq E|H_2$, then $E|H_1 \succeq E|(H_1 \cup H_2)$. The converse holds if $H_2 \not \in \nul$.
    Further, if either side of the biconditional is $\succ$, the other side is $\succ$ too.
\end{lemma}
\begin{proof}
   In the left-to-right direction, assume that $E|H_1 \succeq E|H_2$, so that $H_1, H_2 \not \in \nul$.  Suppose, for the sake of contradiction that $E|H_1 \not \succeq E|(H_1 \cup H_2)$.  Since $H_1 \not \in \nul$, it follows from \autoref{lm:krantz8} (part 3)  that $H_1 \cup H_2 \not \in \nul$, and hence, the expression $\cdot |(H_1 \cup H_2)$ is well-defined.  Since $E|H_1 \not \succeq E|(H_1 \cup H_2)$, by totality of $\preceq$ (part of \Ax{1}), it follows that $E|H_1 \prec E|(H_1 \cup H_2)$. By \Ax{4}, we get $E \cap (H_1 \cup H_2) | H_1 \cup H_2 \succ E \cap H_1|H_1$. We now apply \autoref{lm:Aditya} with:
   \begin{alignat*}{3}
       &A = H_1 \cup H_2 & & \\ 
       &B = E \cap (H_1 \cup H_2) &\quad \quad   &B' = H_1 \\
       &C = E\cap H_1 & &
   \end{alignat*}
   Clearly, the necessary containment relations hold and we just showed $B \not \in \nul$. Since $B|A \succ C|B'$, we get that $B'|A \succ C|B$. So:
   \begin{equation}\label{eqn:qllbf1}
       H_1|H_1 \cup H_2 \succ E\cap H_1|E \cap (H_1 \cup H_2)
   \end{equation}
   Further, since $E|H_2 \preceq E|H_1$ and $E|H_1 \prec E|(H_1 \cup H_2)$, we know $E|H_2 \prec E|H_1 \cup H_2$. By similar reasoning as before (swapping $H_1, H_2$ in our application of \autoref{lm:Aditya}), we can thus conclude that:
   \begin{equation}\label{eqn:qllbf2}
       H_2|H_1 \cup H_2 \succ E \cap H_2|E \cap (H_1 \cup H_2)
   \end{equation}
   We now apply \Ax{5} to equations \eqref{eqn:qllbf1} and \eqref{eqn:qllbf2}.  To do so, let
   \begin{alignat*}{3}
       &A = H_1, && \quad \quad \quad && A' = E\cap H_1 \\ 
       &B = H_2, && \quad \quad \quad && B' = E\cap H_2 \\
       &C = H_1 \cup H_2, && \quad \quad \quad  && C' = E \cap (H_1 \cup H_2) 
   \end{alignat*}
   Note that $A\cap B = \emptyset$ by our initial assumption that $H_1 \cap H_2 = \emptyset$. From this, it follows from $A' \cap B' = \emptyset$. Now, \autoref{eqn:qllbf1} says that $A|C \succ A'|C'$ and \autoref{eqn:qllbf2} says that $B|C \succ B'|C'$. Thus, we get that $A \cup B|C \succ A' \cup B'|C'$ by \Ax{5}. In other words:
   \begin{equation}
       H_1 \cup H_2|H_1 \cup H_2 \succ (E \cap H_1)\cup (E\cap H_2)|E \cap (H_1 \cup H_2) = E \cap (H_1 \cup H_2)|E \cap (H_1 \cup H_2) 
   \end{equation}
   This equation says that $X|X \succ Y|Y$ for two events $X,Y$.  This contradicts \Ax{3}, so we get that $E|H_1 \succeq E|(H_1 \cup H_2)$, as desired. 
   
Note that if $E|H_1 \succ E|H_2$, \autoref{eqn:qllbf1} would be $\succeq$, since we would get $E|H_1 \preceq E|(H_1 \cup H_2)$ from our initial steps. However, \autoref{eqn:qllbf2} would still be $\succ$, since $E|H_2 \prec E|H_1$. And for \Ax{5}, as long as either premise is $\succ$, the conclusion is $\succ$. Thus, we would still reach a contradiction, and thus get $E|H_1 \succ E|(H_1 \cup H_2)$.
   
Now, for the right-to-left direction assume $E|H_1 \succeq E|(H_1 \cup H_2)$ and that $H_2 \not \in \nul$.  We want to show $E|H_1 \succeq E|H_2$.  

Note that if $E \cap (H_1 \cup H_2) \in \nul$, then since $E \cap H_i \subseteq E \cap (H_1 \cup H_2)$ for $i=1,2$, we know from Part 3 of \autoref{lm:krantz8} that $E \cap H_i \in \nul$.  It follows that:
  \begin{eqnarray*}
    E|H_1 &\sim& E \cap H_1 | H_1 \mbox{ by \Ax{4}} \\ 
    &\sim& \emptyset | H_2 \mbox{ by Part 2 of  \autoref{lm:krantz9}, since } E \cap H_1 \in \nul \\ 
    &\sim& E \cap H_2 | H_2 \mbox{ by Part 2 of  \autoref{lm:krantz9}} E \cap H_2 \in \nul \\
    &\sim& E \cap H_2 | H_2 \mbox{ by \Ax{4}} \\ 
  \end{eqnarray*}
and $E|H_1 \succeq E|H_2$ as desired.
 
Thus, assume that $E \cap (H_1 \cup H_2)\not \in \nul$, and define:
   \begin{alignat*}{3}
       &A = H_1 \cup H_2 & & \\ 
       &B =  E \cap (H_1 \cup H_2) &\quad \quad   &B' = H_1 \\
       &C = E\cap H_1 & &
   \end{alignat*}
 We use \autoref{lm:Aditya}.  Because $C|B' \succeq B|A$ (by \Ax{4}) and we have assumed that $B = E \cap (H_1 \cup H_2)\not \in \nul$, we get that $C|B \succeq B'|A$. This means that:
   \begin{equation}\label{eqn:qllbf3}
       E \cap H_1 | E \cap (H_1 \cup H_2) \succeq H_1 | H_1 \cup H_2
   \end{equation}
   Suppose, for the sake of contradiction, that $E|H_1 \not \succeq E|H_2$. We know $H_1 \not \in \nul$ (because $E|H_1 \succeq E|(H_1 \cup H_2)$) and we've simply assumed that $H_2 \not \in \nul$ for this direction of the proof.  Thus, both $\cdot|H_1$ and $\cdot|H_2$ are well-defined, and so by the totality of $\preceq$ (part of \Ax{1}), it follows that $E|H_2 \succ E|H_1$.  Because $E|H_1 \succeq E|(H_1 \cup H_2)$, by transitivity we get $E|H_2 \succ E|(H_1 \cup H_2)$. So  if we apply the same reasoning we did to derive \autoref{eqn:qllbf3} (switching $H_1$ and $H_2$),we get that:
   \begin{equation}\label{eqn:qllbf4}
       E \cap H_2 | E \cap (H_1 \cup H_2) \succ H_2 | H_1 \cup H_2
   \end{equation}
   Note that this comparison is strict as the premise was strict. From here, as in the other direction of the proof, we use \Ax{5} to derive $E \cap (H_1 \cup H_2)|E \cap (H_1 \cup H_2) \succ H_1 \cup H_2 | H_1 \cup H_2$, contradicting \Ax{3} (the conclusion is strict because at least one of the premises is strict).  So we've shown $E|H_1 \succeq E|H_2$.
   
   If our assumption contained a strict inequality (i.e., $E|H_1 \succ E|(H_1 \cup H_2)$), then the only changes in the preceding argument would be that \autoref{eqn:qllbf3} would be strict whereas \autoref{eqn:qllbf4} would not. This would still lead to a contradiction, and so in this case we would get $E|H_1 \succ E|H_2$.
\end{proof}

\begin{proof}[Proof of \autoref{prop:QBayesFactorFavoring}]
Suppose $H_1 \cap H_2 = \emptyset$. \\

\noindent [Part 1:] Suppose $E|H_1 \succeq E|H_2$ and that $E \cap (H_1 \cup H_2) \not \in \nul$.
By \autoref{lm:qllcf}, it follows that  $E|H_1 \succeq E|(H_1 \cup H_2)$.  Applying \Ax{4}, we get 
    \begin{equation}\label{eqn:qbff1a}
        E\cap H_1|H_1 \succeq E\cap (H_1 \cup H_2)|H_1 \cup H_2
    \end{equation}
    Define:
    \begin{alignat*}{3}
        &A = H_1 \cup H_2 & & \\ 
        &B = E \cap (H_1 \cup H_2)  & \quad \quad & B' = H_1 \\
        &C = E\cap H_1 & &
    \end{alignat*}
    Note that \autoref{eqn:qbff1a} says that $C|B' \succeq B|A$ and we've assumed that $B = E \cap (H_1 \cup H_2) \not \in \nul$. So  \autoref{lm:Aditya} tells us that $C|B \succeq B'|A$, i.e.,
    \begin{equation*}
        E\cap H_1|E\cap (H_1 \cup H_2) \succeq H_1|H_1\cup H_2
    \end{equation*}
    Applying \Ax{4} to the left-hand side of that inequality yields the desired result. Note that if we had assumed $E|H_1 \succ E|H_2$, then our conclusion would contain $\succ$ because both  \autoref{lm:Aditya} and \autoref{lm:qllcf}  yield strict comparisons. 

    In the right to left direction, suppose $H_1|E\cap (H_1 \cup H_2) \succeq H_1|H_1\cup H_2$ and that $H_1, H_2 \not \in \nul$. By \Ax{4}:
    \begin{align*}
        H_1&|E\cap (H_1 \cup H_2) \\
        &\sim H_1\cap E\cap (H_1 \cup H_2)|E\cap (H_1 \cup H_2) \\
        &= E \cap H_1|E\cap (H_1 \cup H_2)
    \end{align*}
    So, we know $E \cap H_1|E\cap (H_1 \cup H_2) \succeq H_1|H_1\cup H_2$. Now, we apply \autoref{lm:Aditya} again with:
    \begin{alignat*}{3}
        &A = H_1 \cup H_2 &  \\
        &B = H_1 & \quad \quad 
        &B' =  E \cap (H_1 \cup H_2) \\
        &C = E\cap H_1 & &
    \end{alignat*}
    The containment relations are satisfied; $B = H_1 \not \in \nul$ by assumption, and we just showed $C|B' \succeq B|A$. Thus, we get $C|B \succeq B'|A$, i.e., 
    \begin{align*}
        E\cap H_1|H_1 &\succeq E\cap (H_1\cup H_2)|(H_1 \cup H_2).
    \end{align*}
    Applying \Ax{4} to both sides of the last inequality yields $E|H_1 \succeq E|H_1 \cup H_2$. So by \autoref{lm:qllcf} and the assumption that $H_2 \not \in \nul$, we get $E|H_1 \succeq E|H_2$, as desired. As before, if the premise were $\succ$, the conclusion would also be $\succ$ because the necessary lemmas would yield strict comparisons.
\end{proof}

\subsection{Proof of \autoref{prop:qBayesFactorLikelihoods}}
\setcounter{lemma}{17}
The proof of \autoref{prop:qBayesFactorLikelihoods} requires a few lemmata.
\begin{lemma}\label{lm:qbfl1} Suppose $H_1 \cap H_2 = \emptyset$.  If $E|H_1 \sim E|H_2$ then $E|(H_1 \cup H_2) \sim E|H_1 \sim E|H_2$.
\end{lemma}
\setcounter{lemma}{1}
\begin{proof} 
The result follows immediately from \autoref{lm:qllcf}.

\end{proof}

\begin{lemma}\label{lm:qbfl2} Suppose $H_1 \cap H_2 = \emptyset$.
    \begin{enumerate}[a)]
        \item If $E|H_1, E|H_2 \preceq E|H_3$, then $E|(H_1 \cup H_2) \preceq E|H_3$.
        \item If $E|H_3 \preceq E|H_1, E|H_2$, then $E|H_3 \preceq E|(H_1 \cup H_2)$.
    \end{enumerate}
    If the premise is not strict AND $E|H_1 \sim E|H_2$, then the conclusion is not strict. Otherwise, the conclusion is strict. 
\end{lemma}
\begin{proof}
    We first prove a).
    By totality of $\preceq$ (which is part of \Ax{1}), one of the following three cases must hold:
        \begin{enumerate}
            \item $E|H_2 \prec E|H_1$,
            \item $E|H_1 \prec E|H_2$,
            \item $E|H_1 \sim E|H_2$,
        \end{enumerate}
    If the first case holds, then \autoref{lm:qllcf} entails $E|(H_1 \cup H_2) \prec E|H_1$.  Hence, \Ax{1} (transitivity) and the assumption that  $E|H_1, E|H_2 \preceq E|H_3$ together entail that $E|(H_1 \cup H_2) \prec E|H_3$.
    
    If the second case holds, we reason analogously to show
    $E|(H_1 \cup H_2) \prec E|H_2 \prec E|H_3$.
    
    If the third case holds, then \autoref{lm:qbfl1} entails that $E|(H_1 \cup H_2) \sim E|H_1 \preceq E|H_3$, and we're done. Note that in the first two cases the conclusion is always strict. And even in this case, if $E|H_1 \sim E|H_2 \prec E|H_3$, the conclusion would be strict.
    
    The proof of part b) is analogous to a). 
\end{proof}
\setcounter{lemma}{18}

We're now ready to prove the main proposition.

\begin{proof}[Proof of \autoref{prop:qBayesFactorLikelihoods}]
    In the right-to-left direction, suppose $E|\theta_2 \sqsubseteq E|\theta_1$ for all $\theta_1 \in H_1$ and $\theta_2 \in H_2$.  We want to show $E|H_2 \preceq E|H_1$ for all orderings $\preceq$ for which $H_1, H_2 \not \in \nul$.  So let $\preceq$ be any such ordering.

    We show $E|H_2 \preceq E|H_1$ by induction on the maximum of the number of elements in $H_1$ or $H_2$.

    In the base case, suppose both $H_1 = \{\theta_1\}$ and $H_2=\{\theta_2\}$ both have one element.  Then by \Ax{0} and the assumption that $H_1, H_2 \not \in \nul$, it immediately follows that $E|H_2 \preceq E|H_1$.

    For the inductive step, suppose the result holds for all natural numbers $m \leq n$, and assume that $H_1 = \{\theta_{1,1}, \ldots \theta_{1,k}\}$ and $H_2=\{\theta_{2,1} \ldots, \theta_{2,l}\}$ where either $k$ or $l$ (or both) is equal to $n+1$.  Define $H_1'=\{\theta_{1,1}\}$ and $H_1'' = H_1 \setminus H_1'$, and similarly, define $H_2'=\{\theta_{2,1}\}$ and $H_2'' = H_2 \setminus H_2'$.  Then $H_1', H_1'', H_2',$ and $H_2''$ all have fewer than $n$ elements, and by assumption, $E|\theta_2 \sqsubseteq E|\theta_1$ for all $\theta_2 \in H_2', H_2''$ and all $\theta_1 \in H_1', H_1''$.  Suppose, for the moment, that $H_1', H_1'', H_2',$ and $H_2''$ are all non-null.  Then by inductive hypothesis, it would follow that $E|H_1', E|H_1'' \preceq E|H_2', E|H_2''$.  By repeated application of \autoref{lm:qbfl2}, it follows that $E|H_1 = E|(H_1' \cup H_1'') \preceq E|(H_2' \cup H_2'') = E|H_2$.
    
    If any of  $H_1', H_1'', H_2',$ and $H_2''$ are members of $\nul$, the proof involves only a small modification.  Not that since $H_1 \not \in \nul$ and $H_1 = H_1' \cup H_1''$, at least one of $H_1'$ and $H_1''$ must be non-null by Part 4 of \autoref{lm:krantz8}.  Suppose, without lost of generality, that $H_1' \in \nul$ but $H_1'' \not \in \nul$.  Then by Part 6 of \autoref{lm:krantz9}, we get $E|H_1 = E|H_1' \cup H_1'' \sim E|H_1''$ and the inductive hypothesis already applies to $H_1''$ because it contains fewer than $n$ elements.  Similar, reasoning apply to $H_2$.  

    Note that if the premise were $\sqsubset$, \Ax{0} would prove the strict version of the base case. In the inductive step, every inequality would also be strict because \autoref{lm:qbfl2} preserves the strictness. Thus, we would get $E|H_2 \prec E|H_1$.

    In the left-to-right direction, we proceed by contraposition.  Suppose that \autoref{assum} holds and that it's \emph{not} the case that  $E|\theta_2 \sqsubseteq E|\theta_1$ for all $\theta_1 \in H_1$ and $\theta_2 \in H_2$.  We construct an ordering $\preceq$ such that $H_1, H_2 \not \in \nul$ and $E|H_2 \not \preceq E|H_1$.

    To do so, note that because it's \emph{not} the case that  $E|\theta_2 \sqsubseteq E|\theta_1$ for all $\theta_1 \in H_1$ and $\theta_2 \in H_2$, we can fix $\theta_1 \in H_1$ and $\theta_2 \in H_2$ for which $E|\theta_2 \not \sqsubseteq E|\theta_1$.  By totality of $\sqsubseteq$ (which is part of \Ax{1}), it immediately follows that $E|\theta_1 \sqsubset E|\theta_2$. 
    
    Using \autoref{assum} with $H = \{\theta_1,\theta_2\}$, define $\preceq$ so that it satisfies all of the axioms and:
    \begin{eqnarray*}
        \theta_1 |\Delta, \theta_2| \Delta &\succ& \emptyset|\Delta \\
        A|\Delta  &\sim& \emptyset| \Delta \mbox{ if } \theta_1, \theta_2 \not \in A
    \end{eqnarray*}

Because  $\theta_1 |\Delta, \theta_2| \Delta \succ \emptyset|\Delta$, \Ax{2} entails that $\{\theta_1\}, \{\theta_2\} \not \in \nul$. Since $\{\theta_1\}, \{\theta_2\} \not \in \nul$ and $E|\theta_1 \sqsubset E|\theta_2$, by \Ax{0} we obtain that $E|\theta_1 \prec E|\theta_2$.

Now consider $H_1' = H_1 \setminus \{\theta_1\}$ and $H_2' = H_2 \setminus \{\theta_2\}$.  Since $H_1$ and $H_2$ are disjoint, neither $H_1'$ nor $H_2'$ contain either $\theta_1$ or $\theta_2$.  Thus by construction of $\preceq$, it follows that $H_1', H_2' \in \nul$.  So by Part 6 of \autoref{lm:krantz9}, we obtain that 
\begin{align*}
    E|H_1 &= E|H_1' \cup \{\theta_1\} \sim E|\theta_1 \mbox{ and } \\
    E|H_2 &= E|H_2' \cup \{\theta_2\} \sim E|\theta_2
\end{align*}
Since $E|\theta_1 \prec E|\theta_2$, it follows that $E|H_1 \prec E|H_2$, as desired.

If the premise were $\prec$, the conclusion would clearly be $\sqsubset$, by simply swapping all strict and weak inequalities in the preceding argument.
\end{proof}

\section{Posterior and Support Equivalence}

\begin{claim}\label{thm:qfavoring_and_posterior_equivalence}
    $E$ and $F$ are qualitative posterior equivalent if and only if they are qualitative support equivalent.
\end{claim}
\begin{proof}
    In the right to left direction, assume $E$ and $F$ are qualitative support equivalent.  We must show that $E$ and $F$ are posterior equivalent, i.e., for all orderings satisfying the above axioms  (1) $E \in \nul_{\preceq}$ if and only if $F \in \null_{\preceq}$ and (2) $H|E \sim H|F$ for any hypothesis $H \subseteq \Theta$ whenever $E,F \not \in \null_{\preceq}$.   So let $\preceq$ be an arbitrary ordering and $H \subseteq \Theta$ be an arbitrary hypothesis.  Define $H_1 = H$ and $H_2 = \Theta \setminus H$. Notice that $H_1$ and $H_2$ are disjoint and $H_1 \cup H_2 = \Theta$ is the entire space.  We'll now proceed to prove both 1 and 2.
    
    To show 1, note that $E=E \cap (H_1 \cup H_2)$ and $F=F \cap (H_1 \cup H_2)$.  Since $E$ and $F$ are qualitatively support equivalent, it follows that $E \cap (H_1 \cup H_2) \in \nul$ if and only if $F \cap (H_1 \cup H_2) \in \nul$, from which it follows that $E \in \nul$ if and only if $F \in \nul$ as desired. 
    
    To show the 2, again note that by qualitative support equivalence and transitivity, we get
    $$H|E = H_1|E = H_1|E \cap (H_1 \cup H_2) \sim  H_1|F \cap (H_1 \cup H_2) = H_1|F = H|F$$
    
    In the left to right direction, assume that $E$ and $F$ are qualitatively posterior equivalent.  Let $H_1$ and $H_2$ be disjoint hypotheses, and let $\preceq$ be an ordering satisfying the axioms above.  We must show that (1) $E \cap (H_1 \cup H_2) \in \nul_{\preceq}$ if and only if $ F \cap (H_1 \cup H_2)$ and (2) $H_1|E \cap (H_1 \cup H_2) \sim _1|F \cap (H_1 \cup H_2)$ whenever both sides of that expression are well defined.
    
    To show 1, notice that if either $E$ or $F$ is in $\nul$, then because $E$ and $F$ are posterior equivalent, so is the other.  Because $E \cap (H_1 \cup H_2) \subseteq E$ and  $F \cap (H_1 \cup H_2) \subseteq F$, it would then follow from Part 3 of \autoref{lm:krantz8} entails that $E \cap (H_1 \cup H_2), F \cap (H_1 \cup H_2) \in \nul$.  Thus, if  either $E$ or $F$ is in $\nul$, then $E \cap (H_1 \cup H_2)$ if and only if $F \cap (H_1 \cup H_2) \in \nul$.
    
    So consider the case in which neither $E$ nor $F$ is a member of $\nul$, and thus, the expressions $\cdot|E$ and $\cdot|F$ are well-defined.  We'll show that if $E \cap (H_1 \cup H_2) \in \nul$, then $F \cap (H_1 \cup H_2)$; the converse is proved identically.  If $E \cap (H_1 \cup H_2) \in \nul$, then both $E \cap H_1 \in \nul$ and $E \cap H_2 \in \nul$ by Part 3 of \autoref{lm:krantz8} since $E \cap H_1, E\cap H_2 \subseteq E \cap (H_1 \cup H_2)$.  Therefore:
    \begin{alignat*}{3}
        &\emptyset|E &&\sim  &&E \cap H_1|E  \mbox{ by Part 2 of \autoref{lm:krantz9} and as } E \cap H_1 \in \nul \\
        &~&&\sim  && H_1|E  \mbox{ by  \Ax{4}} \\
        &~ &&\sim  && H_1|F  \mbox{ since } E \& F \mbox{ are posterior equivalent and } E, F \not \in \nul \\
        &~  &&\sim  && F \cap H_1|F  \mbox{ by \Ax{4} }
    \end{alignat*}
    Thus, $\emptyset|E \sim F \cap H_1|F$ and so $F \cap H_1 \in \nul$ by Part 2 of \autoref{lm:krantz9}.  By identical reasoning, $F \cap H_2 \in \nul$.  By Part 4 of \autoref{lm:krantz8}, it follows that  $F \cap (H_1 \cup H_2) \in \nul$ as desired.
    
    Now we show 2, i.e., if $E \cap (H_1 \cup H_2), F \cap (H_1 \cup H_2) \not \in \nul$, then $H_1|E \cap (H_1 \cup H_2) \sim  H_1|F \cap (H_1 \cup H_2)$.  Suppose, for the sake of contradiction, that $H_1|E \cap (H_1 \cup H_2) \not \sim  H_1|F \cap (H_1 \cup H_2)$.
    By \Ax{1}, we obtain that either $H_1|E \cap (H_1 \cup H_2) \prec  H_1|F \cap (H_1 \cup H_2)$ or vice versa.  Without loss of generality, assume 
    \begin{equation}\label{eq:qfap_1}
    H_1|E \cap (H_1 \cup H_2) \prec  H_1|F \cap (H_1 \cup H_2)
    \end{equation}
    Since $E \cap (H_1 \cup H_2), F \cap (H_1 \cup H_2) \not \in \nul$, it follows from Part 3 of \autoref{lm:krantz8} that $E, F \not \in \nul$.  Because $E$ and $F$ are posterior equivalent, we know that
    \begin{equation}\label{eq:qfap_2}
    H_1 \cup H_2|E  \sim  H_1 \cup H_2|F.
    \end{equation}
    Define:
    \begin{center}
    \begin{tabular}{l c l}
    $A=F$   & $\quad$ & $A'=E$ \\
    $B=F \cap (H_1 \cup H_2)$   &$\quad$  & $B'=E \cap (H_1 \cup H_2)$ \\ 
    $C=F \cap H_1$   & $\quad$  &$C'=E \cap H_1$ \\
    \end{tabular}
    \end{center}
    Thus:
    \begin{alignat*}{3}
    B|A &= F \cap (H_1 \cup H_2) | F & \\
        &\sim H_1 \cup H_2 | F & \mbox{ by \Ax{4}} \\
        &\sim H_1 \cup H_2 | E & \mbox{ by \autoref{eq:qfap_2}}  \\
        &\sim E \cap (H_1 \cup H_2) | E & \mbox{ by \Ax{4}}  \\
        &= B'|A' &
    \end{alignat*}
    Further:
    \begin{alignat*}{3}
    C|B &= F \cap H_1 | F \cap (H_1 \cup H_2) & \\
        &= H_1 \cap (F \cap (H_1 \cup H_2)) | F \cap (H_1 \cup H_2) & \\
        &\sim H_1 | F \cap (H_1 \cup H_2) & \mbox{ by \Ax{4}} \\
        &\succ H_1  |  E \cap (H_1 \cup H_2)  & \mbox{ by \autoref{eq:qfap_1}}  \\
        &\sim C'|B' & \mbox{ reversing the above steps}  
    \end{alignat*}
    Note that since $C|B \succ C'|B'$, we know that $C \notin \nul$. Further, since $B|A \sim B'|A'$, \Ax{6b} entails $C|A \succ C'|A'$, i.e., that $F \cap H_1 | F \succ E \cap H_1 | E$.  But then \Ax{4} entails $H_1 | F \succ H_1 | E$, and that contradicts the assumption that $E$ and $F$ are qualitative posterior equivalent. Thus, by contradiction, $E$ and $F$ are qualitative support equivalent.
\end{proof}

\section{A Qualitative Likelihood Principle}
In this section, we prove our second major result, which is a qualitative analog of \autoref{clm:bLPPosteriorEquivalence}. Specifically, we prove:

\begin{theorem}
\label{thm:qmsp}
       Let $\omega_1$ and $\omega_2$ be outcomes of experiments $\mathbb{E}$ and $\mathbb{F}$ respectively.  Then the following statements are equivalent:
       
        \begin{enumerate}
            \item In every mixture $\mathbb{M}$ of $\mathbb{E}$ and $\mathbb{F}$, there is a sufficient statistic $T$ such that $T(\omega_1) = T(\omega_2)$.
            \item  $\omega_1$ and $\omega_2$ are $\mathcal{U}$-posterior equivalent, where $\mathcal{U}$ is the (universal) set of all orderings satisfying our axioms for qualitative probability.
        \end{enumerate}
\end{theorem}

Why is this theorem the analog of \autoref{clm:bLPPosteriorEquivalence}?  Recall, \autoref{clm:bLPPosteriorEquivalence} says that \f{lp} characterizes when two pieces of evidence are posterior and Bayesian-posterior equivalent.   Condition 1 in the above statement is equivalent to \f{lp} in the quantitative setting, and so \autoref{thm:qmsp} is the most natural analog of \autoref{clm:bLPPosteriorEquivalence}.

The definitions of ``mixture'' and ``sufficient statistic'' can be found in the main document.  To prove the theorem, we first need to amass a few facts about sufficient statistics.

\subsection{Sufficiency}

\begin{lemma}\label{lm:qSufficientZeroProb}
  If $T$ is sufficient and $T(\omega)=T(\omega')$, then  for all $\theta$
  \[ \omega|\theta\equiv \emptyset \mbox{ if and only if } \omega'|\theta \equiv \emptyset\]
\end{lemma}
\begin{proof}
    Suppose for the sake of contradiction that there is some $\theta$ such that 
    \begin{equation}\label{eqn:lmqSZP}
        \omega'|\theta \sqsupset \emptyset \equiv \omega|\theta.
    \end{equation}
    We will show that $\omega|\upsilon \equiv \emptyset$ for all $\upsilon$.  Hence, by the definition of ``sufficient'' and our assumption  that $T(\omega)=T(\omega')$, it follows that $\omega'|\upsilon \equiv \emptyset$ for all $\upsilon$.  But that contradicts \autoref{eqn:lmqSZP}.
    
    To show that  $\omega|\upsilon \equiv \emptyset$ for all $\upsilon$, note first that $\omega' \in \{T=T(\omega)\}$ because $T(\omega)=T(\omega')$. Next, since $\omega'|\theta \sqsupset \emptyset$ and $\omega' \in \{T=T(\omega)\}$, \autoref{lm:krantz7} entails that $\{T=T(\omega)\}|\theta \sqsupset \emptyset$. Thus, $\cdot|\{T=T(\omega)\}, \theta$ is well-defined. Since $\omega|\theta \equiv \emptyset$, it follows from \autoref{lm:krantz9} (part 1) that $\omega|\{T=T(\omega)\}, \theta \equiv \emptyset$. Now, let $\upsilon$ be arbitrary. We go by cases.
    
    \paragraph{Case 1:} Suppose that $\{T=T(\omega)\}| \upsilon \sqsupset \emptyset$. Then, by the definition of ``sufficient'', it follows that  $\omega|\{T=T(\omega)\}, \upsilon \equiv \emptyset$.
    Also, note that $\omega|\{T\neq T(\omega)\}, \upsilon \equiv \emptyset$, since $\omega \notin \{T\neq T(\omega)\}$. Now, we apply \autoref{lm:qConditionalProbabilityConstantOnPartition}. To do so, let $X_1 = \omega \cap \{T=T(\omega)\}, X_2 = \omega \cap \{T\neq T(\omega)\}, Y = \upsilon, Z = \{T=T(\omega)\}$, and $G = \Delta$. The three conditions of the lemma simplify to:
    \begin{enumerate}
        \item $\omega | \{T=T(\omega)\}, \upsilon \sim \omega | \{T\neq T(\omega)\}, \upsilon$. This is true, as we showed that both expressions are equiprobable with $\emptyset$.
        \item $\omega \cap \{T=T(\omega)\} | \{T\neq T(\omega)\} \sim \emptyset$. This trivially follows from \Ax{4}.
        \item $\omega \cap \{T \neq T(\omega)\} | \{T = T(\omega)\} \sim \emptyset$. This also follows from \Ax{4}.
    \end{enumerate}
    Since the conditions are satisfied, the lemma implies that $\omega|\upsilon \sim \emptyset$. By \Ax{0}, we get that $\omega|\upsilon \equiv \emptyset$.
    
    \paragraph{Case 2:} Suppose that $\{T=T(\omega)\}| \upsilon \equiv \emptyset$. It immediately follows that $\omega|\upsilon \equiv \emptyset$ since $\omega \in \{T=T(\omega)\}$.

    Thus, we've shown that  $\omega|\upsilon \equiv \emptyset$ for all $\upsilon$, as desired.
\end{proof}

\setcounter{lemma}{8}
\begin{lemma}\label{lm:qSufficientStatisticExamples}
Suppose $\mathbb{E}$ is an experiment with a finite set of outcomes $\Omega$.  Then the following are sufficient statistics.
    \begin{enumerate}
        \item The identity map $T:\Omega \rightarrow \Omega$. 
        
        \item Let $S$ be an equivalence relation such that $S(\omega, \omega')$ holds precisely when $\omega | \theta \equiv \omega'|\theta$ for all $\theta$. Let
        $T:\Omega \rightarrow \Omega/S$ be the quotient map $\omega \mapsto [\omega]_S$ that takes each outcome to its $S$-equivalence class.

    \end{enumerate}
\end{lemma}
\setcounter{lemma}{19}
\begin{proof}
~[1.] If $t  \neq \omega$, then for all $\theta$ we have:
\begin{align*}
    \omega ~|~ \{T=t\}, \theta 
    &\equiv \omega \cap  \{T=t\} ~|~ \{T=t\},\theta  && \mbox{by \Ax{4}} \\
    &= \emptyset ~|~ \{T=t\}, \theta  && \mbox{since } t \neq \omega \mbox{ and } T \mbox{ is the identity map} \\
    &= \emptyset ~|~ \Delta && \mbox{by \autoref{lm:krantz6-2} }  
\end{align*}
Note the last line does not depend on $\theta$. If $t  = \omega$, then for all $\theta$ we have:
\begin{align*}
    \omega ~|~ \{T=t\}, \theta 
    &\equiv \omega \cap  \{T=t\} \cap \theta ~|~ \{T=t\},\theta  && \mbox{by \Ax{4}} \\
    &=  \{T=t\} \cap \theta ~|~ \{T=t\}, \theta  && \mbox{since } \{T = t\} = \{\omega\} \\
    &= \Delta ~|~ \Delta && \mbox{by \Ax{3} }  
\end{align*}
Part 2 of this lemma is proven below because the proof is longer, and thus broken up into a series of shorter lemma. The final statement is just \autoref{lm:qLikelihoodFunctionSufficient} below.
\end{proof}

\begin{lemma}\label{lm:qLFS1} If $A|C \sim B|C$ and $A, B \subseteq D$, then $A|C \cap D \sim B| C \cap D$.  And similarly if $\sim$ is replaced by $\equiv$.  
\end{lemma}
\begin{proof}
    Suppose for the sake of contradiction that  $A|C \cap D \not \sim B| C \cap D$.  Then by \Ax{1} (specifically, totality of the ordering), either $A|C \cap D \prec B| C \cap D$ or vice versa.  Without loss of generality, assume $A|C \cap D \prec B| C \cap D$.  Define:
    \begin{align*}
            &X = X' = C       \\
            &Y= Y' = C \cap D\\
            &Z= A \cap C \cap D && Z' = B \cap C \cap D 
        \end{align*}
    Then $Z, Z' \subseteq Y=Y' \subseteq X=X'$.  Clearly $Y|X\equiv Y'|X'$ as $Y|X=Y'|X'$.  Further:
    \begin{eqnarray*}
    Z|Y &=&   A \cap C \cap D |  C \cap D \\
     	&\sim& A | C \cap D \mbox{ by  \Ax{4} } \\
     &\prec&  B| C \cap D \\
     &\sim&  B \cap C \cap D| C \cap D \mbox{ by  \Ax{4} }\\
     &=& Z'|Y'
    \end{eqnarray*}
    So by \Ax{6b}, it follows that $Z|X \prec Z'|X'$, i.e., that 
    \[ A \cap C \cap D|C \prec B \cap C \cap D| C \] 
    By \Ax{4}, it follows that $A \cap D|C \prec B \cap D|C$. Finally, because $A, B \subseteq D$, we know $A \cap D =A$ and $B \cap D =D$.  Thus, we obtain $A|C \prec B|C$, contradicting our initial assumption.
\end{proof}

\begin{lemma}\label{lm:qLFS2} If $[\omega]_S = [\omega']_S$, then $\omega'~|~[\omega]_S \cap \theta\ \equiv \omega~|~[\omega]_S \cap \theta$  for all $\theta$.
\end{lemma}
\begin{proof}
    
We apply \autoref{lm:qLFS1}.  To do so, define:
    \begin{eqnarray*}
        A &=&\{\omega\} \\
        B &=&\{\omega'\} \\
        C &=&  \{\theta\}\\
        D &=& [\omega]_S
    \end{eqnarray*}
Since $[\omega]_S = [\omega']_S$, it follows that $A|C = \omega|\theta \equiv \omega'|\theta = B|C.$. Moreover, since $\omega \in [\omega]_S$, we have $A \subseteq D$.  Similarly, since $[\omega]_S = [\omega']_S$ and $\omega' \in [\omega]_S$, and thus, $B \subseteq D$. Applying \autoref{lm:qLFS1}, we obtain $A| C\cap D \equiv B |C \cap D$, i.e., that $\omega'~|~[\omega]_S \cap \theta\ \equiv \omega~|~[\omega]_S \cap \theta$.  Since $\theta$ was arbitrary, the result follows.
\end{proof}

\begin{lemma}\label{lm:qLikelihoodFunctionSufficient}  Let $T:\omega \mapsto [\omega]_S = \{\omega' \in  \Omega_{\mathbb{E}}:  \omega'|\theta\equiv \omega|\theta \mbox{ for all } \theta \in \Theta\}$. Then $T$ is sufficient.
\end{lemma}
\begin{proof}
    We must show that $\omega|T=t, \theta\equiv \omega|T=t, \upsilon$  for all $\omega$ and for all $\theta, \upsilon \in \Theta$, where $T:\omega \mapsto [\omega]_S$.    So let $\theta, \upsilon$ be arbitrary.

Consider first the case in which $t \neq [\omega]_S$, and so $\omega \cap \{T=t\} = \emptyset$.  Then  $\omega|T=t, \theta \equiv \omega \cap \{T=t\} \cap \{\theta\} |T=t, \theta$ by \Ax{4}, and so   $\omega|T=t, \theta \equiv \emptyset | T=t, \theta$.   Similarly,  $\omega|T=t, \upsilon \equiv \emptyset | T=t, \upsilon$.  It suffices to show only that $\emptyset|A \equiv \emptyset | B$ for all $A, B$ (when defined), and this is precisely what \autoref{lm:krantz6-2} asserts.

Next, consider the case in which $t = [\omega]_S$, and so $\{T=t\} = [\omega]_S$.  Then we must show that  $\omega|[\omega]_S, \theta\equiv \omega|[\omega]_S, \upsilon$. Suppose for the sake of contradiction not.  Then by \Ax{1} (specifically, totality of $\sqsubseteq$), it follows that  $\omega|[\omega]_S, \theta \sqsubset \omega|[\omega]_S, \upsilon$ or vice versa.  Without loss of generality, assume $\omega|[\omega]_S, \theta \sqsubset \omega|[\omega]_S, \upsilon$.

Next, define $R_{\omega} = [\omega]_S \setminus \{\omega\}$.  We claim there is some $\omega' \in R_{\omega}$ such that either (1) $\omega'|[\omega]_S, \theta \not \equiv \omega|[\omega]_S, \theta$ or (2) $\omega'|[\omega]_S, \upsilon \not \equiv \omega|[\omega]_S, \upsilon$.  If neither (1) nor (2) holds, then we have $\omega'|[\omega]_S, \theta \equiv \omega|[\omega]_S, \theta$ AND $\omega'|[\omega]_S, \upsilon \equiv \omega|[\omega]_S, \upsilon$ for all $\omega' \in R_{\omega}$.  So by repeatedly applying \Ax{5} (recall, we've assumed $\Omega_{\mathbb{E}}$ is finite) to the assumption that $\omega|[\omega]_S, \theta \sqsubset \omega|[\omega]_S, \upsilon$, we obtain that $R_{\omega}|[\omega]_S,\theta \sqsubset R_{\omega}|[\omega]_S,\upsilon$.  But then because $R_{\omega}|[\omega]_S,\theta \sqsubset R_{\omega}|[\omega]_S,\upsilon$ and $\omega|[\omega]_S, \theta \sqsubset \omega|[\omega]_S, \upsilon$, one last application of \Ax{5} yields that  $[\omega]_S|[\omega]_S, \theta \sqsubset [\omega]_S|[\omega]_S, \upsilon$.  By \Ax{4}, it follows that $[\omega]_S,\theta|[\omega]_S, \theta \sqsubset [\omega]_S, \upsilon |[\omega]_S, \upsilon$.  But that contradicts \Ax{3}.

So either (1) or (2) obtains.  But notice that both contradict lemma \ref{lm:qLFS2} as $\omega' \in R_{\omega} \subseteq	 [\omega]_S$.
\end{proof}

\subsection{QLP:  The Single Experiment Case}
To prove \autoref{thm:qmsp}, we begin by proving the theorem in the the case in which $\mathbb{E}=\mathbb{F}$, i.e., when $\omega_1$ and $\omega_2$ are outcomes of the same experiment.

\begin{proposition}\label{prop:qSufficientWeakPosteriorEquivalence}
Suppose $\Theta$ is finite and that $\omega_1, \omega_2 \in \Omega^{\mathbb{E}}$ are outcomes of the same experiment.  Then the following two claims are equivalent:
\begin{enumerate}
    \item  $\omega_1$ and $\omega_2$ are ${\cal U}$-posterior equivalent.
    \item There is a sufficient statistic $T$ such that $T(\omega_1)=T(\omega_2)$.
\end{enumerate}
\end{proposition}

The proof of \autoref{prop:qSufficientWeakPosteriorEquivalence} requires three more technical lemmata.

\begin{lemma}\label{lm:qConditionalProbabilityConstantOnPartition}
    Suppose $X = X_1 \cup X_2$, where $X_1 \cap X_2 = \emptyset$. Suppose $X_1|Y \cap G \cap Z \sim X_2|Y \cap (G \setminus Z)$.  Further, suppose that both 
       \begin{enumerate}[(i)]
           \item $X_1| (G \setminus Z) \sim \emptyset|Z$, and
           \item $X_2| G \cap Z \sim \emptyset|Z$. 
       \end{enumerate}
    
    Then $X|Y \cap G \sim X_1|Y \cap G \cap Z \sim X_2|Y \cap (G \setminus Z)$. 
\end{lemma}
\begin{proof}
    First, suppose that $X \cap Y \cap G \in \nul$. Then, $X|Y \cap G \sim \emptyset|Y \cap G$. We also claim that $X_1|Y \cap G \cap Z \sim \emptyset|Y \cap G \cap Z$, which gives us the desired result by \autoref{lm:krantz6-2}. Suppose this wasn't the case -- since $Y \cap G \cap Z \notin \nul$, we know that $X_1|Y \cap G \cap Z \succ \emptyset|Y \cap G \cap Z$. We can use \Ax{5} to add this with $X_2|Y \cap G \cap Z \succeq \emptyset|Y \cap G \cap Z$ (\autoref{lm:krantz6-1}), getting $X|Y \cap G \cap Z \succ \emptyset|Y \cap G \cap Z$. Note that by \autoref{lm:krantz9}, that since $X \cap Y \cap G \in \nul$ we know that $X|Y \cap G \cap Z \sim \emptyset|Y \cap G \cap Z$, we thus get the desired contradiction.

    We now assume $X \cap Y \cap G \notin \nul$. By applying Part 2 of \autoref{lm:krantz9} (since $Y \cap (G \setminus Z) \notin \nul$) and \autoref{lm:krantz6-2} to (i), we get that $X_1|Y \cap (G \setminus Z) \sim \emptyset|Y \cap G \cap Z$. Now, we apply \Ax{5} as follows:
    \begin{align}
        \emptyset | Y \cap G \cap Z &\sim X_1|Y \cap (G \setminus Z) \label{eq:qCond11} \\
        X_1 | Y \cap G \cap Z &\sim X_2 | Y \cap (G \setminus Z) & \text{Given} \label{eq:qCond12} \\
        X_1 | Y \cap G \cap Z &\sim X_1 \cup X_2 | Y \cap (G \setminus Z) &\text{\Ax{5} on \autoref{eq:qCond11} and \autoref{eq:qCond12}} \label{eq:qCond1}
    \end{align}
    We apply the same steps to (ii). Using Part 2 of \autoref{lm:krantz9} (since $Y \cap G \cap Z \notin \nul$) and \autoref{lm:krantz6-2}, we get $X_2 | Y \cap G \cap Z \sim \emptyset|Y \cap (G \setminus Z)$. Then, we apply \Ax{5}:
    \begin{align}
        X_2 | Y \cap G \cap Z           &\sim \emptyset|Y \cap (G \setminus Z) \label{eq:qCond21} \\
        X_1 | Y \cap G \cap Z           &\sim X_2 | Y \cap (G \setminus Z) & \text{Given} \label{eq:qCond22} \\
        X_1 \cup X_2 | Y \cap G \cap Z  &\sim X_2 | Y \cap (G \setminus Z) &\text{\Ax{5} on  \autoref{eq:qCond21} and \autoref{eq:qCond22}} \label{eq:qCond2}
    \end{align}
    Because $X_1|Y \cap G \cap Z \sim X_2|Y \cap (G \setminus Z)$, from \autoref{eq:qCond1} and \autoref{eq:qCond2} we get $X|Y \cap G \cap Z \sim X|Y \cap (G \setminus Z)$. 
    
    Now, suppose for the sake of contradiction that $X|Y \cap G \not \sim X|Y \cap G \cap Z$. We know that $Y \cap G \cap Z \notin \nul$ by the premise, which also means (via \autoref{lm:krantz8}) that $Y \cap G \notin \nul$. Thus, either $X|Y \cap G \succ X|Y \cap G \cap Z$ or vice versa; without loss of generality, assume $X|Y \cap G \succ X|Y \cap G \cap Z$.    
    Now, we apply \autoref{lm:Aditya} (since we are assuming $X \cap Y \cap G \notin \nul$) with:
    \begin{alignat*}{3}
        &A = Y \cap G & & \\ 
        &B = X \cap Y \cap G &\quad \quad   &B' = Y \cap G \cap Z \\
        &C = X \cap Y \cap G \cap Z & &
    \end{alignat*}
    Note that the necessary containment relations hold. Since $B|A \succ C|B'$, we get $B'|A \succ C|B$, or $Y \cap G \cap Z|Y \cap G \succ X \cap Y \cap G \cap Z|X \cap Y \cap G$. After applying \Ax{4}, we get:
    \begin{equation}\label{eq:qLm12-3}
        Z|Y \cap G \succ Z|X \cap Y \cap G
    \end{equation}

    Since $X|Y \cap G \cap Z \sim X|Y \cap (G \setminus Z)$, we know from our contradicting assumption that $X|Y \cap G \succ X|Y \cap (G \setminus Z)$. We now proceed as before, applying \autoref{lm:Aditya} with:
    \begin{alignat*}{3}
        &A = Y \cap G & & \\ 
        &B = X \cap Y \cap G &\quad \quad   &B' = Y \cap (G \setminus Z) \\
        &C = X \cap Y \cap (G \setminus Z) & &
    \end{alignat*}
    Note that the necessary containment relations hold. Since $B|A \succ C|B'$, we get $B'|A \succ C|B$, or (after \Ax{4}):
    \begin{equation} \label{eq:qLm12-4}
    G\setminus Z | Y \cap G \succ G \setminus Z|X \cap Y \cap G
    \end{equation}    
    Applying \Ax{5} to \autoref{eq:qLm12-3} and \autoref{eq:qLm12-4} gives us $G|Y \cap G \succ G|X \cap Y \cap G$. Applying \Ax{4} gives us $\Delta \succ \Delta$, which is a contradiction.
\end{proof}

When trying to grasp the motivation for the steps in the proofs immediately below, the reader may rely on her intuitions about how the standard, quantitative definition of conditional independence behaves.

\begin{lemma}\label{lm:QConditionalIndependenceIsSymmetric}
Conditional independence is symmetric.  In other words, if  $A \ind_C B$, then  $B \ind_C A$.  As special case, when $C=\Delta$, if $A|B \sim A|\Delta$ and $B \not \in \nul$, then $B|A \sim A|\Delta$.
\end{lemma}
\begin{proof}
    From $A \ind_C B$, we get that $B \cap C \in \nul$ or that $A|B \cap C \sim A|C$. Notice that if $A \cap C \in \nul$, we're done. So, suppose that this is not the case. We proceed by cases. First, suppose that $B \cap C \in \nul$. Then, we get that $B \cap C|C \sim \emptyset$ (\autoref{lm:krantz9}). Similarly, we get that $A \cap B\cap C|A \cap C \sim \emptyset$. Thus, by \Ax{4} and transitivity, we get that $B|C \sim B|A \cap C$, which shows that $B \ind_C A$.

    Now, suppose we are the in the second case, where $A|B \cap C \sim A|C$. Then, we apply \autoref{lm:Aditya} with:
    \begin{align*}
       \mathfrak{A} &= C \\
       \mathfrak{B} &= A \cap C, &&\mathfrak{B}' = B \cap C \\
       \mathfrak{C} &= A \cap B \cap C
    \end{align*}
    By \Ax{4} and assumption that $A|B \cap C \sim A|C$, we have $\mathfrak{B}|\mathfrak{A} \sim \mathfrak{C}|\mathfrak{B}'$. Thus, the lemma tells us that $\mathcal{B}'|\mathfrak{A} \sim \mathfrak{C}|\mathfrak{B}$, i.e., that $B|C \sim B|A \cap C$. This shows us that $B \ind_C A$, as desired.
\end{proof}

\begin{lemma}\label{lm:qLPEntailsPosteriorEquivalence2}
    Suppose $B_1, \ldots, B_n$ partition $G$, and that $A \cap B_i \notin \nul$, for all $i$. Further, suppose $A \ind_{B_i} C_i$ with respect to $\preceq$ and for all $i, j \le n: C_i | B_i \sim C_j | B_j$. Then, for all $k \le n$,
    $$\bigcup_{i\le n} C_i \cap B_i|A \cap G \sim C_k |B_k$$
\end{lemma}
\begin{proof}
    Since $A \ind_{B_i} C_i$ for all $i \leq n$, we have by \autoref{lm:QConditionalIndependenceIsSymmetric} that $C_i \ind_{B_i} A$.  We now proceed 
    by induction on $n$. When $n = 1$, note that $G = B_1$. So: 
    \begin{eqnarray*} 
        \bigcup_{i\le n} C_1 \cap B_i|A \cap G &=& C_1 \cap G|A \cap G \\
        &\sim& C_1 | A \cap G \mbox{ by \Ax{4}} \\
        &\sim&  C_1 | G \mbox{ because } C_1 \ind_G A
    \end{eqnarray*}
    For the inductive step, suppose $B_1, \ldots, B_{n+1}$ partition $G$, and define $G' = \{B_1, \ldots, B_n\}$. Let $i$ be arbitrary.
    
    Note that the conditions of the inductive hypothesis are satisfied with respect to $G'$ and $B_1, \ldots, B_n$. Thus, we get:
    \begin{equation*}
        \bigcup_{j \le n} C_j \cap B_j|A \cap G' \sim C_i|B_i
    \end{equation*}

    We also know that:
    \begin{align*}
        C_{n+1} \cap B_{n+1} | A \cap B_{n+1} &\sim C_{n+1} | A \cap B_{n+1} &\text{by \Ax{4}}\\
                                            &\sim C_{n+1} | B_{n+1} &\text{since $A \ind_{B_i} C_i$ for all $i$} \\
                                            &\sim C_{i} | B_i &\text{since $C_i | B_i \sim C_j | B_j$ for all $i, j \le n+1$}
    \end{align*}
    So we've shown that (1) $\bigcup_{j \le n} C_j \cap B_j|A \cap G' \sim C_i|B_i$ and (2) $C_{n+1} \cap B_{n+1} | A \cap B_{n+1} \sim C_{i} | B_{i}$. We now apply \autoref{lm:qConditionalProbabilityConstantOnPartition} with $X_1 = \bigcup_{j \le n} C_j \cap B_j$, $X_2 = C_{n+1} \cap B_{n+1}$, $Y = A$, $Z = G'$ and $G = G$. We have shown (from (1) and (2)) that $X_1 | Y \cap G \cap Z \sim X_2 | Y \cap (G \setminus Z)$. Note that $X_1 | (G \setminus Z) \sim \emptyset|Z$ by \Ax{4}, as $X_1 \cap (G \setminus Z) = \emptyset$. Similarly, $X_2 | G \cap Z \sim \emptyset$. With the conditions satisfied, we get $X|Y \cap G \sim X_1|Y \cap G \cap Z$ or $\bigcup_{j \le n+1} C_j \cap B_j|A \cap G \sim \bigcup_{j \le n} C_j \cap B_j|A \cap G'$. And from (1), we get $\bigcup_{j \le n+1} C_j \cap B_j|A \cap G \sim C_i|B_i$, as desired.
\end{proof}

Now we prove \autoref{prop:qSufficientWeakPosteriorEquivalence}.

\begin{proof}
\noindent [Proof of 1 $\Rightarrow$ 2:] For all $\omega$, define:
\[ [\omega] = \{\omega' \in \Omega^{\mathbb{E}}: \omega \mbox{ and } \omega' \mbox{ are } {\cal U} -\mbox{posterior equivalent}\} \]
Define $T$ to be the map $\omega \mapsto [\omega]$ that takes every outcome to its equivalence class.  Clearly, $T(\omega_1)=T(\omega_2)$ if $\omega_1$ and $\omega_2$ are posterior equivalent.  So it remains only to be shown that $T$ is sufficient.

For any $\omega \in \Omega$, we define:
\begin{eqnarray*}
\Theta(\omega,0) :=\{\theta \in \Theta: \omega|\theta \equiv \emptyset\} \\
\Theta(\omega,T>0) := \{\theta \in \Theta: \{T=T(\omega)\}|\theta \sqsupset \emptyset\} 
\end{eqnarray*}

We begin by verifying that $T$ satisfies the second condition in the definition of ``sufficient'', i.e., we must show that if $\omega|\theta \equiv \emptyset$ for all $\theta$ and $T(\omega)=T(\omega')$, then $\omega'|\theta \equiv \emptyset$ for all $\theta$.  

In fact, we will show something strictly stronger then the second condition. We will show that if $T(\omega)=T(\omega')$, then 
\begin{equation}\label{eqn:qSWPE0}
 \omega|\theta \equiv \emptyset \mbox{ if and only if } \omega'|\theta \equiv \emptyset \mbox{ for all } \theta \in \Theta.
\end{equation}
In other words, we show  that $\Theta(\omega,0)= \Theta(\omega',0)$.  If both $\Theta(\omega,0)$ and $\Theta(\omega',0)$ are empty, then they're clearly equal.  So suppose $\Theta(\omega,0)$ is non-empty.  Using \autoref{assum}, there is an ordering $\preceq$ such $\theta \succ \emptyset$ if and only if $\theta \in \Theta(\omega,0)$.  Since $\mathcal{U}$ is the set of \emph{all} orderings, we know $\preceq \in \mathcal{U}$.

We claim that $\omega \in \nul^\preceq$. By the definition of $\Theta(\omega, 0)$, for all $\theta \in \Theta(\omega, 0)$ we know that $\omega | \theta \sim \emptyset$.  By \autoref{lm:qConditionalIntersection} we get that $\omega \cap \theta \sim \emptyset$  for all $\theta \in \Theta(\omega, 0)$. For $\theta \notin \Theta(\omega, 0)$, we know that $\theta \sim \emptyset$ by the definition of the prior $\preceq$.
It follows that $\omega \cap \theta \sim \emptyset$ for all $\theta$.  Because $\Theta$ is finite, repeatedly applying \Ax{5} yields $\omega \sim \emptyset$, i.e., that $\omega \in \nul^{\preceq}$.


Now because $\omega$ and $\omega'$ are posterior equivalent (as $T(\omega)=T(\omega')$) and $\omega \in \nul^{\preceq}$, we obtain that $\omega' \in \nul^{\preceq}$.  Hence, $\omega'|\theta \sim \emptyset$ for all $\theta \in \Theta(\omega,0)$.  
By \Ax{0}, it follows that that $\omega'|\theta \equiv \emptyset$ for all $\theta \in \Theta(\omega,0)$.  Thus, $\Theta(\omega,0) \subseteq \Theta(\omega', 0)$.  The reverse inclusion follows from symmetry.

Next, we verify that $T$ satisfies first condition in the definition of ``sufficient'', i.e., we show that for all $\omega$:
\begin{equation}\label{eqn:qSWEP1}
   \omega|\theta, \{T=T(\omega)\} \equiv \omega|\eta , \{T=T(\omega)\} \mbox{ for all } \theta,\eta \in \Theta(\omega, T>0)
\end{equation}

If $\Theta(\omega, T>0)$ is empty, then \autoref{eqn:qSWEP1} is satisfied vacuously.  So suppose that $\Theta(\omega, T>0)$ is non-empty.   Using \autoref{assum}, there is an ordering $\preceq$ such that $\Theta(\omega,T>0)\sim \Delta$ and $\theta \succ \emptyset$ for all $\theta \in \Theta(\omega,T>0)$.   By \Ax{0}, it suffices to show that  \autoref{eqn:qSWEP1} holds when $\equiv$ is replaced by $\sim$.

To derive  \autoref{eqn:qSWEP1} (with $\equiv$ replaced by $\sim$), it suffices to show:
\begin{equation}\label{eqn:qSWEP2}
   \omega|\theta, \{T=T(\omega)\} \sim \omega| \{T=T(\omega)\}  \mbox{ for all } \theta \in \Theta(\omega, T>0).
\end{equation}
This suffices because the right-hand side of that equation does not depend on $\theta$.

To derive \autoref{eqn:qSWEP2}, we proceed in three steps. 

First, we show that if  $\Theta(\omega, T>0) \neq \emptyset$, then $\omega \not \in \nul^{\preceq}$, and hence, $\cdot | \omega$ is well-defined. Suppose for the sake of contradiction that $\omega \in \nul^{\preceq}$.  Then, given our choice of $\preceq$, it must be the case that $\omega|\theta \equiv \emptyset$ for all $\theta \in \Theta(\omega, T>0)$.  Fix any such $\theta$.  By \autoref{eqn:qSWPE0} , it follows that $\omega'|\theta \equiv \emptyset$ for all $\omega' \in \{T=T(\omega)\}$.  Since $\{T=T(\omega)\}$ is finite, we can repeatedly apply \Ax{5} to show that $\{T=T(\omega)\}|\theta \equiv \emptyset$, which is a contradiction since $\theta \in \Theta(\omega, T>0)$.

Second, we show that
\begin{equation}\label{eqn:qSWEP3}
   \theta|\omega \sim \theta|\{T=T(\omega)\} \mbox{ for all } \theta \in \Theta(\omega, T>0)
\end{equation}
Why? We just showed  that $\omega \not \in \nul^{\preceq}$, and hence, $\cdot | \omega^*$ is well-defined for all $\omega^* \in \{T=T(\omega)\}$.  Thus, by definition of $T$ and posterior equivalence, we have $\theta|\omega \sim \theta|\omega'$ for all $\omega' \in [\omega]$.  By repeatedly applying \autoref{lm:qbfl1}, we obtain that $\theta|\omega \sim \theta|[\omega]$.  By definition of $T$, we have $\{T=T(\omega)\} = [\omega]$, and so we've shown that $\theta|\omega \sim \theta|\{T=T(\omega)\}$. 

Third and finally, we use \autoref{eqn:qSWEP3} to derive \autoref{eqn:qSWEP2} as follows.  Define:
  \begin{alignat*}{3}
        &A = A' = \{T= T(\omega)\}  \\
        & B = \{T= T(\omega)\} \cap \theta & \quad \quad & B'= \omega \\
        & C = C' =  \omega \cap \theta
    \end{alignat*}
By definition, $C|A = C'|A'$.  Further, \autoref{eqn:qSWEP3} (with \Ax{4}) entails that $C'|B' \sim B|A$.  Thus, \Ax{6} entails that $B'|A' \sim C|B$, i.e., that $\omega|\{T= T(\omega)\} \sim \omega \cap \theta| \{T= T(\omega)\} , \theta$.  Applying \Ax{4} to the right hand side of that equation yields \autoref{eqn:qSWEP2}, as desired.\\

\noindent [Proof of 2 $\Rightarrow$ 1:]  Let $T$ be some sufficient statistic such that $T(\omega_1)=T(\omega_2)$.  We must show that $\omega_1$ and $\omega_2$ are posterior equivalent.  So let $\preceq$ be any ordering satisfying our axioms for qualitative probability.

We first show that $\omega_1 \in \nul^{\preceq}$ if and only if  $\omega_2 \in \nul^{\preceq}$. We start by defining $\Theta_N = \{\theta: \theta \in \nul^\preceq\}$ and $\Theta_S = \Theta \setminus \Theta_N$. $\Theta_N$ is a null event, by \autoref{lm:krantz8} part 4, and thus $\Theta_S$ is an almost sure event. 

Now, assume that $\omega_1 \in \nul^\preceq$. By \autoref{lm:krantz8} part 3, this implies that $\omega_1 \cap \theta \in \nul^\preceq$, which implies that $\omega_1|\theta \sim \emptyset$, for all $\theta \in \Theta_S$. By \autoref{lm:qSufficientZeroProb}, it follows that $\omega_2|\theta \sim \emptyset$ for all $\theta \in \Theta_S$. Let $\theta_1, \dots, \theta_n$ denote the elements of $\Theta_S$. We apply \autoref{lm:qLPEntailsPosteriorEquivalence2}. Let $A = \Delta$, $B_i = \theta_i$, and $C_i = \omega_2$ for all $i$. Thus, $G = \Theta_S$. We now check the conditions of the lemma:
\begin{enumerate}
    \item $\Delta \cap \theta_i \notin \nul$. This follows from the definition of $\Theta_S$.
    \item $\Delta \ind_{\theta_i} \omega_2$. This is equivalent to $\Delta | \omega_2, \theta_i \sim \Delta | \theta_i$, which is true as both sides are equiprobable with $\Delta|\Delta$.
    \item $\omega_2|\theta_i \sim \omega_2|\theta_j$. This is true as both sides are equiprobable with $\emptyset$.
\end{enumerate}
With the conditions satisfied, the lemma tells us that $\bigcup_i \omega_2 \cap \theta_i|\Theta_S \sim \emptyset$, which simplifies to $\omega_2|\Theta_S \sim \emptyset$. Now, since $\Theta_S$ is an almost sure event, we can apply \autoref{lm:krantz9} part 5 to get that $\omega_2|\Delta \sim \omega_2|\Theta_S \sim \emptyset$, which shows that $\omega_2 \in \nul^{\preceq}$, as desired.

Next, we show that  for all $\theta$ and $\omega \not \in \nul^{\preceq}$:
\begin{equation}\label{eqn:qSWPE1}
   \theta|\{T=T(\omega)\} \sim  \theta|\omega 
\end{equation}
By repeated applications of \Ax{5}, it follows that $H|\{T=T(\omega)\} \sim  H|\omega$ for all finite $H \subseteq \Theta$ and all $\omega$.  Since $\Theta$ is finite, it follows that $T(\omega)$ and $\omega$ are posterior equivalent for all $\omega$.  Since $T(\omega_1)=T(\omega_2)$ and posterior equivalence is transitive, it follows that $\omega_1$ and $\omega_2$ are posterior equivalent, as desired.

To prove \autoref{eqn:qSWPE1}, first note that if $\theta \cap \{T = T(\omega)\} \in \nul$, then $\theta \cap \omega \in \nul$ as well, as $\omega \subseteq \{T = T(\omega)\}$. So \autoref{eqn:qSWPE1} follows trivially in this case, as both sides are equiprobable with $\emptyset$.

We now focus on the more difficult case. Let $\Theta_s = \{\theta_1, \dots, \theta_n\}$ enumerate all $\theta_i$ such that $\theta_i \cap \{T = T(\omega)\} \notin \nul$. To apply \autoref{lm:qLPEntailsPosteriorEquivalence2}, define:
\begin{eqnarray*}
    A &=&  \Delta \\
    B_i &=& \theta_i \cap \{T=T(\omega)\} \\
    C_i &=& \omega \\
    G &=& \{T=T(\omega)\}
\end{eqnarray*}

By the definition of $\Theta_s$, $A \cap B_i \notin \nul$ for all $i$. Further, $C_i|B_i \sim C_j|B_j$ for all $i, j$ because $T$ is a sufficient statistic. Thus, the conditions of the lemma are met, and so 
\autoref{lm:qLPEntailsPosteriorEquivalence2} entails that for all $i$, we have: 
\[\bigcup_{j \leq n} C_j \cap B_j|A \cap G \sim C_i|B_i\]
which reduces to:
\begin{equation*}
   \omega \cap \Theta_s|\{T=T(\omega)\} \sim \omega | \{T=T(\omega)\} \cap \theta_i\ 
\end{equation*}
Since $\Theta_s$ is $\{T = T(\omega)\}$-almost sure, \autoref{lm:krantz9} allows us to simplify to:
\begin{equation}\label{eqn:qSWPE2}
   \omega |\{T=T(\omega)\} \sim \omega | \{T=T(\omega)\} \cap \theta_i\ 
\end{equation}
Next, fix any $\theta \in \Theta_s$, and define:
  \begin{alignat*}{3}
        &A = A' = \{T= T(\omega)\}  \\
        & B = \{T= T(\omega)\} \cap \theta & \quad \quad & B'= \omega \\
        & C = C' =  \omega \cap \theta
    \end{alignat*}
By \autoref{eqn:qSWPE2}, we have $C|B \sim B'|A'$.  Further,  $C|A = C'|A'$ by definition.  Thus, by \autoref{lm:Aditya}, we get $C'|B' \sim B|A$, i.e., that $\theta|\omega \sim \theta|T=T(\omega)$, as desired.

\end{proof}

\subsection{Ancillary Statistics and Mixtures}
\label{subsec:Ancillary}

\setcounter{lemma}{4}
\begin{lemma}\label{lm:qAncillaryNoUpdate}
    If $T$ is ancillary, then for all $t \in {\cal R}$, the events $\{T=t\}$ and $\Omega$ are posterior equivalent.
\end{lemma}

To understand the significance of this lemma, it helps to compare it to the quantitative analog.  If $T$ is ancillary, then $Q(\cdot|T=t) = Q(\cdot)$ for any joint distribution $Q$ representing the experimenter's beliefs.  In other words, learning the value of an ancillary statistic does not change any Bayesian's degree of beliefs.  \autoref{lm:qAncillaryNoUpdate} says the same holds in the qualitative setting.\\

\begin{proof}[Proof of \autoref{lm:qAncillaryNoUpdate}]
    By applying \autoref{lm:qbfl1} repeatedly to the fact that $\{T = t\}|\theta \sim \{T=t\}|\upsilon$ for all $\theta, \upsilon \in \Theta$, we get that $\{T = t\}|\theta \sim \{T=t\}|\Delta$.  In other words, the events $\{T=t\}$ and $\theta$ are independent.  Now, we use the fact that (unconditional) independence is symmetric (\autoref{lm:QConditionalIndependenceIsSymmetric}), which gives us that $\theta|\{T = t \} \sim \theta|\Delta$. Finally, to get $H | \{T = t\} \sim H| \Delta$, we can apply \Ax{5} repeatedly on the simple hypotheses in $H$.
\end{proof}

\begin{assumption}\label{assum:qMixturesExist}
For any two experiments $\mathbb{E}$ and $\mathbb{F}$, there is a mixture $\mathbb{M}$ of the two experiments such that each of the two component experiments has an equal probability of being conducted, i.e., $\Omega^{\mathbb{E}} |^{\mathbb{M}} | \theta \equiv \Omega^{\mathbb{E}} |^{\mathbb{M}} | \theta$ for all $\theta$.
\end{assumption}

\subsection{Proof of \autoref{thm:qmsp}}
To finish the proof of \autoref{thm:qmsp}, we need one substantial technical lemma.

\begin{lemma}\label{lm:qBayesLTPGeneral}
    Let $A, B_1 \ldots B_n$ be events in some experiment $\mathbb{E}$, and let $C_1, \ldots C_n, B_1' \ldots B_n'$ be events in an experiment $\mathbb{F}$ (which may or may not be identical to $\mathbb{E}$).  Let $G = \bigcup_{i \leq n} B_i$ and $G' = \bigcup_{i \leq n} B_i'$. 
    Further, suppose
    \begin{enumerate}
        \item $B_1, \ldots, B_n$ are pairwise disjoint,
        \item $B_1', \ldots, B_n'$ are pairwise disjoint,
        \item $B_i \ce G \sim B_i' \cf G'$ for all $i \leq n$, and
        \item  $A \ce B_i \sim C_i \cf  B'_i$  for all $i \le n$.
   \end{enumerate}
 If $A \cap G \not \in \nul$, then  $B_i \ce A \cap G \sim B_i \cap C_i \cf \bigcup_{j \le n} B_j' \cap C_j$ for all $i \le n$.
\end{lemma}

\autoref{lm:qBayesLTPGeneral} can be thought of as an application of Bayes' theorem combined with the Law of Total Probability (as well as the conditional extension). If  $C_i = C$ for all $i$, then quantitative analog of the lemma asserts the following: if $B_1 \ldots B_n$ partition $G$ and  $P(A|B_i)=P(C|B_i)$ for all $i$, then $P(B_i|A \cap G)=P(B_i|C \cap G)$ for all $i$. That claim follows easily from an application of Bayes' theorem and the Law of Total Probability. 

\begin{proof}
    We first prove that if $A \cap G \not \in \nul$, then $\bigcup_{j \le n} B_j' \cap C_j \not \in \nul$, and so the expressions in the statement of the lemma are well-defined.
    
    Because $A \cap G = \bigcup_{i\leq n} A \cap B_i$, it follows from Part 4 of \autoref{lm:krantz9} that $A \cap B_i \not \in \nul$ for some $i \leq n$.  Thus, $B_i \not \in \nul$ since  $A \cap B_i \subseteq B_i$.  Further, $A \ce B_i \succ \emptyset \ce B_i$, for otherwise, $A\cap B_i\ce \Delta \preceq \emptyset \ce \Delta$ by \autoref{lm:qConditionalIntersection} and $A \cap B_i \in \nul$, contradicting assumption.  Finally, since $A \ce B_i \sim C_i \cf B_i'$, it follows that $C_i \cf B_i' \succ \emptyset \cf B_i'$.  It follows quickly, via the same sequence of lemmata that we just used, that $C_i \cap B_i' \not \in \nul$, and hence, $\bigcup_{j \le n}B_j' \cap C_j \not \in \nul$ since  $C_i \cap B_i' \subseteq
    \bigcup_{j \le n}B_j' \cap C_j$.
    
    Now we prove the lemma by induction on $n$. In the base case, when $n=1$, $B_1 = G$. Then: 
    \begin{align*}
        G \ce A \cap G    &\sim A \cap G \ce A \cap G & \text{by \Ax{4}}\\
                        &\sim C_1 \cap B_1'\cf C_1 \cap B_1' & \text{by \Ax{3}}
    \end{align*}
    which is the desired equality.
    
    For the inductive step, assume $B_1, \ldots, B_{n+1}$ partition $G$ and $A \ce B_i \sim C_i \cf B_i'$ for all $i \le n+1$. Let $G^- = G \setminus B_{n+1}$. Clearly, $B_1, \ldots, B_n$ partition $G^-$ and $A\ce B_i \sim C_i \cf B_i'$ for all $i \le n$. We consider three cases: either one of $A \cap G^-$ or $A \cap B_{n+1}$ are not members of $\nul$ or neither of them are (since $A \cap G \notin \nul$, they cannot both be members of $\nul$).

    \paragraph{Case 1: $A \cap G^-, A \cap B_{n+1} \notin \nul$} We handle the difficult case first. By the inductive hypothesis, we know that  for all $i \le n$:
    \begin{equation}\label{eq:qBayesLTP1}
        B_i \ce A \cap G^- \sim B_i'\cap C_i \cf  \bigcup_{j \le n} B_j' \cap C_j 
    \end{equation}
    \autoref{eq:qBayesLTP1} also holds for $i = n+1$.  Why?  Both $B_{n+1} \cap (A \cap G^-)$ and $(B_{n+1}' \cap C_{n+1}) \cap \bigcup_{j \le n} B_j' \cap C_j $ are empty.  Thus,
    \begin{eqnarray*}
        B_{n+1} \cap (A \cap G^-) \ce (A \cap G^-)  &=& \emptyset \ce (A \cap G^-) \\ &\sim& \emptyset \cf \bigcup_{j \le n} B_j' \cap C_j \mbox{ by \autoref{lm:krantz6-2}} \\
        &=& (B_{n+1}' \cap C_{n+1}) \cap \bigcup_{j \le n} B_j' \cap C_j \cf \bigcup_{j \le n} B_j' \cap C_j
    \end{eqnarray*}
    By transitivity of $\sim$ (\Ax{1}), it follows that:
    \[ B_{n+1} \cap (A \cap G^-) \ce (A \cap G^-) \sim (B_{n+1}' \cap C_{n+1}) \cap \bigcup_{j \le n} B_j' \cap C_j \cf \bigcup_{j \le n} B_j' \cap C_j\]
    Applying \Ax{4} to both sides of that equality yields \autoref{eq:qBayesLTP1} for $i=n+1$.
    
    Further, it's easy to show that for all $i \le n+1$:
    \begin{equation}\label{eq:qBayesLTP2}
        B_i|A \cap B_{n+1} \sim B_i' \cap C_i | B_{n+1}' \cap C_{n+1}
    \end{equation}
    If $i = n+1$, then both sides are $\sim$-equivalent to the sure event (just use \Ax{4} and \Ax{3} like in the base case). If $i \neq n+1$, then note that $B_i \cap (A \cap B_{n+1}) = (B_i' \cap C_i) \cap (B_i' \cap C_{n+1}) = \emptyset$. So like above, we use \Ax{4} and \autoref{lm:krantz6-2} to get the equality. Now recall, our goal is to show:
    \begin{equation}\label{eq:qBayesLTPGoal1}
        B_i \ce A\cap G \sim B_i' \cap C_i \cf \bigcup_{j \le n+1} B_j' \cap C_j \text{ for all } i \le n+1
    \end{equation}
    \autoref{eq:qBayesLTPGoal1} would follow if we could ``union'' the events to the right of the conditioning bars in \autoref{eq:qBayesLTP1} and \autoref{eq:qBayesLTP2}. \autoref{lm:qPartitionOnRight} was designed for that purpose.  Let $i$ be arbitrary and define:
    \begin{alignat*}{3}
        & X = B_i               & \quad \quad   & W = B_i' \cap C_i \\
        & Y = A \cap B_{n+1}        & \quad \quad   & Z = B_{n+1}' \cap C_{n+1}\\
        & Y' = A \cap G^{-}   & \quad \quad   & Z' = \bigcup_{k\le n} B_k' \cap C_k
    \end{alignat*}
    We want to apply \autoref{lm:qPartitionOnRight}, so we first check all the conditions. Note that $Y \cap Y' = Z \cap Z' = \emptyset$, since both the $B_i$s and $B_i'$s are pairwise disjoint. Further, $X \ce Y \sim W \cf Z$ is simply \autoref{eq:qBayesLTP2} and $X\ce Y' \sim W \cf Z'$ is \autoref{eq:qBayesLTP1}. Lastly, if $i = n+1$, $X \ce Y' = B_i \ce  A \cap G^{-} \sim \emptyset \ce \Delta^{\mathbb{E}}$ by the disjointness of $B_{n+1}$ and $G^{-}$, along with \Ax{4}, and \autoref{lm:krantz6-2}. If $i \le n$, $X \ce Y = B_i \ce A \cap B_{n+1} \sim \emptyset \ce \Delta^{\mathbb{E}}$, again because $B_i$ and $B_{n+1}$ are disjoint. In this case, we can simply swap $Y$ with $Y'$ and $Z$ with $Z'$ when applying the lemma. With the conditions satisfied, \autoref{lm:qPartitionOnRight} tells us that \autoref{eq:qBayesLTPGoal1} holds if and only if $Y \ce Y \cup Y' \sim Z \cf Z \cup Z'$ i.e.:
    \begin{equation}\label{eq:qBayesLTPGoal2}
        A \cap B_{n+1} \ce A \cap G \sim B_{n+1}' \cap C_{n+1} \cf \bigcup_{k \le n+1} B_k' \cap C_k
    \end{equation}
    To prove \autoref{eq:qBayesLTPGoal2}, we first show:
    \begin{equation}\label{eq:qBayesLTPContra}
        A \cap B_{n+1} \ce G \sim B_{n+1}' \cap C_{n+1} \cf G'
    \end{equation}
    Define:
    \begin{alignat*}{3}
        & X = G               & \quad \quad   & X' = G' \\
        & Y = B_{n+1}        & \quad \quad   & Y' = B_{n+1}'\\
        & Z = A \cap B_{n+1}   & \quad \quad   & Z' = B_{n+1}' \cap C_{n+1}
    \end{alignat*}
Assumption 3 in the statement of the lemma says that $Y | X \sim Y' | X'$, and $Z|Y = A \cap B_{n+1} | B_{n+1} \sim A|B_{n+1}$ by \Ax{4}. Similarly, $Z'|Y' \sim C_{n+1} | B_{n+1}'$. By the assumption that  $A | B_i \sim C_i | B_i'$ for all $i \le n+1$, we get $Z|Y \sim Z'|Y'$. So \Ax{6b} tells us that $Z|X \sim Z'|X'$, which is \autoref{eq:qBayesLTPContra}.
    
We now prove \autoref{eq:qBayesLTPGoal2} by contradiction. Suppose $A \cap B_{n+1} \ce A\cap G \not \sim B_{n+1}' \cap C_{n+1} \cf \bigcup_{k \le n+1} B_k' \cap C_k$, and without loss of generality, say $A \cap B_{n+1} \ce A \cap G \succ B_{n+1}' \cap C_{n+1} \cf \bigcup_{k\le n+1} B_k '\cap C_k$. Now, define:
    \begin{alignat*}{3}
        & X = G               & \quad \quad   & X' = G' \\
        & Y = A \cap G        & \quad \quad   & Y' = \bigcup_{k\le n+1} B_k' \cap C_k\\
        & Z = A \cap B_{n+1}   & \quad \quad   & Z' = B_{n+1}' \cap C_{n+1}
    \end{alignat*}
    By \autoref{lm:qBayesLTPHelper} (letting $A_i = A$ for all $i$), we know that $A \cap G \ce G \sim \bigcup_{k\le n+1} B_k' \cap C_k \cf G'$, i.e. $Y \ce X \sim Y' \cf X'$.
    Note that $Z \ce Y \succ Z' \cf Y'$ by assumption. Since this shows that $Z \notin \nul$, \Ax{6b} entails that $Z\ce X \succ Z' \cf X'$,  contradicting \autoref{eq:qBayesLTPContra}. 
    
    \paragraph{Case 2: $A \cap G^{-} \in \nul, A \cap B_{n+1} \notin \nul$} Notice that since $A \cap G^{-} \in \nul$, $A \cap B_i \in \nul$, for all $i \le n$, by \autoref{lm:krantz9}. Thus, $B_i' \cap C_i \in \nul$, by our assumption. So, for all $i \le n$, our goal is to show that $B_i \ce A \cap G \sim \emptyset$. Since $A \cap B_i \in \nul$, this follows by \Ax{4}. All that remains is to show that $B_{n+1} \ce A \cap G \sim B_{n+1}' \cap C_{n+1} \cf \bigcup_{j \le n} B_j' \cap C_j$. Notice that \autoref{eq:qBayesLTP2} still holds in this case. We simply apply Part 6 of \autoref{lm:krantz9} to both sides of this equation, adding in $A \cap G^{-}$ to the left hand side and $\bigcup_{j \le n} B_i' \cap C_i$ to the right hand side, since those events are null. 

    \paragraph{Case 3: $A \cap B_{n+1} \in \nul, A \cap G^{-} \notin \nul$} Since $A \cap B_{n+1} \in \nul$, $B_{n+1}' \cap C_{n+1} \in \nul$, by our assumption. So, $B_{n+1} \ce A \cap G \sim B_{n+1}' \cap C_{n+1} \cf \bigcup_{j \le n} B_j' \cap C_j$ holds since both sides are equiprobable with the emptyset. Further, \autoref{eq:qBayesLTP1}, the inductive hypothesis, still holds in this case. Just as before, we simply apply Part 6 of \autoref{lm:krantz9} to both sides of this equation, adding in $A \cap B_{n+1}$ to the left hand side and $B_{n+1} \cap C_{n+1}$ to the right hand side.
\end{proof}

\begin{lemma} \label{cor:mixtureComponentEquivalence}
 Let $\omega$ be an outcome of some experiment $\mathbb{E}$ and let $\mathbb{M}$ be an arbitrary mixture including $\mathbb{E}$.  Then observing  $\omega$ in the mixture is posterior equivalent to observing $\omega$ in the component experiment.   In symbols, $H \ce \omega \sim H|^{\mathbb{M}} \omega$ for all hypotheses $H$ and all orderings $\preceq$. 
\end{lemma}
\setcounter{lemma}{25}
\begin{proof}
    This follows directly from \autoref{lm:qBayesLTPGeneral}, with the appropriate substitutions. Let:
  
  \begin{align*}
        A &= \omega^{\mathbb{E} }\\
        B_i &= \theta_i \\
        B_i' &= \theta_i \cap \Omega^{\mathbb{E}}\\
        C_i &= \omega^{\mathbb{M}} \\
    \end{align*}
    
Here, we use $\omega^{\mathbb{E}}$ to refer to the outcome $\omega$ when observed in $\mathbb{E}$, and similar remarks apply to $\omega^{\mathbb{M}}$.
    
 We now show that the conditions of \autoref{lm:qBayesLTPGeneral} hold. Obviously, the events $B_1, \ldots B_n$ are pairwise disjoint, as are $B_1', \ldots B_n'$.  So conditions 1 and 2 are met. Condition 4 of the lemma says that $\omega^{\mathbb{E}}|\theta_i \sim \omega^{\mathbb{M}}|\theta_i, \Omega^{\mathbb{E}}$, which follows immediately from the third condition in the definition of a mixture.
 
Condition 3 says that $\theta_i|^{\mathbb{E}}\Theta \sim \theta_i|^{\mathbb{M}}\Theta \cap \Omega^{\mathbb{E}}$. We prove this in two steps. First, recall that $\theta_i|^{\mathbb{E}} \Theta \sim \theta_i|^{\mathbb{M}} \Theta$, as we assume that priors don't vary across experiments (\Ax{7}). Next, the fourth condition in the definition of a mixture says that the projection statistic is ancillary. Thus, we can apply \autoref{lm:qAncillaryNoUpdate}
to the projection statistic, giving us that $\Omega^{\mathbb{E}}$ is posterior equivalent to $\Omega^{\mathbb{M}}$ (and recall that $\Omega^{\mathbb{M}}$ and $\Theta$ are both shorthand for the same event, $\Delta^{\mathbb{M}}$). This means that $\theta_i|^{\mathbb{M}} \Theta \sim \theta_i|^{\mathbb{M}} \Theta, \Omega^{\mathbb{E}}$. So, we've shown Condition 3 by transitivity.
    
    Now, \autoref{lm:qBayesLTPGeneral} says that $\theta_i|^{\mathbb{E}} \omega^{\mathbb{E}} \sim \theta_i|^{\mathbb{M}} \omega^{\mathbb{M}}, \Omega^{\mathbb{E}} = \theta_i|^{\mathbb{M}} \omega^{\mathbb{M}}$, as desired.
\end{proof}

\begin{proof}[Proof of \autoref{thm:qmsp}]
    In the forward direction, suppose that, in every mixture $\mathbb{M}$ of $\mathbb{E}$ and $\mathbb{F}$, there is some sufficient statistic $T$ such that $T(\omega_1^{\mathbb{M}})=  T(\omega_2^{\mathbb{M}})$.  We must show  $\omega_1^{\mathbb{E}}$ and  $\omega_2^{\mathbb{F}}$ are $\mathcal{U}$, where $\mathcal{U}$ is the set of all orderings satisfying our axioms.  Because  there is a sufficient statistic $T$ for $\mathbb{M}$ such that $T(\omega_1^{\mathbb{M}})=  T(\omega_2^{\mathbb{M}})$, \autoref{prop:qSufficientWeakPosteriorEquivalence} entails that $\omega_1^{\mathbb{M}}$ and $\omega_2^{\mathbb{M}}$ are $\mathcal{U}$-posterior equivalent. From \autoref{cor:mixtureComponentEquivalence}, we know that $\omega_1^{\mathbb{M}}$ is $\mathcal{U}$-posterior equivalent to $\omega_1^{\mathbb{E}}$ and that $\omega_2^{\mathbb{M}}$ is $\mathcal{U}$-posterior equivalent to $\omega_2^{\mathbb{F}}$. By the transitivity of posterior equivalence, we conclude that $\omega_1^{\mathbb{E}}$ and $\omega_2^{\mathbb{F}}$ are $\mathcal{U}$-posterior equivalent, as desired.
    
    The reverse direction follows by symmetry; each step of the above argument is reversible. In more detail, first assume that $\omega_1^{\mathbb{E}}$ and $\omega_2^{\mathbb{F}}$ are ${\cal U}$-posterior equivalent.  By the same steps as above, we conclude that $\omega_1^{\mathbb{M}}$ and $\omega_2^{\mathbb{M}}$ are posterior equivalent in any mixture $\mathbb{M}$. Then, we use \autoref{prop:qSufficientWeakPosteriorEquivalence} and the fact that $\mathcal{U}$ is the universal set of orderings to conclude that there is a sufficient statistic $T$ (for $\mathbb{M}$) such that $T(\omega_1^{\mathbb{M}}) = T(\omega_2^{\mathbb{M}})$, and thus $\omega_1^{\mathbb{E}}$ and $\omega_2^{\mathbb{F}}$ are evidentially equivalent by \f{qmsp}.
\end{proof}

\section{No Analog of LL+}
We prove the following proposition, used in our proof that there is no qualitative analog of $\f{ll}^+$. 

\begin{proposition}\label{prop:qLL+NoAnalogMain}
   Suppose the following three conditions hold:
    \begin{enumerate}
        \item $B_1 \cap B_2 = \emptyset$,
        \item $A_1 | B_1 \succ A_2 | B_1 \succ \emptyset$, and
        \item $A_2 | B_2 \succeq A_1 |B_2 \succ \emptyset$.
    \end{enumerate}
Then 
\[ B_1 | A_1 \cap (B_1 \cup B_2) \succ B_1 | A_2 \cap (B_1 \cup B_2)\]
\end{proposition}
\begin{proof}
    The proposition follows immediately from the next two lemmata. In detail, from the conditions 1 and 2,  \autoref{lm:qLL+A} allows us to infer that
    \[B_1 | A_1 \cap (B_1 \cup B_2) \succ B_1 \cap A_2 | (A_1 \cap B_2) \cup (A_2 \cap B_1)\]
    Under conditions 1 and 3, \autoref{lm:qLL+B}  entails:
    \[ B_1 \cap A_2 | (A_1 \cap B_2) \cup (A_2 \cap B_1) \succeq B_1 | A_2 \cap (B_1 \cup B_2)\]
    Because $\preceq$ is transitive (by \Ax{1}), the desired inequality follows.
\end{proof}

\begin{lemma}\label{lm:qLL+A}
    Suppose the following two conditions hold:
    \begin{enumerate}
        \item $B_1 \cap B_2 = \emptyset$, and
        \item $A_1 | B_1 \succ A_2 |B_1 \succ \emptyset$.
    \end{enumerate}
Then 
\[  B_1 | A_1 \cap (B_1 \cup B_2) \succ B_1 \cap A_2 | (A_1 \cap B_2) \cup (A_2 \cap B_1)\]
\end{lemma}
\begin{proof}
First, we will show
\begin{equation}\label{eqn:qLL+A1}
    A_1 | B_1 \cup B_2 \succ (A_1 \cap B_2) \cup (A_2 \cap B_1)| B_1 \cup B_2
\end{equation}
By \Ax{1} (specifically, reflexivity), we have
\begin{equation}\label{eqn:qLL+A2}
    A_1  \cap B_2 | B_1 \cup B_2 \sim A_1 \cap B_2 | B_1 \cup B_2
\end{equation}
and further, we will show:
\begin{equation}\label{eqn:qLL+A3}
    A_1  \cap B_1 | B_1 \cup B_2 \succ A_2 \cap B_1 |B_1 \cup B_2
\end{equation}
Since $B_1 \cap B_2 = \emptyset$, we can apply \Ax{5} to \autoref{eqn:qLL+A2} and \autoref{eqn:qLL+A3} to obtain:
\[ A_1 \cap (B_1 \cup B_2) | B_1 \cup B_2 \succ (A_1 \cap B_2) \cup (A_2 \cap B_1) | B_1 \cup B_2 \]
And applying \Ax{4} to the left-hand side of the last inequality yields \autoref{eqn:qLL+A1} as desired.

To show \autoref{eqn:qLL+A3}, we apply \Ax{6b} to the following events:
\begin{align*}
        \mathfrak{A} &= \mathfrak{A}' = B_1 \cup B_2  \\ 
        \mathfrak{B} &=  \mathfrak{B}' = B_1 \\ 
        \mathfrak{C} &= A_1 \cap B_1 &&  \mathfrak{C}'= A_2 \cap B_1 
    \end{align*}
Because $A_1 | B_1 \succ A_2 |B_1$ (by assumption), \Ax{4} entails that $\mathfrak{C}|\mathfrak{B} \succ \mathfrak{C}'|\mathfrak{B}'$.  And $\mathfrak{B}|\mathfrak{A} \sim \mathfrak{B}'|\mathfrak{A}'$ by reflexivity (part of \Ax{1}).  Thus, by \Ax{6b}, we obtain $\mathfrak{C}|\mathfrak{A} \succ \mathfrak{C}'|\mathfrak{A}'$, which is just \autoref{eqn:qLL+A3}, as desired.  

That completes the proof of \autoref{eqn:qLL+A1}.  Next, we show:
\begin{equation}\label{eqn:qLL+A4}
    A_1 \cap B_2 | A_1 \cap (B_1 \cup B_2) \prec A_1 \cap B_2 | (A_1 \cap B_2) \cup (A_2 \cap B_1)
\end{equation}
For the sake of contradiction, assume \autoref{eqn:qLL+A4} fails.  By totality of $\preceq$ (part of \Ax{1}), we obtain that
\begin{equation}\label{eqn:qLL+A5}
    A_1 \cap B_2 | A_1 \cap (B_1 \cup B_2) \succeq A_1 \cap B_2 | (A_1 \cap B_2) \cup (A_2 \cap B_1)
\end{equation}
Now define:
\begin{align*}
        \mathfrak{A} &= \mathfrak{A}' = B_1 \cup B_2  \\ 
        \mathfrak{B} &= A_1 \cap (B_1 \cup B_2)  && \mathfrak{B}' = (A_1 \cap B_2) \cup (A_2 \cap B_1) \\ 
        \mathfrak{C} &=  \mathfrak{C}'= A_1 \cap B_2
    \end{align*}
By \autoref{eqn:qLL+A5}, we know $\mathfrak{C}|\mathfrak{B} \succeq \mathfrak{C}'|\mathfrak{B}'$.  Further by \autoref{eqn:qLL+A1} and \Ax{4}, it follows that $\mathfrak{B}|\mathfrak{A} \succ \mathfrak{B}'|\mathfrak{A}'$.  Finally, $\mathfrak{C}'$ is not null because, by assumption $A_1|B_2 \succ \emptyset$.  Thus, \Ax{6b} entails that $\mathfrak{C}|\mathfrak{A} \succeq \mathfrak{C}'|\mathfrak{A}'$, and that's a contradiction since $\mathfrak{C}|\mathfrak{A} = \mathfrak{C}'|\mathfrak{A}'$.

That concludes the proof of \autoref{eqn:qLL+A4}.  We're finally ready to show our ultimate conclusion, i.e., that
 \[ B_1 | A_1 \cap (B_1 \cup B_2) \succ B_1 \cap A_2 | (A_1 \cap B_2) \cup (A_2 \cap B_1).\]
To do so, suppose for the sake of contradiction that inequality fails.  By totality of $\preceq$ (part of \Ax{1}), we obtain that:
\[
  B_1 | A_1 \cap (B_1 \cup B_2) \preceq B_1 \cap A_2 | (A_1 \cap B_2) \cup (A_2 \cap B_1)\]
Apply \Ax{4} to the left-hand side of the last inequality to obtain:
\begin{equation}\label{eqn:qLL+A6}
    A_1 \cap B_1 | A_1 \cap (B_1 \cup B_2) \preceq B_1 \cap A_2 | (A_1 \cap B_2) \cup (A_2 \cap B_1)
\end{equation}
Because $B_1 \cap B_2 = \emptyset$, we can apply \Ax{5} to \autoref{eqn:qLL+A6} and \autoref{eqn:qLL+A4} to obtain that:
\begin{equation}\label{eqn:qLL+A7}
    (A_1 \cap B_1) \cup (A_1 \cap B_2) | A_1 \cap (B_1 \cup B_2) \prec (A_1 \cap B_2) \cup (A_2 \cap B_1) | (A_1 \cap B_2) \cup (A_2 \cap B_1)
\end{equation}
But the last inequality contradicts \Ax{3}, since both sides are of the form $X|X$ for appropriate $X$.
\end{proof}

\begin{lemma}\label{lm:qLL+B}
    Suppose the following two conditions hold:
    \begin{enumerate}
        \item $B_1 \cap B_2 = \emptyset$, and
        \item $A_2 | B_2 \succeq A_1 | B_2 \succ \emptyset$.
    \end{enumerate}
Then 
\[ B_1 \cap A_2 | (A_1 \cap B_2) \cup (A_2 \cap B_1) \succeq B_1 | A_2 \cap (B_1 \cup B_2)\]
\end{lemma}
\begin{proof}
Suppose for the sake of contradiction that 
\[B_1 \cap A_2 | (A_1 \cap B_2) \cup (A_2 \cap B_1) \not \succeq B_1 | A_2 \cap (B_1 \cup B_2)\]
By totality of $\preceq$ (which is part of \Ax{1}), we obtain that
\begin{equation}\label{eqn:qLL+B1}
B_1 \cap A_2 | (A_1 \cap B_2) \cup (A_2 \cap B_1) \prec B_1 | A_2 \cap (B_1 \cup B_2)
\end{equation}
Below, we will show that:
\begin{equation}\label{eqn:qLL+B2}
    A_1 \cap B_2 | (A_1 \cap B_2) \cup (A_2 \cap B_1) \preceq A_2 \cap B_2 | A_2 \cap (B_1 \cup B_2)
\end{equation}
Applying \Ax{5} to the last two inequalities yields:
\[(A_1 \cap B_2) \cup (A_2 \cap B_1)| (A_1 \cap B_2) \cup (A_2 \cap B_1) \prec A_2 \cap (B_1 \cup B_2) | A_2 \cap (B_1 \cup B_2)\]
But that contradicts \Ax{3} because both sides of the inequality are of the form $X|X$ for appropriate $X$.

So it suffices to prove \autoref{eqn:qLL+B2}.  To do so, we first show that
\begin{equation}\label{eqn:qLL+B3}
     (A_1 \cap B_2) \cup (A_2 \cap B_1) \preceq A_2 \cap (B_1 \cup B_2)
\end{equation}
By assumption, $A_1|B_2 \preceq A_2|B_2$, which entails \[A_1 \cap B_2 \preceq A_2 \cap B_2\] \autoref{lm:qConditionalIntersection}.  Further, by reflexivity of $\preceq$ (again, part of \Ax{1}), we know:
\[A_2 \cap B_1 \sim A_2 \cap B_1 \]
Because $B_1 \cap B_2 = \emptyset$, we can apply \Ax{5} to the last two inequalities to obtain \autoref{eqn:qLL+B3}.

Now suppose for the sake of contradiction that \autoref{eqn:qLL+B2} failed.  By totality of $\preceq$ (\Ax{1}), that entails:
\begin{equation}\label{eqn:qLL+B4}
    A_1 \cap B_2 | (A_1 \cap B_2) \cup (A_2 \cap B_1) \succ A_2 \cap B_2 | A_2 \cap (B_1 \cup B_2)
\end{equation}
Define:
\begin{align*}
        \mathfrak{A} &= \mathfrak{A}' = \Delta \\ 
        \mathfrak{B} &= (A_1 \cap B_2) \cup (A_2 \cap B_1)  && \mathfrak{B}' = A_2 \cap (B_1 \cup B_2) \\ 
        \mathfrak{C} &= A_1 \cap B_2  && \mathfrak{C}'= A_2 \cap B_2
\end{align*}
 \autoref{eqn:qLL+B4} and \Ax{4} entail that $\mathfrak{C}|\mathfrak{B} \succ \mathfrak{C}'|\mathfrak{B}'$.  Further by \autoref{eqn:qLL+B3} and \Ax{4}, it follows that $\mathfrak{B}|\mathfrak{A} \succeq \mathfrak{B}'|\mathfrak{A}'$.  Finally, $\mathfrak{C}'$ is not null because by (1) assumption $A_1|B_2 \succ \emptyset$ and (2) therefore, $A_1 \cap B_2 \succ \emptyset$ by \autoref{lm:qConditionalIntersection}.  Thus, \Ax{6b} entails that $\mathfrak{C}|\mathfrak{A} \succ \mathfrak{C}'|\mathfrak{A}'$.  In other words, we have shown that $A_1 \cap B_2 \succ A_2 \cap B_2$. 
 
 Finally, we just showed $A_1 \cap B_2$ is not null, and from that, it follows that $B_2$ is not null \autoref{lm:qConditionalNull}.  Since $A_1 \cap B_2 \succ A_2 \cap B_2$ and $B_2$ is not null, it follows from \autoref{lm:qConditionalIntersection} that $A_1|B_2 \succ A_2|B_2$.  Yet that contradicts our assumption that $A_1|B_2 \preceq A_2|B_2$.  And we're done.
\end{proof}

We also use the following theorem, which is a weak qualitative analog of \f{lp}. We postpone the proof and discussion of this to a later section.
\setcounter{theorem}{3}
\begin{restatable}{theorem}{qlpMinus}
    \label{thm:qLPEntailsPosteriorEquivalence}
    Let $\{C_{\theta} \}_{\theta \in \Theta}$ be events such that
    \begin{enumerate}
            \item $E| \theta \equiv F \cap C_{\theta} | \theta$ for all $\theta \in \Theta$, 
            \item $F \ind_\theta C_{\theta}$ with respect to $\sqsubseteq$ for all $\theta \in \Theta$,
            \item $C_{\theta} | \theta \equiv C_{\eta} | \eta$ for all $\theta, \eta \in \Theta$.
            \item $\emptyset|\theta \sqsubset C_\theta|\theta$ for all $\theta$.
    \end{enumerate}	 
    If $\Theta$ is finite, then $E$ and $F$ are  $\mathcal{U}$-posterior equivalent.
\end{restatable}

\section{Irrelevance of Stopping Rules}
In this section we prove the lemmas used in the proof of the irrelevance of qualitative stopping rules. 

The following lemma is an analog of the chain rule for numerical probabilities.

\setcounter{lemma}{7}
\begin{lemma}\label{lm:chainrule}
    Let $1 \le n_0 \le n$ be arbitrary. Suppose, for all $m \le n_0$, we have that $A_m |^{\mathbb{E}} \bigcap_{k > m} A_k \sim B_m |^{\mathbb{F}} \bigcap_{k > m} B_k$. Then, we know that $\bigcap_{j \le n_0} A_j |^{\mathbb{E}} \bigcap_{k > n_0} A_k \sim \bigcap_{j \le n_0} B_j |^{\mathbb{F}} \bigcap_{k > n_0} B_k$.
\end{lemma}
\setcounter{lemma}{30}
\begin{proof}
   The proof goes by induction on $n$. When $n = 1$, the only choice of $n_0$ is $1$. In this case, the premise and conclusion are identical: $A_1 |^{\mathbb{E}} \Delta \sim B_1 |^{\mathbb{F}} \Delta$.
   
   Now, assume the result holds for $n$ events; we prove it for $n+1$. We have two cases:
   \paragraph{Case 1:} $n_0 = 1$. This case is similar to the base case. The premise and conclusion are identical: $A_1|^{\mathbb{E}} \bigcap_{k > 1} A_k \sim B_1|^{\mathbb{F}} \bigcap_{k > 1} B_k$.
   
   \paragraph{Case 2:} $n_0 > 1$. In this case, we apply the inductive hypothesis to the last $n$ events. By the assumption of the $(n+1)$th step, we know that, for all $m \le n_0$, we have that $A_m |^{\mathbb{E}} \bigcap_{k > m} A_k \sim B_m |^{\mathbb{F}} \bigcap_{k > m} B_k$. The inductive hypothesis thus tells us that (1) $\bigcap_{2 \le j \le n_0} A_j |^{\mathbb{E}} \bigcap_{k > n_0} A_k \sim \bigcap_{2 \le j \le n_0} B_j |^{\mathbb{F}} \bigcap_{k > n_0} B_k$. We also know, by assumption, that (2) $A_1|^{\mathbb{E}} \bigcap_{k > 1} A_k \sim B_1|^{\mathbb{F}} \bigcap_{k > 1} B_k$.
   
   Now, define:
  \begin{alignat*}{3}
      X &= \bigcap_{k > n_0} A_k \qquad &X' &= \bigcap_{k > n_0} B_k \\
      Y &= \bigcap_{j \ge 2} A_j \qquad &Y' &= \bigcap_{j \ge 2} B_j \\
      Z &= \bigcap_{i \ge 1} A_i \qquad &Z' &= \bigcap_{j \ge 1} B_i
  \end{alignat*}
    We know that $Z|^{\mathbb{E}}Y \sim Z'|^{\mathbb{F}}Y'$ by \Ax{4} and assumption (2). Similarly, we know that $Y|^{\mathbb{E}} X \sim Y'\sim^{\mathbb{F}} X'$ by \Ax{4} and assumption (1). Thus, \Ax{6b} entails that $Z|^{\mathbb{E}} X \sim Z'|^{\mathbb{F}} X'$, as desired.
\end{proof}

The other lemma used for the result is just part 2 of \autoref{lm:qSufficientStatisticExamples}.
   

\section{A weaker qualitative, likelihood principle}
In this section, we introduce a weaker version of \f{lp} that has a fairly natural qualitative analog; we call the weaker principle $\f{lp}^-$ and its qualitative analog $\f{qlp}^-$.  We then prove that when the qualitative analog $\f{qlp}^-$ entails that two sets of outcomes $E$ and $F$ are evidentially equivalent, then $E$ and $F$ are $\mathcal{U}$-posterior equivalent, where $\mathcal{U}$ is the set of all orderings satisfying the axioms for qualitative probability.  That theorem -- as well as the machinery we develop in proving it -- turn out to be crucial in deriving our remaining major results.

To motivate our weaker principle, notice that we can multiply numerical probabilities when two events are independent.  So suppose $E$ and $F$ are outcomes of the same experiment and
\begin{enumerate}
    \item $P_{\theta}(E) = P_{\theta}(F \cap C_{\theta})$ for all $\theta$,
    \item For all $\theta \in \Theta$, the events $F$ and $C_{\theta}$ are conditionally independent given $\theta$, and
    \item $P_{\theta}(C_{\theta})=P_{\upsilon}(C_{\upsilon}) > 0$ for all $\theta, \upsilon \in \Theta$.
\end{enumerate}
Roughly, the event $C_{\theta}$ acts as a witness to the equality $P_{\theta}(E) = c \cdot P_{\theta}(F)$.  Specifically, assumptions 1 and 2 encode the equality, and assumption 3 asserts this constant is invariant with respect to the parameter $\theta$.

By \f{lp}, the three conditions entail that $E$ and $F$ are evidentially equivalent. The proof is simple.  Let  $c = P_{\theta_0}(C_{\theta_0})$ for any $\theta_0 \in \Theta$.  Then for all $\theta$:
\begin{eqnarray*}
P_{\theta}(E) &=& P_{\theta}(F \cap C_{\theta}) \mbox{ by Assumption 1} \\
 &=& P_{\theta}(F) \cdot P_{\theta}(C_{\theta}) \mbox{ by Assumption 2} \\
 &=& c \cdot P_{\theta}(F) \mbox{ by Assumption 3} 
\end{eqnarray*} 
Since $P_{\theta}(E) = c \cdot P_{\theta}(F)$ for all $\theta$, then $E$ and $F$ are evidentially equivalent by \f{lp}. 

For an example, suppose we are trying to discern the type of an unmarked urn, with colored balls in frequencies according to \autoref{table:lp_urns}. Let $B$ and $W$ respectively denote the events that one draws a blue and a white ball on the first draw. By \f{lp}, these events are equivalent, as $P_\theta(B) = 1/2 \cdot P_\theta(W)$ for all $\theta \in \{\theta_1, \theta_2\}$, which denote the urn type. We can see this equivalence in another way. Let $C_1/C_2$ respectively denote the events that, after replacing the first draw, one draws a cyan/cobalt ball respectively on the second draw. Then, $P_{\theta_1}(W \cap C_1) = P_{\theta_1}(W) \cdot P_{\theta_1}(C_1) = 1/2  \cdot P_{\theta_1}(W)  = P_{\theta_1}(B)$ by independence of the draws, if the urn is Type 1. Similarly, $P_{\theta_2}(W \cap C_2) = P_{\theta_2}(W) \cdot P_{\theta_2}(C_2) = 1/2 \cdot P_{\theta_2}(W)  = P_{\theta_2}(B)$, if the urn is Type 2. By \f{lp}, because $P_{\theta_1}(C_1) = P_{\theta_2}(C_2)  > 0$ and $P_{\theta_i}(B) =  P_{\theta_i}(W  \cap C_i) = P_{\theta_i}(C_i) \cdot P_{\theta_i}(W)$ for all $i$, we know $B$ and $W$ are evidentially equivalent.

\begin{table}[t]
    \centering
    \begin{tabular}{l | l | l | l | l | l}
        $\ $ & \multicolumn{5}{c}{Number of balls} \\
        \hline
        $\ $ & Blue & White & Cyan & Cobalt & Total  \\
        \hline
        Urn Type 1 &  15 & 30 & 50 & 5 & 100 \\
        Urn Type 2 &  10 & 20 & 20 &  50 & 100 \\
       
    \end{tabular}
    \caption{Drawing a blue ball is evidentially equivalent to drawing a white ball according to \f{LP}. In this example, the equivalence can be ``encoded'' by the cyan and cobalt balls instead of a constant, $c$.}
    \label{table:lp_urns}
\end{table}

So define $\f{lp}^{-}$ to be the thesis that, if conditions 1-3 hold, then $E$ and $F$ are evidentially equivalent.  We just showed that, if $\f{lp}^{-}$ entails $E$ and $F$ are evidentially equivalent, then so does \f{lp}.  By \autoref{clm:bLPPosteriorEquivalence}, it follows that $E$ and $F$ are also Bayesian posterior and favoring equivalent.

Because conditions 1-3 do not contain any arithmetic operations, each has a direct qualitative analog.  Thus, we define $\f{qlp}^-$ to be the thesis that if the qualitative analogs of the above three conditions hold, then $E$ and $F$ are $\mathcal{U}$-posterior equivalent, where $\mathcal{U}$ is the universal set of orderings.  The next theorem asserts that if $\f{qlp}^-$ entails two sets of outcomes are evidentially equivalent, then they outcomes are posterior equivalent.

\qlpMinus*
This theorem follows from the following three claims.

\begin{proposition}\label{prop:qlp0}
    If conditions 1, 2, and 4 of \autoref{thm:qLPEntailsPosteriorEquivalence} hold, then $E \in \nul_{\preceq}$ if and only if $F \in \nul_{\preceq}$.
\end{proposition}

\begin{proposition}\label{prop:qlp1}
    Suppose $E|\theta \sim F\cap C_\theta|\theta$ for all $\theta\in\Theta_s \subseteq \Theta$. If $\Theta_s$ is finite and $E, F \not \in \nul$, then for all $\theta \in \Theta_s$: 
    $$\theta|E \sim F \cap C_\theta \cap \theta|\bigcup_{\eta \in \Theta_s}F \cap C_\eta \cap \eta.$$
\end{proposition}

\begin{proposition}\label{prop:qlp2}
    Suppose that $F \ind_\theta C_{\theta}$ with respect to $\preceq$, that $C_{\theta} | \theta \sim C_{\eta} | \eta$, and that $F \cap \theta \notin \nul$ for all $\theta, \eta \in \Theta_F \subseteq \Theta$.   If $\Theta_F$ is finite, then for all $\theta \in \Theta_F$:
    $$\theta | F \sim F \cap C_{\theta} \cap \theta |  \bigcup_{\eta \in \Theta_F} F \cap C_{\eta} \cap \eta.$$
\end{proposition}

\begin{proof}[Proof of \autoref{thm:qLPEntailsPosteriorEquivalence}]
    The first proposition shows that $\cdot|E$ is well-defined if and only if  $\cdot|F$ is, so the first condition of qualitative posterior equivalence is satisfied. Let $\Theta_s$ be the set of $\theta \notin \nul_\preceq$, and let $\Theta_F$ be the set of all $F \cap \theta \notin \nul_\preceq$ (notice that $\Theta_F \subseteq \Theta_s$). We claim that $\theta|E \sim \theta|F$, for all $\theta \in \Theta$. We go by cases.

    \paragraph{Case 1: $\theta \notin \Theta_s$.} Then, $\theta \in \nul$ by definition. This means that $\theta|E \sim \emptyset|E \sim \emptyset|F \sim \theta|F$.

    \paragraph{Case 2: $\theta \in \Theta_s \setminus \Theta_F$.} Then, $F \cap \theta \in \nul$, so $\theta|F \sim \emptyset$. We claim that $\theta|E \sim \emptyset$ as well. To see why, note that, for all $\eta \in \Theta_s$, we get that $E|\eta \sim F \cap C_\eta|\eta$, from condition 1 of the theorem and \Ax{0}. Thus, we can apply \autoref{prop:qlp1} to $\theta$ and $\Theta_s$, yielding:
    \begin{align}\label{eq:qlp1}
        \theta|E \sim F \cap C_\theta \cap \theta|\bigcup_{\eta \in \Theta_s}F \cap C_\eta \cap \eta
    \end{align}
    Now, since $F \cap \theta \in \nul$, it follows from \autoref{lm:krantz8} that $F \cap C_\theta \cap \theta \in \nul$. So, $\theta|E \sim \emptyset$, completing this case.

    \paragraph{Case 3: $\theta \in \Theta_F$.} This implies that $\theta \in \Theta_s$ as well; thus, \autoref{eq:qlp1} remains true. Now, for all $\theta_F, \eta_F \in \Theta_F$, we get that $F \cap \theta_F \notin \nul$, that $F \ind_{\theta_F} C_{\theta_F}$ with respect to $\preceq$, and that $C_{\theta_F}|\theta_F \sim C_{\eta_F}|\eta_F$. We thus apply \autoref{prop:qlp2} to $\theta$ and $\Theta_F$, yielding:
    \begin{align}\label{eq:qlp2}
        \theta|F \sim F \cap C_{\theta} \cap \theta |  \bigcup_{\eta \in \Theta_F} F \cap C_{\eta} \cap \eta
    \end{align}
    Now, for any $\theta_s \in \Theta_s \setminus \Theta_F$, we know that $F \cap C_{\theta_s} \cap \theta \in \nul$, and thus applying Part 6 of \autoref{lm:krantz9} to \autoref{eq:qlp2}, we get $\theta|F \sim F \cap C_{\theta} \cap \theta |  \bigcup_{\eta \in \Theta_s} F \cap C_{\eta} \cap \eta$. Along with \autoref{eq:qlp1}, this means that $\theta|E \sim \theta|F$.

    Now, one can repeatedly apply \Ax{5} over the constituent simple hypotheses in any $H$ to get $H|E \sim H|F$, showing that $E$ and $F$ are qualitative posterior equivalent. With \autoref{thm:qfavoring_and_posterior_equivalence}, we also get qualitative support equivalence. 
\end{proof}

Before getting to the proofs of the propositions, we first prove the following corollary to \autoref{thm:qLPEntailsPosteriorEquivalence}.
\begin{corollary}\label{cor:IdenticalLikelihoodsEntailsPosteriorEquivalence}
    If $\Theta$ is finite and $E|\theta \equiv F|\theta$ for all $\theta \in \Theta$, then $E$ and $F$ are qualitatively posterior and support equivalent.
\end{corollary}
\begin{proof}
    We simply apply \autoref{thm:qLPEntailsPosteriorEquivalence} with $C_\theta = \Delta$ for all $\theta$. To see that the conditions holds:
    \begin{enumerate}
        \item Note that, for any $\theta$, $E|\theta \equiv F|\theta$ implies that $E|\theta \equiv F\cap \Delta|\theta$, since $F \cap \Delta = F$.
        \item Similarly, $F \ind_\theta \Delta$, since $F|\Delta \cap \theta \equiv F|\theta$, since $\Delta \cap \theta = \theta$.
        \item We have to show $\Delta|\theta \equiv \Delta|\eta$ for all $\theta, \eta \in \Theta$. Note that by \Ax{4}, $\Delta|\theta \sim \Delta \cap \theta|\theta = \theta|\theta$. By \Ax{3}, $\theta|\theta \sim \eta|\eta$. Applying \Ax{4} again shows that $\eta|\eta \sim \Delta|\eta$, so by transitivity, $\Delta|\theta \sim \Delta|\eta$. By \Ax{0}, we get that $\Delta|\theta \equiv \Delta|\eta$, as desired.
        \item $\Delta|\theta \sqsupset \emptyset|\theta$ by \Ax{2}.
    \end{enumerate}
    From these conditions and the fact that $\Theta$ is finite, we get that $E$ and $F$ are qualitatively posterior and support equivalent.
\end{proof}

\subsection{Proof of \autoref{prop:qlp0}}
We prove two short lemmata first.    The first is the qualitative analog of the fact that, if $P(B \cap C)>0$, then $P(A \cap B|C) = 0$ if and only if $P(A|B \cap C) = 0$.

\setcounter{lemma}{27}
\begin{lemma}\label{lm:NullIntersection}
  Suppose $B\cap C \not \in \nul$.  Then $A \cap B|C \sim \emptyset|D$ if and only if $A|B \cap C \sim \emptyset|D$ for all $D \not \in \nul$.
\end{lemma}
\begin{proof} Since $B \cap C \not \in \nul$, we know $C \not \in \nul$ by Part 3 of \autoref{lm:krantz8}.  Thus:
\begin{alignat*}{3}
&A \cap B|C \sim \emptyset|D && \Leftrightarrow A \cap B|C \sim \emptyset|C & \quad \mbox{by Part 2 of \autoref{lm:krantz9} since } C \not \in \nul \\
&~ && \Leftrightarrow A \cap B \cap C|\Delta \sim \emptyset|\Delta & \quad  \mbox{by \autoref{lm:qConditionalIntersection}} \\
&~  && \Leftrightarrow A|B \cap C \sim \emptyset|B \cap C & \quad \mbox{by \autoref{lm:qConditionalIntersection} since } B \cap C \not \in \nul \\
&~  && \Leftrightarrow A|B \cap C \sim \emptyset|D & \quad \mbox{by Part 2 of \autoref{lm:krantz9} since } D \not \in \nul 
\end{alignat*}  
\end{proof}

\begin{lemma}\label{lm:LP-Null}
    Suppose $A \ind_C B$ and $\emptyset|C \prec B|C$.  Then $A \cap B|C \sim \emptyset|C$ if and only if $A|C \sim \emptyset|C$.
\end{lemma}

\begin{proof} 

The right to left direction follows immediately from Part 3 of \autoref{lm:krantz8} since $A \cap B \subseteq A$.

In the left to right direction, first note that $B \cap C|C \sim B|C$ by \Ax{4} and that  $B|C \succ \emptyset|C$ by assumption.  Thus,  $B \cap C|C \succ \emptyset|C$, from which it follows that $B \cap C \not \in \nul$. Thus:
\begin{alignat*}{3}
&A \cap B|C \sim \emptyset|C && \Leftrightarrow A |B \cap C \sim \emptyset|C& \quad \mbox{by \autoref{lm:NullIntersection} since }  B \cap C \not \in \nul  \\
&~ && \Leftrightarrow A|C \sim \emptyset|C & \quad \mbox{since }  A \ind_C B
\end{alignat*}  
\end{proof}


\begin{proof}[Proof of \autoref{prop:qlp0}:]
Let's first show that if  $E$ is a member of $\nul_{\preceq}$, then so is $F$.  The proof of the converse is nearly identical because the lemmata we use are all biconditionals.

Suppose $E \in \nul_{\preceq}$.  Thus, $E \cap \{\theta\}  \in \nul$ for all $\theta$ by Part 3 of \autoref{lm:krantz8}, as $E \cap \{\theta\} \subseteq E$.  So by \Ax{2}, it follows that $E \cap \{\theta\} | \Delta \sim \emptyset |\Delta$.  Thus, for all $\theta$ such that $\{\theta\} \not \in \nul_{\preceq}$, we obtain that $E|\theta \sim \emptyset|\theta$ by \autoref{lm:qConditionalIntersection}.

Fix any arbitrary $\upsilon \in \Theta$ such that $\{\upsilon\} \not \in \nul_{\preceq}$. Below we'll show $F|\upsilon \sim \emptyset|\upsilon$.  This will suffice to show $F \in \nul_{\preceq}$ by the following reasoning.  Since $F|\upsilon \sim \emptyset|\upsilon$, \autoref{lm:qConditionalIntersection}  entails that $F \cap \{\upsilon\}|\Delta \sim \emptyset|\Delta$.  Since $\upsilon$  was arbitrary, it follows that  $F \cap \{\theta\}|\Delta \sim \emptyset|\Delta$ for all $\theta$ such that $\{\theta\} \not \in \nul_{\preceq}$.  Since $\Theta$ is finite by assumption, it follows from \Ax{5} that $F \cap H|\Delta \sim \emptyset|\Delta$ where $H = \{\theta \in \Theta:  \{\theta\} \not \in \nul_{\preceq}\}$.  Because $\Theta$ is finite, it's clear that $\neg H  := \{\theta \in \Theta: \{\theta\} \in \nul_{\preceq } \}$ is itself a member of $\nul_{\preceq}$, and hence, $H \in {\cal S}$ is an almost-sure event.  Thus, by Part 5 of \autoref{lm:krantz9}, we obtain that $F \cap H|\Delta \sim F|\Delta$, and so $F|\Delta \sim \emptyset|\Delta$.  Thus, by \Ax{2},  $F \in \nul_{\preceq}$ as desired.

Now we show that $F|\upsilon \sim \emptyset|\upsilon$ for all $\upsilon \in \Theta$ such that $\{\upsilon\} \not \in \nul_{\preceq}$, as we claimed.  Above, we showed that $E|\upsilon \sim \emptyset|\upsilon$.  From \Ax{0}, it follows that  $E|\upsilon \equiv \emptyset|\upsilon$.  By the first assumption of the proposition, we know $E|\upsilon \equiv F \cap C_{\upsilon}|\upsilon$.  Thus, $F \cap C_{\upsilon}|\upsilon \equiv \emptyset|\upsilon$.  From our assumptions that $F \ind_{\upsilon} C_{\upsilon}$ and the fact $C_{\upsilon}|\upsilon \sqsupset \emptyset|\upsilon$, \autoref{lm:LP-Null} allows us to infer that $F|\upsilon \equiv \emptyset|\upsilon$.  Since $\{\upsilon\} \not \in \nul_{\preceq}$, we can then apply  \Ax{0} to conclude $F|\upsilon \sim \emptyset|\upsilon$, as desired.

The proof of the converse simply reverses the direction of all the inferences above.  That is, if we assume $F \in \nul_{\preceq}$, we can argue just as above that $F|\theta \sim \emptyset|\theta$ for all $\theta$ such that $\{\theta\} \not \in \nul$.  From there, we use \autoref{lm:LP-Null} to argue $F \cap C_{\theta}|\theta \sim \emptyset|\theta$ for all non-null $\theta$.  By the assumptions of the proposition, that entails $E|\theta \sim \emptyset|\theta$ for all non-null $\theta$.   By \autoref{lm:qConditionalIntersection} and additivity, we can argue that $E|\Delta \sim \emptyset|\Delta$, and hence, $E \in \nul_{\preceq}$ as desired.
\end{proof}

\subsection{Proof of \autoref{prop:qlp1}}
The proof of \autoref{prop:qlp1} requires a few lemmata.
\begin{lemma}\label{lm:qBayesLTPHelper}
    Suppose $B_1, \ldots, B_n$ partition $G$. If $A_i|B_i \sim C_i|B_i$ for all $i \le n$, then
    $$\bigcup_{i \le n}A_i \cap B_i | G \sim \bigcup_{i \le n}C_i \cap B_i | G \text{ for all } i \le n$$
\end{lemma}
\begin{proof}
    We first use \autoref{lm:qConditionalIntersection} on the premise, which yields $A_i \cap B_i | \Delta \sim C_i \cap B_i | \Delta$ for all $i \le n$. Applying \Ax{5} $n$ times yields $\bigcup_{i\le n} A_i \cap B_i|\Delta \sim \bigcup_{i\le n} C_i \cap B_i|\Delta$. Because $B_1, \ldots, B_n$ partition $G$, this is equivalent to $G \cap \bigcup_{i\le n} A_i \cap B_i|\Delta \sim G \cap \bigcup_{i\le n} C_i \cap B_i|\Delta$. Finally, we apply \autoref{lm:qConditionalIntersection} again to get the desired result (note that $G \notin \nul$ since $B_i \notin \nul$ by the fact that $A_i|B_i \sim C_i|B_i$).
\end{proof}

The next lemma outlines conditions under which we can ``add'' two qualitative equations by unioning the conditioning events. We will use this in the proof of \autoref{lm:qBayesLTPGeneral}.

\begin{lemma}\label{lm:qPartitionOnRight}
Suppose $Y \cap Y' = Z \cap Z' = \emptyset$ and that $X|Y \sim W|Z$ and $X|Y' \sim W|Z' \sim \emptyset|\Delta $. Then $X|Y \cup Y' \preceq W|Z \cup Z'$ if and only if $Y|Y \cup Y' \preceq Z | Z \cup Z'$.
\end{lemma}
\begin{proof} Note that since $X|Y \sim W|Z$ and $X|Y' \sim W|Z'$, we know $Y,Y', Z, Z' \not \in \nul$ and can be conditioned on freely.  Thus, by Part 3 of \autoref{lm:krantz8}, it follows that $Y \cup Y'$ and $Z \cup Z'$ are non-null, and so we can condition on them freely as well.  Now we proceed.

In the left to right direction, suppose  $X|Y \cup Y' \preceq W|Z \cup Z'$.  Suppose for the sake of contradiction that $Y|Y \cup Y' \not \preceq Z|Z \cup Z'$.  By \Ax{1}, it follows that $Y|Y \cup Y' \succ Z|Z \cup Z'$.  
    
    Define:
        \begin{alignat*}{3}
        &A = Y \cup Y', && \quad \quad \quad && A' = Z \cup Z' \\ 
        &B = Y && \quad \quad \quad && B' = Z \\
        &C = X \cap Y, && \quad \quad \quad  && C' = W \cap Z
    \end{alignat*}
    By \Ax{4}, $C|B=X \cap Y | Y \sim X|Y$ and $C'|B' =  W \cap Z|Z \sim W|Z$, and by assumption, $X|Y \sim W|Z$.  Thus, $C|B \sim C'|B'$.  Further, by assumption of the reductio, $B|A = Y|Y \cup Y' \succ W|Z \cup Z' = B'|A'$.  So by \Ax{6b}, it follows that $C|A \succ C'|A'$, i.e., that $X \cap Y |Y \cup Y' \succ W \cap Z|Z \cup Z'$.  
    Further, because $Z \cap Z' = \emptyset$ and $W|Z' \sim \emptyset|\Delta$, it follows from \autoref{lm:qConditionalNull} that $W \cap Z' | Z \cup Z' \sim \emptyset | \Delta$.  Hence, $X \cap Y' | Y \cup Y' \succeq W \cap Z' |Z \cup Z'$ by \autoref{lm:krantz6-1}.   So we've shown that
    \begin{align*}
        X \cap Y |Y \cup Y' &\succ W \cap Z | Z \cup Z'\mbox{, and} \\
        X \cap Y' | Y \cup Y' &\succeq W \cap Z'|Z \cup Z'.
    \end{align*}
    Because $Y \cap Y' = Z \cap Z' =  \emptyset$, applying \Ax{5} to these equations yields $X \cap (Y \cup Y')|Y \cup Y' \succ W \cap (Z \cup Z')|Z \cup Z'$.  Applying \Ax{4} to both sides of the equation yields $X|Y \cup Y' \succ W|Z \cup Z'$, which contradicts our assumption.

    In the right to left direction, suppose $Y|Y \cup Y' \preceq Z | Z \cup Z'$. We must show  $X|Y \cup Y' \preceq W|Z \cup Z'$. 
    We first show that 
    
    \begin{equation}\label{eqn:qPartitionOnRight1}
        X \cap Y |Y \cup Y' \preceq W \cap Z|Z \cup Z'
    \end{equation}
    That argument will use the assumption that $X|Y \sim W|Z$.  We then prove 
    \begin{equation}\label{eqn:qPartitionOnRight2}
       X \cap Y' |Y \cup Y' \sim W \cap Z'|Z \cup Z'
    \end{equation}
    That argument will use the assumption that $X|Y' \sim W|Z' \sim \emptyset|\Delta$. Because $Y \cap Y' = Z \cap  Z' = \emptyset$ (and hence $(X \cap Y) \cap (X \cap Y') = (W \cap Z) \cap (W \cap Z') = \emptyset$ ),  applying \Ax{5} to \autoref{eqn:qPartitionOnRight1} and \autoref{eqn:qPartitionOnRight2} yields 
    $$(X \cap Y) \cup (X \cap Y')|Y \cup Y' \preceq (W \cap Z) \cup (W \cap Z')|Z \cup Z'.$$ 
    Applying \Ax{4} to both sides of that equation yields the desired conclusion that $X|Y \cup Y' \preceq W|Z \cup Z'$.

    To prove \autoref{eqn:qPartitionOnRight1}, define:
    \begin{alignat*}{3}
            &A = Y \cup Y', && \quad \quad \quad && A' = Z \cup Z' \\ 
            &B = Y && \quad \quad \quad && B' = Z \\
            &C = X \cap Y, && \quad \quad \quad  && C' = W \cap Z
    \end{alignat*}
    By \Ax{4}, $C|B = X \cap Y|Y \sim X|Y$ and $C'|B' = W \cap Z|Z \sim W|Z$.  Since $X|Y \sim W|Z$ by assumption, we have $C|B \sim C'|B'$.  And our assumption that $Y|Y \cup Y' \preceq Z | Z \cup Z'$ means that $B|A \preceq B'|A'$.  From \Ax{6b}, it follows that $C|A \preceq C'|A'$, i.e., that $X \cap Y|Y \cup Y' \preceq W \cap Z | Z \cup Z'$. 
    
    To prove \autoref{eqn:qPartitionOnRight2}, notice that, since $X|Y' \sim \emptyset|\Delta$ by assumption,  \autoref{lm:qConditionalNull} immediately entails  $X \cap Y' |Y \cup Y' \sim \emptyset|\Delta$.  By the same reasoning, because $W|Z' \sim \emptyset|\Delta$ by assumption, it follows that $W\cap Z'|Z \cup Z' \sim \emptyset|\Delta$.  By transitivity (specifically \Ax{1}), $X \cap Y' |Y \cup Y' \sim W\cap Z'|Z \cup Z'$ as desired.
\end{proof}


\subsection{Proof of \autoref{prop:qlp2}}

\begin{proof}[Proof of \autoref{prop:qlp2}]
    Let $\Theta_F = \{\theta_1, \ldots, \theta_n\}$ and let $B_i = \{\theta_i\}$. Thus, $G = \Theta_F$. Let $A = F$ and $C_i = C_{\theta_i}$. The conditions of \autoref{lm:qLPEntailsPosteriorEquivalence2} are satisfied by the assumptions of the proposition, and thus $\bigcup_{\eta \in \Theta_F}  C_{\eta} \cap \eta|F \sim C_{\theta}|\theta $ for all $\theta \in \Theta_F$. Since $F \ind_\theta C_{\theta}$, it follows from
    \autoref{lm:QConditionalIndependenceIsSymmetric} that $C_{\theta} \ind_\theta F$ and so $C_{\theta}|\theta \sim C_{\theta}|F \cap \theta$.  Thus, we've shown: 
    
    \begin{equation}\label{eq:qLPLm1Prem}
        \bigcup_{\eta \in \Theta_F}  C_{\eta} \cap \eta|F \sim C_{\theta}|F \cap \theta    
    \end{equation}
    Now, we apply \autoref{lm:Aditya} with:
    \begin{alignat*}{3}
        & A = F     \\
        & B = \bigcup_{\eta \in \Theta_F} F \cap C_{\eta} \cap \eta               & \quad \quad   & B' = F \cap \theta  \\
        & C = F \cap C_{\theta} \cap \theta
    \end{alignat*}
    \autoref{eq:qLPLm1Prem} says $B|A \sim C|B'$. So by  \autoref{lm:Aditya}, we get $B'|A \sim C|B$, which is the desired conclusion.
\end{proof}

\section{Conditional Independence and the Semi-Graphoid Axioms}
\label{sec:GraphoidAxioms}


Recall, we say $A$ is qualitatively conditionally independent of $B$ given $C \not \in \nul$ when $A|B\cap C \sim A|C$ or $B \cap C \in \nul$. We write this as $A \ind_C B$.    In this section, we show the relation $\ind$ satisfies the semi-graphoid axioms \citep{pearl_graph-based_1987, dawid_conditional_1979}, which suggests that $\ind$ indeed possesses the same properties as the standard notion of conditional independence from the quantitative setting.

A \textbf{random variable} for an experiment $\mathbb{E}$ is a function $X:\Omega_{\mathbb{E}} \rightarrow {\cal R}$.\footnote{In standard probability theory, a random variable is required to be \emph{measurable}.  Because we have assumed $\Delta$ is finite and that the algebra of events is just the standard power set, we ignore measure-theoretic complications.} In most applications, ${\cal R} = \mathbb{R}^n$ for some $n$.  As is standard, for all $x \in {\cal R}$, we define $\{X = x\} := \{\omega \in \Omega_{\mathbb{E}}: X(\omega) = x \}$ to be the set of experimental outcomes $\omega$ that take the value $x$ under the random variable $X$.  In probability theory, one typically drops the set brackets and writes $P(X=x)$ instead of $P(\{X=x\})$.  We do so the same and write, for example, $X=x|\theta$ instead of $\{X=x\}|\theta$.  

Given random variables $X, Y, Z$, we say that $X \ind_Z Y$  if   $X = x \ind_{Z = z} Y = y$ for all $x \in \ran(X), y \in \ran(Y)$, and $z \in \ran(Z)$, where $\ran(X) = \{X(\omega): \omega \in \Omega \}$ is the range of $X$. 

The (semi-)graphoid aximos for qualitative conditional probability are:
\begin{enumerate}
    \item Symmetry: $X \ind_Z Y \iff Y \ind_Z X$.
    \item Decomposition: $X \ind_Z Y \cap W \implies X \ind_Z Y$.
    \item Weak union: $X \ind_Z Y \cap W \implies X \ind_{Z \cap W} Y$.
    \item Contraction: $X \ind_Z Y \mbox{ and } X \ind_{Z \cap Y} W \implies X \ind_Z Y \cap W$.
\end{enumerate}

We also list some basic properties of conditional independence, which will help us below. 
\begin{lemma}\label{lm:CIUnion}
    If $E_i \ind_G F$ for $i = 1, \dots, n$, where the $E_i$'s are pairwise disjoint, then $\bigcup_i E_i \ind_G F$.
\end{lemma}
\begin{proof}
    If $F \cap G \in \nul$, then we're done. So assume otherwise. From the premise we get that $E_i|F \cap G \sim E_i|G$ for all $i$. Now, we can repeatedly apply \Ax{5}, since the $E_i$ are pairwise disjoint. This results in $\bigcup_i E_i|F \cap G \sim \bigcup_i E_i|G$, as desired.
\end{proof}

We first prove our definition of conditional independence satisfies the Symmetry axiom.

\begin{lemma}\label{lm:GraphoidSymmetry}
$X \ind_Z Y \iff Y \ind_Z X$ for all random variables $X,Y,$ and $Z$.
\end{lemma}
\begin{proof}
    Follows from \autoref{lm:QConditionalIndependenceIsSymmetric}, as symmetry of random variables follows from symmetry of events.
\end{proof}

Next, we prove $\ind$ satisfies the Decomposition axiom.

\begin{lemma}
    For random variables $X$, $Y$, $W$, and $Z$, suppose that $X \ind_Z Y, W$. Then, we get that $X \ind_Z Y$. 
\end{lemma}
\begin{proof}
    We go by induction on $n = |\ran(W)|$, letting $X, Y, Z$ be arbitrary. For notation, index the support of $W$ and let $W_i$ represent the event $W = w_i$. For our base cases, notice that if $n = 1$, the result is trivial, since $W_1 = \Delta$. 

    So, for the inductive hypothesis assume that the result holds for $n \ge 1$. Assume that $X \ind_Z Y, W$, where $W$ partitions the space into $W_1, \dots, W_{n+1}$. Define $W'$ to agree with $W$ everywhere except when $W = w_{n+1}$; in that case, $W' = w_{n}$. So $W'$ partitions the space into $W_1, \dots, W_n \cup W_{n+1}$. Notice that $W'$ has support of size $n$. Further, since we know that $X_i \ind_{Z_j} Y_k \cap W_n$, and $X_i \ind_{Z_j} Y_k \cap W_{n+1}$, \autoref{lm:CIUnion} shows us that $X_i \ind_{Z_j} (Y_k \cap W_n) \cup (Y_k \cap W_{n+1})$ (and conditional independence holds for the other $W_l$ by assumption). So, we get that $X \ind_Z Y, W'$. Since $|\ran(W')| = n$, we can apply the inductive hypothesis, which gives us that $X \ind_Z Y$, as desired.
\end{proof}

Next up, we verify Weak Union.
\begin{lemma}\label{lm:GraphoidWeakUnion}
    For random variables $X, Y, W, Z$, suppose that $X \ind_Z Y, W$. Then, $X \ind_{Z, W} Y$.
\end{lemma}
\begin{proof}
    First, we use decomposition to get that $X \ind_Z W$. Now, notice that for any $Z_i \cap Y_j \cap W_k \in \nul$, the conclusion is automatically satisfied. So, consider the non-null cases. From the premise, we have that $X_i |Y_j, W_k, Z_l \sim X_i | Z_l$, and from decomposition, we get that $X_i|W_k, Z_l \sim X_i|Z_l$. So, by transitivity we get that $X_i|Y_j, W_k, Z_l \sim X_i|W_k, Z_l$. Since this holds for all $i, j, k, l$, we get that $X \ind_{Z, W} Y$, as desired.
\end{proof}


Next, we verify the Contraction axiom. The proof is essentially the previous proof in reverse.
\begin{lemma}\label{lm:GraphoidContraction}
    For random variables $X, Y, W, Z$, suppose that $X \ind_Z Y$ and $X \ind_{Z, Y} W$. Then, $X \ind_Z Y, W$.
\end{lemma}
\begin{proof}
    For $Z_i \cap Y_j \cap W_k \in \nul$, the conclusion is trivially true. For $Z_i \cap Y_j \cap W_k \notin \nul$, we know that $Z_i \cap Y_j \notin \nul$ as well. Thus, we get that $X_l|Z_i, Y_j \sim X_l|Z_i$ and $X_l|Z_i, Y_j, W_k \sim X_l|Z_i, Y_j$. So by transitivity, we get that $X_l|Z_i \sim X_l|Z_i, Y_j, W_k$, i.e. $X \ind_Z Y, W$ as desired.
\end{proof}





\bibliography{References}{}
\bibliographystyle{plainnat}

\end{document}